%% file: OTunderLCA.tex
\documentclass[11pt,a4paper]{article}

\usepackage{natbib}
\usepackage{apalike}
\usepackage{comment}
\usepackage{lipsum}
\usepackage{microtype}

\usepackage[font=small,labelfont=bf]{caption}
\usepackage
[
        a4paper,
        left=3cm,
        right=3cm,
        top=3cm,
        bottom=3cm,
]
{geometry}

\usepackage{setspace}
\usepackage{graphicx}
\graphicspath{{./figures}}

\usepackage{multicol,latexsym, array}
\usepackage{textcomp}
\usepackage[ansinew]{inputenc}
\usepackage{amsmath,amssymb, amsthm}
\usepackage[OT1]{fontenc}
\usepackage[english]{babel}
\usepackage{hyperref}

\hypersetup{
}

\usepackage{enumitem}
\usepackage{MnSymbol} 
\usepackage{empheq}
\usepackage{rotating}
\usepackage{titletoc}
\usepackage{framed}
\usepackage{multirow,bigdelim}
\usepackage{subfigure}
\usepackage{bbm}
\usepackage[noend]{algpseudocode}
\usepackage{algorithm}
\usepackage{pgf}
\usepackage{pgfplots}
\usepackage{tikz}
\usepackage{shadethm}
\usepackage{thmtools}
\usepackage{mdframed}
\usepackage{lipsum}
\usepackage{fixmath}
\usepackage{multicol}
\usepackage{footmisc}

\usepackage{caption}
\usepackage{xargs}  
\usepackage{xcolor} 

\definecolor{DarkGray}{RGB}{90,90,90}

\usepackage[colorinlistoftodos,prependcaption,textsize=tiny]{todonotes}
\usepackage{cleveref}

\date{\today}

\newenvironment{ackno}%
    {
    \paragraph{Acknowledgements:} 
    }

\usepackage{caption}
\usepackage{mathrsfs}
\usetikzlibrary{arrows}

\usepackage{xpatch}

\makeatletter
\xpatchcmd{\endmdframed}
  {\aftergroup\endmdf@trivlist\color@endgroup}
  {\endmdf@trivlist\color@endgroup\@doendpe}
  {}{}
\makeatother

\usepackage{nicefrac}
\usepackage{dsfont}

\usepackage{bigints}
\usepackage{mathtools}
\usepackage{fixltx2e}

\usepackage{mathtools}
\usepackage{color}
\usepackage{fancyhdr}
\usepackage{lastpage}

\usepackage{ifthen}

\newcommand{\OT}{T}
\newcommand{\OTn}{\hat\OT_{c,n}}

\newcommand{\RR}{\mathbb{R}}

\newcommand{\NN}{\mathbb{N}}

\newcommand{\XC}{\mathcal{X}}

\newcommand{\YC}{\mathcal{Y}}

\newcommand{\CC}{\mathcal{C}}
\newcommand{\DC}{\mathcal{D}}

\newcommand{\ZC}{\mathcal{Z}}

\newcommand{\UC}{\mathcal{U}}
\newcommand{\VC}{\mathcal{V}}
\newcommand{\FC}{\mathcal{F}}
\newcommand{\GC}{\mathcal{G}}

\newcommand{\NC}{\mathcal{N}}
\newcommand{\RC}{\mathcal{R}}
\newcommand{\PC}{\mathcal{P}}
\newcommand{\pf}{\mathfrak{p}}

\newcommand{\FCot}{\mathcal{F}_\mathrm{c}}

\newcommand{\sphere}{\mathbb{S}}

\newcommand{\norm}[1]{\left\lVert#1\right\rVert}

\newcommand{\uniform}{\mathrm{Unif}}

\newcommand{\id}{\mathrm{id}}
\newcommand{\dif}{\mathrm{d}}

\newcommand{\supp}{\mathrm{supp}}

\newcommand{\EV}[1]{\mathbb{E}\left[#1 \right]}

\newcommand{\Var}{\mathrm{Var}}

\newcommand{\powerC}{^c}

\newcommand{\coloneqq}{:=}
\newcommand{\eqqcolon}{=:}

\renewcommand{\phi}{\varphi}

\newcommand{\diam}[1]{\mathrm{diam}(#1)}

\renewcommand*{\epsilon}{\varepsilon}

\newcommand{\BL}[1]{\mathrm{BL}_{1}(#1)}
\newcommand{\dimM}[1]{\overline\dim_M(#1)}

\theoremstyle{plain}
\newtheorem{theorem}{Theorem}[section]

\newtheorem{lemma}[theorem]{Lemma}
\newtheorem{proposition}[theorem]{Proposition}
\newtheorem*{theorem*}{Theorem}

\theoremstyle{definition}

\newtheorem{remark}[theorem]{Remark}
\newtheorem{example}[theorem]{Example}
\newtheorem*{example*}{Example}

\makeatletter
\def\mysequence#1{\expandafter\@mysequence\csname c@#1\endcsname}
\def\@mysequence#1{%
  \ifcase#1\or (Lip)\or (SC)\or (Hol)\or (Supercalifragilisticexpialidocious)\else\@ctrerr\fi}
\makeatother

\newtheorem{assumption}{Assumption}

\makeatletter
\newcommand{\mylabel}[2]{#2\def\@currentlabel{#2}\label{#1}}
\makeatother

\numberwithin{equation}{section}

\newcommand{\footremember}[2]{
	\footnote{#2}
	\newcounter{#1}
	\setcounter{#1}{\value{footnote}}
}

\newcommand{\footrecall}[1]{
	\footnotemark[\value{#1}]
}

\begin{document}

\author{Shayan Hundrieser
  \hspace{-0.6em}\footremember{equalcontr}{These authors contributed equally}%
  \hspace{-0.65em}\footremember{ims}{\scriptsize
    Institute for Mathematical
		Stochastics, University of G\"ottingen,
		Goldschmidtstra{\ss}e 7, 37077 G\"ottingen}%
  \hspace{-0.65em}\footremember{mbexc}{\scriptsize
    Cluster of Excellence "Multiscale Bioimaging: from Molecular Machines to Networks of Excitable Cells" (MBExC),
    University Medical Center,
    Robert-Koch-Stra{\ss}e 40, 37075 Göttingen}
	\\
  \footnotesize{\href{mailto:s.hundrieser@math.uni-goettingen.de}{s.hundrieser@math.uni-goettingen.de}}
  \\[2ex]
	Thomas Staudt
  \hspace{-0.6em}\footrecall{equalcontr}%
  \hspace{-0.3em}\footrecall{ims}%
  \hspace{-0.3em}\footrecall{mbexc}
	\\
  \footnotesize{\href{mailto:thomas.staudt@uni-goettingen.de}{thomas.staudt@uni-goettingen.de}}
  \\[2ex]
	Axel Munk
  \hspace{-0.6em}\footrecall{ims}%
  \hspace{-0.3em}\footrecall{mbexc}%
  \hspace{-0.0em}\footnote{\scriptsize
    Max Planck Institute for Biophysical Chemistry,
    Am Fa{\ss}berg 11, 37077 G\"ottingen}
	\\
  \footnotesize{\href{mailto:munk@math.uni-goettingen.de}{munk@math.uni-goettingen.de}}
}

\title{Empirical Optimal Transport between Different Measures Adapts to Lower Complexity}
 
\pagenumbering{arabic}

\maketitle

\begin{abstract}
\noindent The empirical optimal transport (OT) cost between two probability measures from random data is a fundamental quantity in transport based data analysis. In this work, we derive novel guarantees for its convergence rate when the involved measures are \emph{different}, possibly supported on different spaces. Our central observation is that the statistical performance of the empirical OT cost is determined by the \emph{less} complex measure, a phenomenon we refer to as \emph{lower complexity adaptation} of empirical OT.  For instance, under Lipschitz ground costs, we find that the empirical OT cost based on $n$ observations converges at least with rate $n^{-1/d}$ to the population quantity if one of the two measures is concentrated on a $d$-dimensional manifold, while the other can be arbitrary. For semi-concave ground costs, we show that the upper bound for the rate improves to $n^{-2/d}$. Similarly, our theory establishes the general convergence rate $n^{-1/2}$ for semi-discrete OT.
All of these results are valid in the two-sample case as well, meaning that the convergence rate is still governed by the simpler of the two measures. On a conceptual level, our findings therefore suggest that the curse of dimensionality only affects the estimation of the OT cost when \emph{both} measures exhibit a high intrinsic dimension. Our proofs are based on the dual formulation of OT as a maximization over a suitable function class $\FC_c$ and the observation that the $c$-transform of $\FC_c$ under bounded costs has the same uniform metric entropy as $\FC_c$ itself.
\end{abstract}
\vspace{0.5cm}
\vspace{0.5cm}


\newpage \section{Introduction}
The theory of optimal transport (OT) allows for an effective comparison of probability measures that is faithful to the geometry of the underlying ground space (see \citealt{rachev1998massTheory, rachev1998massApplications, vil03,villani2008optimal,santambrogio2015optimal} for comprehensive treatments). Origins of OT date back to the seminal work by \cite{monge} and its measure theoretic generalization by \cite{kant1942_original, kant58}, paving the way for a rich theory and many applications. With recent computational advances (for a survey see \citealt{bertsimas1997introduction, cuturi18}) OT based methodology is also quickly emerging as a useful tool for data analysis with diverse applications in statistics. This includes bootstrap and resampling \citep{bickel1981, sommerfeld19FastProb, Heinemann2020}, goodness of fit testing \citep{del1999tests, hallin2021multivariate}, multivariate quantiles and ranks \citep{chernozhukov2017monge, deb2021multivariate, hallin2021distribution} and general notions of dependency \citep{nies2021transport, deb2021rates, mordant2022measuring}. For a recent survey see  \cite{panaretos2019statistical}. Further areas of application include machine learning \citep{arjovsky2017wasserstein, altschuler2017near, dvurechensky2018computational}, and computational biology \citep{evans2012phylogenetic,Schiebinger19,tameling2021Colocalization, wang2020optimal}, among others.
 
Intuitively, OT aims to transform one probability measure into another one in the most cost-efficient way. For a general formulation, let $\mu \in \PC(\XC)$ and $\nu\in \PC(\YC)$ be probability measures on Polish spaces $\XC$ and $\YC$, and consider a measurable cost function $c \,\colon \XC\times \YC\rightarrow \RR$. The \emph{optimal transport cost} between $\mu$ and $\nu$ is defined as
\begin{equation}
  \OT_c(\mu, \nu) \coloneqq \inf_{\pi\in \Pi(\mu,\nu)} \int_{\XC\times \YC} c(x,y) \,\dif\pi(x,y),
  \label{eq:PrimalOTProblem}
\end{equation}
where $\Pi(\mu, \nu)$ represents the set of all couplings between $\mu$ and $\nu$, i.e., the probability measures on $\XC\times\YC$ with marginal distributions $\mu$ and $\nu$.
In statistical problems, the measure $\mu$ is typically unknown and only i.i.d.\ observations $X_1, \dots, X_n\sim \mu$, defining the empirical measure $\hat\mu_n \coloneqq \frac{1}{n}\sum_{i = 1}^{n} \delta_{X_i}$, are available. A standard approach to estimate $\OT(\mu, \nu)$ in this setting is by means of the \emph{empirical optimal transport cost} $\OT_c(\hat \mu_n, \nu)$, whose convergence to the population value for increasing $n$ has been the subject of numerous works. Most research in this context, of which we can only give a selective overview, is devoted to the analysis of the \emph{Wasserstein distance} (cf. \citealt{mallows1972note, shorack1986empirical, villani2008optimal}) where $\XC = \YC$ and the cost $c$ in \eqref{eq:PrimalOTProblem} corresponds to the $p$-th power of a metric $d$ on $\XC$. More specifically, for $\mu, \nu \in \PC(\XC)$, the $p$-Wasserstein distance for $p \ge 1$ is defined by
\begin{equation*}
	W_p(\mu, \nu) \coloneqq \big(\OT_{d^p}(\mu, \nu)\big)^{1/p},
\end{equation*}
which is a metric on the space of probability measures on $(\XC,d)$ with finite $p$-th moment.
 
A first fundamental contribution for the analysis of the empirical Wasserstein distance $W_p(\hat \mu_n, \mu)$ in case of $p = 1$  was made by \cite{dudley1969}  via metric entropy bounds, asserting\footnote{Throughout this work, we write  $a_n \lesssim b_n$ for two non-negative real-valued sequences $(a_n)_{n \in \NN}$ and $(b_n)_{n\in \NN}$ if there exists a constant $C>0$ such that $a_n \leq C b_n$ for all $n \in\NN$. If $a_n \lesssim b_n \lesssim a_n$, we write $a_n \asymp b_n$. 
}  $\mathbb{E}\left[W_1(\hat{\mu}_n,\mu)\right]\lesssim n^{-1/d}$ for compactly supported probability measures $\mu$ on $\RR^d$ with $d \geq 3$. In particular, if $\mu$ is absolutely continuous with respect to the Lebesgue measure, this upper bound is tight. Under similar conditions, \cite{dobric1995asymptotics} derived almost sure limits of $n^{1/d}W_1(\hat{\mu}_n,\hat\mu_n')$ through explicit matching arguments for two independent empirical measures $\hat \mu_n$ and $\hat \mu_n'$ of a common distribution $\mu$. 
Extensions to $p > 1$ in Polish metric spaces were obtained by \cite{boissard2014} relying on covering arguments of the underlying ground space. 
For probability measures on Euclidean spaces with possibly unbounded support, \cite{dereich2013constructive} and \cite{fournier2015rate} derived upper bounds on the $p$-th moment $\EV{W_p^p(\hat \mu_n, \mu)}$  under certain moment assumptions by explicitly constructing a couplings between $\hat\mu_n$ and $\mu$. For a compactly supported probability measure $\mu$ on $\RR^d$, their main result implies for $n \geq 1$ that
 \begin{equation}\label{eq:FournierBound}
\EV{W_p(\hat \mu_n, \mu)} \leq \EV{W^p_p(\hat \mu_n, \mu)}^{1/p} \lesssim r_{p,d}(n)\coloneqq
  \begin{cases}
  	n^{-1/2p} & \text{ if } d< 2p,\\
  	n^{-1/2p}\log(n)^{1/p} & \text{ if } d= 2p,\\
  	n^{-1/d} & \text{ if } d> 2p.\\
  \end{cases}
\end{equation}
This bound is known to be tight in several settings, e.g., for $d < 2p$ when $\mu$ is discretely supported and for $d > 2p$ when $\mu = \uniform[0, 1]^d$ is the uniform distribution on the unit cube. 
For $d= 2p$, the differences between $\hat \mu_n$ and $\mu$ at multiple scales culminate in the proof of the upper bound to an additional logarithmic factor, however, it remains open whether it is of correct order. For instance, contributions by \cite{ajtai1984optimal} and \cite{talagrand1994matching} show for $d = 2$ and $p \ge 1$ that $\EV{W_p(\hat \mu_n,\mu)} \asymp n^{-1/2}\log(n)^{1/2}$ if $\mu = \uniform[0,1]^2$ (see also \citealt{bobkov2019simple} for an alternative proof) which improves \eqref{eq:FournierBound} for $p =1$ by an additional $\log(n)^{1/2}$ factor. Notably, the bounds in \eqref{eq:FournierBound} are known to delimit the accuracy of \emph{any} estimator $\tilde \mu_n$ of $\mu$ with respect to the Wasserstein distance in the high-dimensional regime. More precisely, without additional assumptions on $\mu$, the rates in \eqref{eq:FournierBound} are (up to logarithmic factors) minimax optimal \citep{singh2018minimax}, which demonstrates that the estimation of measures in the Wasserstein distance severely suffers from the \emph{curse of dimensionality}.

To overcome this issue, there has been increased interest in structural properties of $\mu$ that allow for improved convergence rates. 
For probability measures on a compact Polish space, \cite{weed2019sharp} derived tight bounds in terms of a notion of intrinsic dimension of $\mu$ (the upper and lower Wasserstein dimension). In particular, if $\mu$ is compactly supported on $\RR^d$ with upper Wasserstein dimension $s > 2p$, they established that $\EV{W_p(\hat \mu_n, \mu)} \lesssim n^{-1/s}$. 
Moreover, for uniformly distributed $\mu$ on a compact connected Riemannian manifold of dimension $d \ge 3$, \cite{ledoux2019OptimalMatchingI} derived the bound $\EV{W_p(\hat \mu_n, \mu)}\asymp n^{-1/d}$, effectively improving upon \eqref{eq:FournierBound} if $3 \leq d\leq 2p$.
Faster convergence rates can also be obtained under smoothness assumptions on Lebesgue absolutely continuous measures by taking suitable wavelet or kernel density estimators \citep{weedBerthet19, deb2021rates, Manole2021_Plugin}, which exploit the smoothness explicitly in contrast to the vanilla empirical OT cost. 
Under a high degree of smoothness, they approach the population measure in Wasserstein distance nearly with the parametric rate $n^{-1/2}$ (instead of $n^{-1/d}$), but come with additional computational challenges \citep{Vacher21}.

So far, we only discussed the situation when $\hat \mu_n$ is compared to $\mu$. From a statistical perspective, however, it is of similar interest to investigate $W_p(\hat \mu_n, \nu)$ for a different measure $\nu$. We refer to \cite{Munk98} and \cite{sommerfeld2018} for various applications, such as testing for relevant differences and confidence intervals for $W_p$. One way to transfer the rates from $W_p(\hat \mu_n, \mu)$ to $W_p(\hat \mu_n,\nu)$ is by means of the triangle inequality,
\begin{subequations}\label{eq:triangleIneq}
\begin{equation}\label{eq:triangleIneqOneSample}
	\left| W_p(\hat \mu_n, \nu) - W_p(\mu, \nu) \right|\leq W_p(\hat \mu_n, \mu).
\end{equation}
Hence, all of the previous bounds on $W_p(\hat \mu_n, \mu)$ immediately imply the same upper bounds for the convergence rate of $W_p(\hat \mu_n, \nu)$ towards $W_p(\mu, \nu)$ when $\mu$ and $\nu$ are \emph{distinct} measures on a common metric space. In the two-sample case, when $\nu$ is additionally estimated by $\hat \nu_n\coloneqq \frac{1}{n} \sum_{i =1}^{n} \delta_{Y_i}$ based on an i.i.d.\ sample $Y_1, \dots, Y_n\sim \nu$, the triangle inequality yields 
\begin{equation}\label{eq:triangleIneqTwoSample}
  |W_p(\hat \mu_n, \hat\nu_n) - W_p(\mu, \nu)| \leq W_p(\hat \mu_n, \mu) + W_p(\hat \nu_n, \nu),
\end{equation}
\end{subequations}
which implies the same upper bounds as in \eqref{eq:FournierBound} (as well as all improvements described above) for compactly supported $\mu, \nu$ on $\RR^d$. Therefore, with $r_{p,d}(n)$ as in \eqref{eq:FournierBound},
\begin{equation}
  \EV{\big| W_p(\hat \mu_n, \hat\nu_n)-W_p(\mu, \nu)\big|}  \lesssim r_{p,d}(n).
  \label{eq:FournierBoundDifferent}
\end{equation}
These upper bounds match the minimax rates (up to logarithmic factors) among all estimators of $W_p(\mu, \nu)$ when no additional assumptions are placed on the measures \citep{liang2019minimax, niles2019estimation}. 
 In particular, this suggests that estimation of the Wasserstein distance between (potentially different) two measures is (without additional assumptions) statistically as difficult as estimation of the underlying measure with respect to Wasserstein loss. 

However, crucial to the minimax optimality of \eqref{eq:FournierBoundDifferent} is the fact that $\mu$ and $\nu$ can be chosen to be arbitrarily close. In fact, in case $\mu \neq \nu$ are sufficiently separated, faster convergence rates may occur.
Indeed, for compactly supported $\mu, \nu$ on $\RR^d$, \cite{chizat2020} employed the dual formulation of the squared $2$-Wasserstein distance (with a similar strategy as for the $1$-Wasserstein distance by \citealt{sriperumbudur2012empirical}) to derive the bound
\begin{subequations}\label{eq:ChizatBound}
\begin{equation}\label{eq:ChizatBoundSquared}
  \EV{\big| W_2^2(\hat \mu_n, \hat\nu_n)-W_2^2(\mu, \nu)\big|} \lesssim r_{2,d}^2(n).
\end{equation}
If $\mu\neq \nu$ with $W_2(\mu, \nu)\geq\delta>0$, this implies squared convergence rates 
\begin{equation}\label{eq:ChizatBoundDifferent}
  \EV{\big| W_2(\hat \mu_n, \hat\nu_n)-W_2(\mu, \nu)\big|} \lesssim r_{2,d}^2(n)/\delta
\end{equation}
\end{subequations}
when compared to \eqref{eq:FournierBoundDifferent}.
For $d \geq 5$, these upper bounds were recently generalized by \cite{Manole21} to arbitrary $p\geq 1$, asserting the convergence rate $n^{-\min(p,2)/d}$ for the empirical $p$-Wasserstein distance. They also provided analogous bounds under convex H\"older smooth costs and proved their sharpness for certain instances as well as  minimax rate optimality up to logarithmic factors. 

Inspired by these developments, this work is dedicated to a comprehensive understanding of the statistical performance of the empirical OT cost when the underlying probability measures are not only different but may additionally be supported on distinct spaces, for example if $\XC$ and $\YC$ are submanifolds of $\RR^d$ with (possibly) different dimension.
This setting is practically relevant, since the concentration of observations from a high-dimensional ambient space on a low dimensional subspace is a commonly encountered phenomenon, reflected by the popularity of nonlinear dimensionality reduction techniques like manifold learning (see, e.g., \citealt{talwalkar2008large,zhu2018image}). Based on the upper bound in~\eqref{eq:triangleIneq},
 one is inclined to believe that the convergence rate is determined by the slower rate, i.e., by the measure with \emph{higher} intrinsic dimension. However, the pivotal (and maybe unexpected) finding of this work is that the convergence rate is actually determined by the measure with \emph{lower} intrinsic dimension.
In this sense, empirical OT naturally adapts to measures with distinct complexity in the most favorable way, and estimating the population value is statistically no harder than estimating the \emph{simpler} one of the measures $\mu$ and $\nu$. We refer to this phenomenon of OT as \emph{lower complexity adaptation} (LCA). 

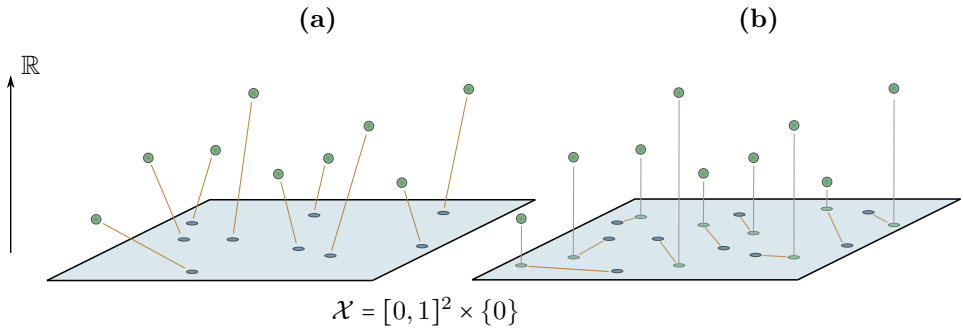
\begin{figure}[t]
  \small\centering
  \input{figures/tex_cartesian.pdf_tex}
  \caption{Optimal transport between two- and three-dimensional point clouds
    (blue and green)  for squared Euclidean costs $\norm{x-y}^2$. The optimal
    assignment in \textbf{(a)} is characterized by first projecting each
    point to the plane spanned by $\XC$ before matching the data points, as
    depicted in \textbf{(b)}.}
  \label{fig:cartesian}
\end{figure}

\begin{example*}
Consider $\YC = [0,1]^{d_2}$ for $d_2 \geq 1$ and let $\XC\subset\YC$ be a convex subset with dimension $d_1 \leq d_2$.
In \Cref{sec:LCA_primal}, we establish that the optimal transportation of any $\nu\in\PC(\YC)$ to any $\mu\in\PC(\XC)$ under squared Euclidean costs can be decomposed into two motions (see \Cref{fig:cartesian}): first an orthogonal projection onto the linear space spanned by $\XC$, and then an OT assignment within that linear space. Since such a projection is statistically negligible when compared to an OT assignment, it follows for $W_2(\mu,\nu) \geq \delta > 0$ by \eqref{eq:ChizatBound} that
\begin{equation}\label{eq:W2BoundOurs} 
  \EV{\big| W_2(\hat \mu_n, \hat\nu_n) - W_2(\mu, \nu)\big|} 
  \lesssim r_{2,d_1}^2(n)/\delta,
\end{equation}
which is independent of $d_2$, reflecting the LCA principle. 
\end{example*}

Our core contribution is to show that this phenomenon is a hallmark feature of empirical OT that far exceeds the scope of convex subsets and orthogonal projections.
To formalize our main result, let $\XC$ and $\YC$ be Polish spaces and consider a continuous bounded cost function $c\colon \XC\times \YC\rightarrow \RR$. In this setting, the OT cost enjoys a dual formulation \citep{villani2008optimal}
\begin{equation*}
  \OT_c(\mu, \nu) = \max_{f\in \FCot} \int_{\XC} f \,\dif\mu +  \int_{\YC} f\powerC \dif\nu
  ,
  \label{eq:DualFormulation}
\end{equation*}
where $\FCot$ is a suitable collection of uniformly bounded measurable functions on $\XC$ (defined in \Cref{sec:Ctransforms}) and $f\powerC(y) \coloneqq \inf_{x \in \XC} c(x,y)- f(x)$ denotes the $c$-transform of $f \in \FCot$. To investigate the empirical OT cost, we quantify the complexity of the class $\FCot$ and its $c$-transformed counterpart $\FCot\powerC=\{f\powerC\mid f \in \FCot\}$ in terms of their \emph{uniform metric entropy}. The uniform metric entropy of a class $\GC$ of real-valued functions on a set $\ZC$ is defined as the logarithm of the covering number with respect to uniform norm $\norm{\cdot}_\infty$, which is given for $\epsilon > 0$ by
\begin{equation*}
  \NC(\epsilon, \GC, \norm{\cdot}_\infty) \coloneqq \inf\left\{n\in \NN \,\Big|\, \text{there exist } g_1, \dots, g_n \colon \ZC\rightarrow \RR \text{ with }  \sup_{g \in \GC} \min\limits_{1\leq i \leq n}  \norm{g - g_i}_\infty \leq \epsilon\right\}.
\end{equation*}
A simple but crucial observation, which lies at the heart of this work, is that $c$-transformation with bounded costs is a Lipschitz operation under the uniform norm. Since $f^{cc} = f$ for all $f \in \FCot$, this in particular implies (\Cref{thm:coveringUnderCTransforms}) 
\begin{equation}\label{eq:IdenticalMetricEntropy}
  \NC(\epsilon, \FCot\powerC, \norm{\cdot}_\infty) = \NC(\epsilon, \FCot, \norm{\cdot}_\infty).
\end{equation}
This captures the LCA principle from the dual perspective: we only need to be able to control the complexity of either $\FCot$ or $\FCot\powerC$. Then, under the growth condition 
\begin{equation*}
	\log \NC(\epsilon, \FCot, \norm{\cdot}_\infty) \lesssim \epsilon^{-k}
\end{equation*}
for $\epsilon > 0$ sufficiently small and a fixed $k>0$, we prove for arbitrary $\mu\in \PC(\XC)$ and $\nu \in \PC(\YC)$ that any of the empirical estimators \begin{equation}\label{eq:EmpiricalEstimator}
  \OTn \in \{\OT_c(\hat \mu_n, \nu), \OT_c(\mu, \hat \nu_n), \OT_c(\hat \mu_n, \hat\nu_n)\}
\end{equation}
  satisfies the upper bound (\Cref{thm:AbstractUpperBound})
  \begin{equation*}
    \EV{\big| \OTn - \OT_c(\mu, \nu)\big|}
    \lesssim 
    \begin{cases}
      n^{-1/2}        & \text{if } k < 2,\\
      n^{-1/2}\log(n) & \text{if } k = 2,\\
      n^{-1/k}        & \text{if } k > 2.\\
    \end{cases}
  \end{equation*} 

As we discuss in \Cref{sec:Applications}, suitable bounds for the uniform metric entropy of the function class $\FCot$ are typically determined by the space $\XC$ and the regularity properties of the cost function. Since $\XC$ can be chosen as the support of $\mu$, these bounds can often be understood in terms of the intrinsic complexity of $\mu$ (or the one of $\nu$, if it turns out to be lower).
To explore the consequences of the LCA principle, we put special focus on the setting where the measure $\mu$ is supported on a space with \emph{low intrinsic dimension} while $\nu$ may live on a general Polish space. For example, we obtain convergence rates for semi-discrete OT, where $\mu$ is supported on finitely many points only. In this setting, we find that the empirical estimator $\OTn$ always enjoys the parametric rate (\Cref{thm:SemiDiscreteOT})
 \begin{equation*}
    \EV{\big| \OTn - \OT_c(\mu, \nu)\big|}\lesssim n^{-1/2}.
\end{equation*}
This complements distributional limits for the one-sample estimator $\OTn = \OT_{c}(\hat \mu_n, \nu)$  by \cite{delBarrio2021SemidiscreteCLT}.
We also derive metric entropy bounds for $\FC_c$ if $\XC$ is given in terms of the image of sufficiently regular functions on sufficiently nice domains, where we exploit the Lipschitz continuity, semi-concavity, or H\"older continuity of the cost function to obtain novel theoretical guarantees for the convergence rate of $\OTn$ (Theorems \ref{thm:LipschitzOT}, \ref{thm:SemiconcaveOT} and \ref{thm:HolderOT}). 
For example, a special case of \Cref{thm:SemiconcaveOT} states that the bound \eqref{eq:W2BoundOurs} for the $2$-Wasserstein distance on $\RR^d$ remains valid for compactly supported probability measures $\mu$ and $\nu$ if $\mu$ is concentrated on a $d_1$-dimensional $\CC^2$ submanifold (\Cref{ex:SemiconcaveExamples}$(iii)$).

 In \Cref{sec:Simulations}, we gather computational evidence for the LCA principle and present simulation results for various settings where the underlying measures have different intrinsic dimensions. In particular, we observe that the numerical findings are in line with the predictions of our theory.
Section \ref{sec:Discussion} concludes our work with a discussion and an outline of some open questions. \Cref{sec:OmittedProofs} contains bounds on the uniform metric entropy of $\FCot$.

\section{Lower Complexity Adaptation (LCA)}\label{sec:GeneralRates}

Throughout the manuscript, we work with absolutely bounded and continuous cost functions $c\,\colon \XC\times \YC \to \RR$ on Polish spaces $\XC$ and $\YC$. For convenience, we formulate our theory for costs whose range is restricted to the interval $[0, 1]$. Since $\OT_{ac + b} = a\cdot \OT_{c} + b$ for any $a>0$ and $b\in \RR$, this is not a genuine restriction, and all of our results can easily be adapted to general costs that are absolutely bounded.

\subsection{Duality and Complexity}\label{sec:Ctransforms}

In the following, we consider the dual formulation of the OT problem. This requires the notion of $c$-conjugacy under a given cost function~$c\,\colon \XC\times \YC \to [0, 1]$ for non-empty sets $\XC$ and $\YC$.
The \emph{$c$-transforms} of $f\,\colon\XC\to\RR$ and
$g\,\colon \YC\to\RR$ are defined by
\begin{equation}
  f^c(y) \coloneqq \inf_{x\in\XC} c(x,y) - f(x)
  \qquad\text{and}\qquad
  g^c(x) \coloneqq \inf_{y\in\YC} c(x,y) - g(y).
\end{equation}
A function $f\,\colon \XC \rightarrow \RR$  is called \emph{$c$-concave} if there exists
$g\,\colon \YC\rightarrow \RR$ such that $f = g^c$.
For Polish spaces $\XC$ and $\YC$ and a continuous cost function, any $c$-transform
$f^c$ or $g^c$ is upper semi-continuous (as an infimum over continuous
functions) and thus (Borel-)measurable. The following existence statement, which
is tailored to bounded costs, shows that dual solutions of OT can always
be assumed to be bounded and $c$-concave.

\begin{theorem*}[Duality]
  \label{prop:ExistenceOTPotential}
  Let $\XC$ and $\YC$ be Polish spaces and let $c\,\colon \XC\times \YC\rightarrow [0, 1]$ be continuous. 
  Denote the class of feasible $c$-concave potentials by
	\begin{equation}\label{eq:NiceTransforms}
  \FCot
  =
  \left\{ f\,\colon\XC\to [-1, 1] \,\Big|\,  f \text{ is $c$-concave with } \norm{f^c}_\infty \leq 1 \right\}.
  	\end{equation}
	Then, for any $\mu\in \PC(\XC)$ and $\nu \in \PC(\YC)$, it holds that
  \begin{equation}\label{eq:ExistenceOTPotential}
    \OT_c(\mu, \nu)
    =
    \max_{ f \in \FCot}\int  f \dif\mu+\int f\powerC \dif\nu.
  \end{equation}
\end{theorem*}
\begin{proof}
  Strong duality and the existence of maximizers in $\FCot$ follow from Theorem~5.10(iii) in \cite{villani2008optimal}, where the bounds on $f\in\FCot$ and $f^c$ are detailed in step~4 of the proof.
\end{proof}

The properties of the feasible $c$-concave potentials $\FCot$ strongly depend on the cost function $c$ and the ground space $\XC$. For example, if the family $\big\{c(\cdot, y)\,|\, y\in\YC\big\}$ of partially evaluated costs has a common modulus of continuity (with respect to a metric that metrizes $\XC$), then all $c$-concave functions $f\in\FCot$ are continuous with the same modulus. Hence, if $c(\cdot,y)$ is Lipschitz continuous uniform in $y \in \YC$, then each $ f \in\FCot$ is Lipschitz as well \citep[Section 1.2]{santambrogio2015optimal}. 
We denote the element-wise $c$-transform of the set $\FCot$ by $\FCot^c
= \{ f^c\,|\, f\in\FCot\}$, which is by definition also uniformly bounded by one. 
A crucial observation is that the uniform metric entropy of the function class
$\FCot^c$ is bounded by the one of $\FCot$, no matter how complex the space
$\YC$ is or how badly the functions $c(x, \cdot)$ for fixed $x\in\XC$ behave. 

\begin{lemma}[Complexity under $c$-transformation]\label{thm:coveringUnderCTransforms}
Let $\XC$ and $\YC$ be non-empty sets and $c \,\colon \XC\times \YC \rightarrow
[0,1]$. 
If $\FC$ is a bounded function class on $\XC$, it follows for $\epsilon>0$ that 
\begin{equation*}
	\NC( \epsilon, \FC\powerC, \norm{\cdot}_\infty) \leq \NC( \epsilon, \FC, \norm{\cdot}_\infty).
\end{equation*}
\end{lemma}

\begin{proof}
  If $N \coloneqq \NC(\epsilon,\FC, \norm{\cdot}_\infty) = \infty$, the claim is
  trivial, so assume $N < \infty$. Let $\{ f_1, \dots,  f_N\}$ be an
  $\epsilon$-covering for $\FC$ with respect to the uniform norm on $\XC$.  
  For $f \in \FC$, consider $f_i$ such that $\norm{f
  - f_i}_\infty \leq \epsilon$. Since $f$ and $c$ are both bounded,
  it follows that the $c$-transform $f\powerC$ is bounded on $\YC$.
  For all $y \in \YC$, we obtain
  \begin{align*}
		f\powerC(y) &= \inf_{x \in \XC} c(x,y) - f(x)
    =
    \inf_{x \in \XC} c(x,y) - f_i(x) + f_i(x) - f(x) \\
		&\hspace{-0.5em}\begin{cases}
			\leq
      \inf_{x \in \XC} c(x,y) - f_i(x)
      + \sup_{x \in \XC} | f_i(x) - f(x)|  \leq f_i\powerC(y) + \epsilon\\
			\geq
      \inf_{x \in \XC} c(x,y) - f_i(x)
      - \sup_{x \in \XC} | f_i(x) - f(x)| \geq f_i\powerC(y) - \epsilon,
		\end{cases}
	\end{align*}
  which implies $\sup_{y \in \YC}|f\powerC(y) -f_i\powerC(y)  | \leq \epsilon$. Thus, $\{ f_1\powerC, \dots,  f_N\powerC\}$ is an $\epsilon$-covering of $\FC^c$ with respect to the uniform norm.
\end{proof}

Since any feasible $c$-concave $f \in \FCot$ fulfills $f =  f^{cc}$ \cite[Proposition~1.34]{santambrogio2015optimal}, we conclude that the uniform metric entropies of the function classes $\FCot$ and $\FCot\powerC$ are identical for any covering radius $\epsilon>0$, see \eqref{eq:IdenticalMetricEntropy}.
Hence, to control the complexity of both function classes simultaneously, it suffices to upper bound only one of them. 
In particular, in a concrete setting where different bounds for $\FCot$ and $\FCot\powerC$ are available, their ($\epsilon$-wise) minimum can be employed to derive an upper bound for the convergence rate in \Cref{thm:AbstractUpperBound} below.
Methods to control the uniform metric entropy of $\FCot$ and $\FCot\powerC$ typically exploit regularity properties of the underlying spaces $\XC$ and $\YC$ as well as the ground cost $c$ (cf.\ \Cref{sec:Applications}).

\subsection{LCA: Dual Perspective} \label{sec:LowerComplexityAdapation}
The observation that $\FCot$ and $\FCot\powerC$ have identical uniform metric entropies implies the following upper bound on the convergence of the empirical estimators $\OTn$ in~\eqref{eq:EmpiricalEstimator}, demonstrating the LCA principle. Notably, in the two-sample case $\OTn = \OT_c(\hat\mu_n, \hat\nu_n)$, the statement holds irregardless of the dependency structure between the empirical measures $\hat\mu_n$ and $\hat\nu_n$. 

\begin{theorem}[General LCA]\label{thm:AbstractUpperBound}
  Let $\XC$ and $\YC$ be Polish spaces and let $c\,\colon \XC\times \YC\rightarrow[0, 1]$ be continuous. Consider the function class $\FCot$ from \eqref{eq:NiceTransforms} and assume that there exist $k >0$, $K>0$, and $\epsilon_0 \in (0, 1]$ such that the uniform metric entropy of $\FCot$ is bounded by
      \begin{equation}\label{eq:BoundOnUniformEntropy}
        \log \NC\big(\epsilon, \FCot, \norm{\cdot}_{\infty}\big)
        \leq K \epsilon^{-k}
        \qquad\text{for }
      0 < \epsilon \le \epsilon_0.
       \end{equation}
	Then, for any $\mu\in\PC(\XC)$ and $\nu\in\PC(\YC)$, the empirical estimator $\OTn$ from \eqref{eq:EmpiricalEstimator} satisfies
  \begin{equation}\label{eq:AbstractUpperBound}
    \EV{\big| \OTn - \OT_c(\mu, \nu)\big|}
    \lesssim
    \begin{cases}
      n^{-1/2}        & \text{if } k < 2,\\
      n^{-1/2}\log(n) & \text{if } k = 2,\\
      n^{-1/k}        & \text{if } k > 2,\\
    \end{cases}
  \end{equation}
  where the implicit constant only depends on $k$, $K$, and $\epsilon_0$.
\end{theorem}

When interpreting this statement, one should keep in mind that it is always possible to assume $\XC = \supp(\mu)$ and $\YC = \supp(\nu)$ if this leads to improved bounds for the uniform metric entropy of $\FCot$ (or equivalently $\FCot\powerC$). In this sense, \Cref{thm:AbstractUpperBound} can be tailored to take advantage of intrinsic properties of (the support of) $\mu$ (or $\nu$).
The proof employs arguments from empirical process theory and generalizes the technique of \cite{sriperumbudur2012empirical} and \cite{chizat2020}, where upper bounds for bounded convex sets $\XC = \YC\subseteq \RR^d$ and Euclidean costs $c(x,y) = \norm{x-y}^p$ for $p =1$ and $p = 2$ were derived.

\begin{proof}[Proof of Theorem~\ref{thm:AbstractUpperBound}] 
We first consider $\OTn = \OT_c(\hat \mu_n, \hat\nu_n)$. The definition of $\FCot$ implies
\begin{align*}
	\OT_c(\hat \mu_n, \hat\nu_n) - \OT_c(\mu, \nu) &= \max_{ f \in \FCot}\left( \int_{\XC}  f \dif \hat \mu_n + \int_{\YC}  f\powerC \dif \hat \nu_n\right)  - \max_{ f \in \FCot}\left( \int_{\XC}  f\dif \mu + \int_{\YC}  f\powerC \dif \nu \right)\\
	&\hspace{-0.5em}\begin{cases}
		\leq \phantom{-}\sup_{ f \in \FCot}\int_{\XC} f \,\dif (\hat \mu_n - \mu) + \int_{\YC}  f\powerC \dif ( \hat\nu_n - \nu),\\[.5ex]
		\geq -\sup_{ f\in \FCot}\int_{\XC}  f\,\dif (\mu - \hat \mu_n ) + \int_{\YC}  f\powerC \dif  ( \nu-\hat\nu_n ),
	\end{cases}
\end{align*}
and hence
\begin{align}\label{eq:UpperBoundEmpOT}
	|\OT_c(\hat \mu_n, \hat\nu_n)  - \OT_c(\mu, \nu)| \leq \sup_{ f \in \FCot}\left|\int_{\XC}  f \,\dif (\hat \mu_n - \mu)\right| + \sup_{ f\powerC \in \FCot\powerC}\left|\int_{\YC}  f\powerC \dif ( \hat\nu_n - \nu)\right|.
\end{align}
Note by \Cref{thm:coveringUnderCTransforms} that both $\FCot$ and $\FCot\powerC$ have finite uniform metric entropy for any $\epsilon>0$. Hence, both function classes contain subsets of at most countable cardinality that are dense in uniform norm, and the right hand side of \eqref{eq:UpperBoundEmpOT} can thus be considered as a countable supremum and is therefore measurable. 
Further, recall that all elements in $\FCot$ and $\FCot^c$ are absolutely bounded by one.
Taking the expectation and invoking symmetrization techniques (see \citealt[Proposition~4.11]{wainwright2019high}), we obtain
  \begin{align*}
   \EV{\big| \OT_c(\hat \mu_n, \hat\nu_n)  - \OT_c(\mu, \nu)\big|} 
    &\leq \EV{\sup_{ f \in \FCot}\left|\int_{\XC}  f \,\dif (\hat \mu_n - \mu)\right| }
    +  \EV{\sup_{ f\powerC \in \FCot\powerC}\left|\int_{\YC}  f\powerC \dif ( \hat\nu_n - \nu)\right|} \\
    &\leq 
    2\big(\RC_n(\FCot) + \RC_n(\FCot\powerC)\big),
  \end{align*}
  where $\RC_n(\FCot)$ and $\RC_n(\FCot\powerC)$ denote the Rademacher complexities of the function classes $\FCot$ and $\FCot\powerC$. The Rademacher complexity of $\FCot$ is defined by
  \begin{equation*}
  \RC_n(\FCot) \coloneqq \EV{\sup_{f\in \FCot}\left|\frac{1}{n}\sum_{i =1}^{n} \sigma_i f(X_i)\right| },	
  \end{equation*}
  for i.i.d.\ $X_1, \dots, X_n \sim \mu$ and independent i.i.d.\ Rademacher variables $\sigma_1, \dots, \sigma_n\sim \uniform\{-1,1\}$.
  This quantity is dominated by Dudley's entropy integral (see \citealt[Theorem~16]{luxburg2004distance}),
  \begin{equation*}
    \RC_n(\FCot)
    \leq 
    \inf_{\delta\in[0, 1]}
    \left(2\delta + \sqrt{32}\,n^{-1/2} \int_{\delta/4}^{1} \sqrt{\log\,\NC(\epsilon, \FCot, \norm{\cdot}_{\infty})}\,\dif\epsilon\right).
  \end{equation*}
  Let $\tilde{K} \coloneqq \sqrt{32\,K}$. Since the covering number is a decreasing function in $\epsilon$, a short calculation shows that the assumption $\log \NC(\epsilon, \FCot, \|\cdot\|_\infty) \leq K \epsilon^{-k}$ for $\epsilon \leq \epsilon_0$ implies that
  \begin{align*}
    \RC_n(\FCot)\leq \;&
    \inf_{\delta\in [0,1]}\left(2\delta +\tilde{K} n^{-1/2} \int_{\delta/4}^{1} \min(\epsilon, \epsilon_0)^{-k/2} \dif\epsilon\right)\\
     \leq \;&
    \begin{cases}
     \tilde{K}\left( \frac{\epsilon_0^{1-k/2}}{1 - k/2} + \frac{1 - \epsilon_0}{\epsilon_0^{k/2}} \right) n^{-1/2}        & \text{ if }  k < 2 \text{ for } \delta = 0,\\
    \left(8+ \tilde{K}\log(\epsilon_0) + \frac{\tilde{K}(1 - \epsilon_0)}{\epsilon_0} \right)  n^{-1/2}  +  \frac{\tilde{K}}{2}n^{-1/2}\log(n)& \text{ if }  k = 2 \text{ for } \delta = 4n^{-1/2},\\
 	 \tilde{K} \left(\frac{\epsilon_0^{1-k/2}}{1 - k/2} + \frac{1 - \epsilon_0}{\epsilon_0^{k/2}} \right)   n^{-1/2}   +\left(8 + \frac{\tilde{K}}{k/2-1}\right)  n^{-1/k}       & \text{ if }  k > 2 \text{ for } \delta = 4n^{-1/k}, 
    \end{cases}
   \end{align*}
   where we assume $n$ large enough such that $\delta \leq \epsilon_0$ for the respective choices. As $\FCot$ and $\FCot\powerC$ share the same covering number with respect to uniform norm (\Cref{thm:coveringUnderCTransforms}) and since all functions in $\FCot\powerC$ are bounded in absolute value by one as well, it follows that $\RC_n(\FCot\powerC)$ can be bounded just like $\RC_n(\FCot)$. 
  Finally, for the other estimators $\OTn$ in \eqref{eq:EmpiricalEstimator}, we obtain
  \begin{equation*}
    \EV{\big| \OT_c(\hat \mu_n, \nu)  - \OT_c(\mu, \nu)\big|} \leq  \RC_n(\FCot)
    \qquad\text{and}\qquad
    \EV{\big| \OT_c(\mu, \hat\nu_n)  - \OT_c(\mu, \nu)\big|} \leq  \RC_n(\FCot\powerC)
  \end{equation*}
  in analogous fashion, which proves the bounds from \eqref{eq:AbstractUpperBound} and finishes the proof.
\end{proof}

\subsection{LCA: Primal Perspective}\label{sec:LCA_primal}

The proof of \Cref{thm:AbstractUpperBound} relies on a technical observation about the nature of $c$-transforms and does not convey a geometric interpretation why empirical OT should follow the LCA principle. To provide some additional intuition, we next consider the LCA phenomenon from the primal perspective. Even though this approach yields less general results, its more explicit character has benefits, e.g., for establishing matching lower bounds.

\begin{proposition}[Decomposition under additive costs]\label{prop:LCAadditiveCosts}
  Let $\XC$ be a Polish space and $\YC = \YC_1\times\YC_2$ be the product of two Polish spaces. Let $c\,\colon \XC \times \YC \to [0,1]$ be continuous so that
  \begin{equation}\label{eq:cartesianCosts}
    c(x,y) = c_1(x, y_1) + c_2(y_2)
  \end{equation}
  for all $x\in \XC$ and $y = (y_1,y_2) \in \YC$ with continuous $c_1 \,\colon \XC\times \YC_1 \rightarrow [0, 1]$ and $c_2\,\colon \YC_2\rightarrow [0, 1]$, and let
 $\pf\,\colon \YC \to \YC_1$ be the Cartesian projection to $\YC_1$.
 Then, for any $\mu\in\PC(\XC)$ and $\nu\in\PC(\YC)$,
\begin{equation}
    \OT_c(\mu, \nu) = \OT_{c_1}(\mu, \pf_{\#}\nu) + R_{c_2}(\nu), \label{eq:CartesianDecomposition}
\end{equation}
where $R_{c_2}(\nu) := \int_{\YC} c_2(y_2)\,\dif\nu(y_1, y_2)$.
\end{proposition}

For example, this statement can be applied in Euclidean
spaces under $l_p^p$ costs if $\XC \subset \YC_1$.  
In this case, we find $c(x, y) = \|x - y_1\|_p^p + \|y_2\|_p^p$ (see also
\Cref{fig:cartesian} in the introduction), which satisifies
condition~\eqref{eq:cartesianCosts}. If $p = 2$, a relation of the form
\eqref{eq:CartesianDecomposition} clearly remains valid whenever $\XC$ is
contained in an affine linear subspace of $\YC$ and $\pf$ is the corresponding
orthogonal projection.
 
\begin{proof}[Proof of \Cref{prop:LCAadditiveCosts}]
 For any coupling $\pi\in\Pi(\mu, \nu)$, we consider the decomposition
\begin{equation*}
  \int_{\XC\times \YC} c\,\dif\pi = \int_{\XC\times \YC} c_1 \,\dif\pi + \int_{\XC\times \YC} c_2\,\dif\pi. 
\end{equation*}
The second term on the right is independent of $\pi$ and equals $R_{c_2}(\nu) := \int_{\YC} c_2(y_2)\,\dif\nu(y_1, y_2)$, while the first term can be
rewritten in terms of $\tilde\pi = (\id, \pf)_{\#}\pi \in \Pi(\mu, \pf_{\#}\nu)$ by a change of variables, such that $\int_{\XC\times \YC}  c_1\,\dif\pi = \int_{\XC\times \YC_1} c_1\,\dif\tilde\pi$.
Conversely, the gluing lemma (cf. \citealt[Chapter~1]{villani2008optimal}) implies that each $\tilde\pi \in \Pi(\mu, \pf_{\#}\nu)$ gives rise to
a $\pi\in\Pi(\mu, \nu)$ for which $\int_{\XC\times \YC}  c_1\,\dif \pi = \int_{\XC\times \YC_1}
c_1\,\dif\tilde\pi$ holds as well. Therefore, we conclude that
\begin{align}
  \OT_c(\mu, \nu)
  &=
  \inf_{\pi\in\Pi(\mu, \nu)} \int_{\XC\times \YC} c_1\,\dif\pi + \int_{\XC\times \YC} c_2\,\dif\pi \notag\\
  &=
  \inf_{\tilde\pi\in\Pi(\mu, \pf_{\#}\nu)} \int_{\XC\times \YC_1} c_1\,\dif\tilde\pi + R_{c_2}(\nu) 
  = 
  \OT_{c_1}(\mu, \pf_{\#}\nu) + R_{c_2}(\nu).\;\;\qedhere 
\end{align}
\end{proof}
Since relation \eqref{eq:CartesianDecomposition} in \Cref{prop:LCAadditiveCosts} holds for any pair of probability measures $\mu\in \PC(\XC)$ and $\nu\in\PC(\YC)$, we
obtain
\begin{equation}\label{eq:RelationEmpiricalMeasuresAdditiveCosts}
	\OT_c(\hat \mu_n,\hat \nu_n) - \OT_c(\mu, \nu) = \OT_{c_1}(\hat\mu_n, \pf_{\#}\hat\nu_n) -  \OT_{c_1}(\mu, \pf_{\#}\nu) + R_{c_2}(\hat\nu_n - \nu),
\end{equation}
where $\hat\nu_n - \nu$ is understood as a signed measure.
Thus, the statistical performance when estimating $\OT_c(\mu, \nu)$ via the empirical OT cost is governed by the (potentially much simpler) complexity of $\OT_{c_1}(\mu, \pf_{\#}\nu)$.
By the reverse triangle inequality, it furthermore follows that
\begin{equation}\label{eq:LowerBoundTool1}
  \EV{|\OT_c(\hat \mu_n,\hat \nu_n) - \OT_c(\mu, \nu)|} \geq \EV{\left||\OT_{c_1}(\hat\mu_n, \pf_{\#}\hat\nu_n) -  \OT_{c_1}(\mu, \pf_{\#}\nu)| - |R_{c_2}(\hat\nu_n - \nu)|\right|}, 
\end{equation}
where we note $\EV{|R_{c_2}(\hat\nu_n - \nu)|}\asymp n^{-1/2}$ if $\sigma_{c_2}^2\coloneqq \Var_{Y\sim \nu}[c_2(Y)]>0$. Let $\hat f_n \in \FC_{c_1}$ denote an optimizer for \eqref{eq:ExistenceOTPotential} between $\hat \mu_n$ and $\pf_{\#}\nu$ under the cost function $c_1$.
If the i.i.d.\ random variables $X_1, \dots, X_n\sim \mu$ and $Y_1, \dots, Y_n\sim \nu$ are independent, we find
\begin{align}
  &\EV{\OT_{c_1}(\hat\mu_n, \pf_{\#}\hat\nu_n) -  \OT_{c_1}(\hat\mu_n, \pf_{\#}\nu)} \notag
  \\
  &\qquad\qquad\qquad= \EV{ \max_{f \in \FC_{c_1}} \left(\int_{\XC} f \dif \hat\mu_n +  \int_{\YC_1} f^{c_1} \dif \pf_{\#}\hat\nu_n\right) - \max_{f \in \FC_{c_1}}\left(\int_{\XC} f \dif \hat\mu_n +  \int_{\YC_1} f^{c_1} \dif \pf_{\#}\nu \right) }\notag\\
  &\qquad\qquad\qquad\geq \mathbb{E}_{X_1, \dots, X_n}\left[ \mathbb{E}_{Y_1, \dots, Y_n}\left[ \int_{\YC_1} \hat f_n^{c_1} \dif (\pf_{\#}\hat\nu_n - \pf_{\#} \nu) \,\Big|\, X_1, \dots, X_n\right] \right] \notag\\
  &\qquad\qquad\qquad= 0, \notag 
\end{align}
where independence is crucial for the final equality.
In particular, this implies  
 \begin{align}
 	\EV{|\OT_{c_1}(\hat\mu_n, \pf_{\#}\hat\nu_n) -  \OT_{c_1}(\mu, \pf_{\#}\nu)|} &\geq \EV{\OT_{c_1}(\hat\mu_n, \pf_{\#}\hat\nu_n) -  \OT_{c_1}(\mu, \pf_{\#}\nu)} \notag\\
 &\geq \EV{\OT_{c_1}(\hat\mu_n, \pf_{\#}\nu) -  \OT_{c_1}(\mu, \pf_{\#}\nu)}, \label{eq:LowerBoundTool2}
 \end{align}
which means that lower bounds for the two-sample setting can be obtained from lower bounds for the one-sample case.
In Examples~\ref{ex:LowerBoundsLipschitzcosts} and \ref{ex:LowerBoundsSemiconcavecosts} of the subsequent section, we employ relation \eqref{eq:LowerBoundTool1} and \eqref{eq:LowerBoundTool2} to this end. 

\begin{remark}[Dependent empirical measures] \label{rem:Dependencies} 
Dependencies between the empirical measures $\hat\mu_n$ and $\hat\nu_n$ can lead to parametric convergence rates of order $n^{-1/2}$ irregardless of the underlying spaces $\YC_1$ and $\YC_2$. For example, if $\XC = \YC_1$ and $\mu = \pf_{\#}\nu$ with empirical measures related by $\hat\mu_n = \pf_\#\hat\nu_n$, it follows from \eqref{eq:RelationEmpiricalMeasuresAdditiveCosts} that
\begin{equation*}
  \EV{\big|\OT_c(\hat \mu_n,\hat \nu_n) - \OT_c(\mu, \nu)\big|} = \EV{\big|R_{c_2}(\hat\nu_n - \nu)\big|} \asymp  n^{-1/2}
\end{equation*}
for any non-negative cost function $c_1$ with $c_1(y_1,y_1)= 0$ for all $y_1\in \YC_1$ if $\sigma_{c_2}>0$.
\end{remark}

\section{Applications and Examples}\label{sec:Applications}

The LCA principle, as formalized in \Cref{thm:AbstractUpperBound}, can readily be employed whenever suitable bounds on the uniform metric entropy of $\FCot$ are available. In the following, we consider a number of settings where well-known entropy bounds lead to novel results on the convergence rate of empirical OT. In order to efficiently exploit the properties of the space $\XC$ in settings of low intrinsic dimensionality, the following observation is useful.

\begin{lemma}[Union bound]\label{lem:MetricEntropyUnionBound}
Let $\FC$ be a class of real valued functions on a set $\XC = \bigcup_{i =1}^{I}\XC_i$ for (not necessarily disjoint) subsets $\XC_i\subset\XC$ and $I \in \NN$, and let $\FC|_{\XC_i}\coloneqq \big\{ f|_{\XC_i} \,\colon \XC_i\rightarrow \RR\,|\,f\in\FC \big\}$ be the class of functions restricted to $\XC_i$ for all $i\in\{1, \ldots, I\}$. Then, for each $\epsilon>0$,
\begin{equation*}
	\log\NC(\epsilon, \FC, \norm{\cdot}_\infty)\leq \sum_{i=1}^{I}\log\NC(\epsilon, \FC|_{\XC_i}, \norm{\cdot}_\infty).
\end{equation*}
\end{lemma}
\begin{proof}
	Suppose that the right hand side is finite (otherwise the bound is trivial). Denote by ${\FC}_i$ a minimal $\epsilon$-covering of $\FC|_{\XC_i}$ with respect to the uniform norm. Let $\tilde \XC_1 \coloneqq \XC_1$ and define $\tilde \XC_i \coloneqq \XC_i \backslash \left(\bigcup_{j= 1}^{i-1} \XC_j\right)$ for $i \geq 2$. Then $\big\{ \sum_{i=1}^{I}f_i \cdot \mathds{1}_{\tilde \XC_i}  \,|\, f_i \in \FC_i\big\}$ is an $\epsilon$-covering of $\FC$ under the uniform norm with cardinality at most $\smash{\prod_{i=1}^{I} \NC(\epsilon, \FC|_{\XC_i}, \norm{\cdot}_\infty)}$. 
\end{proof}

\subsection{Semi-Discrete Optimal Transport}\label{sec:Semidiscrete}

We first address the setting of \emph{semi-discrete} OT, where $\XC = \{x_1, \dots, x_I\}$ is a finite discrete space with $I\in\NN$ elements.
Structural and computational properties of semi-discrete OT have been investigated extensively, especially in Euclidean contexts (see, e.g., \citealt{aurenhammer1998minkowski,merigot2011multiscale, geiss2013optimally,hartmann2020semi}). 
For our purposes, consider a general Polish space $\YC$ and a continuous cost function $c\,\colon \XC\times \YC\rightarrow [0,1]$. By definition \eqref{eq:NiceTransforms}, the function class $\FCot$ is absolutely bounded by one and we find that
\begin{equation}\label{eq:MetricEntropyDiscrete}
   \NC(\epsilon, \FCot|_{\{x_i\}},  \norm{\cdot}_{\infty}) \leq \left\lceil1 /\epsilon\right\rceil
\end{equation}
for any $\epsilon > 0$ and $i\in\{1,\ldots,I\}$.
Hence, it follows by \Cref{lem:MetricEntropyUnionBound} that $\log \NC( \epsilon, \FCot, \norm{\cdot}_{\infty}) \leq I \log\left\lceil 1/\epsilon\right\rceil \lesssim I /\epsilon$  and we can apply \Cref{thm:AbstractUpperBound} to derive the following bound.

\begin{theorem}[Semi-discrete LCA]\label{thm:SemiDiscreteOT}
   Let $\XC= \{x_1, \dots, x_I\}$ be a finite discrete space, $\YC$ a Polish space, and $c\,\colon
   \XC\times \YC \to [0,1]$ continuous. 
   Then, for any $\mu\in\PC(\XC)$ and $\nu\in\PC(\YC)$, the empirical estimator $\OTn$ from \eqref{eq:EmpiricalEstimator} satisfies
  \begin{align}\label{eq:SemiDiscreteOTBound}
    \EV{\big|\OTn  - \OT_c(\mu, \nu)\big|} \lesssim n^{-1/2}.
  \end{align}
\end{theorem}

This result is in line with recent findings by \cite{delBarrio2021SemidiscreteCLT}, who derive a central limit theorem for the empirical semi-discrete OT cost. Their result allows for possibly unbounded costs, but it is limited to the one-sample estimator $\OTn = \OT_c(\hat\mu_n, \nu)$. According to Markov's inequality, \Cref{thm:SemiDiscreteOT} implies that the sequence of random variables $\sqrt{n}\big(\OTn - \OT_c(\mu, \nu)\big)$ is tight even for $\OTn = \OT_c(\mu, \hat\nu_n)$ or $\OT_c(\hat\mu_n, \hat\nu_n)$, which indicates that it might also be possible to derive limit distributions when the measure on the general space $\YC$ is estimated empirically.
Moreover, \Cref{thm:SemiDiscreteOT} asserts novel bounds for the Wasserstein distance when one measure is supported on finitely many points only while the support of the other measure is bounded.

\subsection{Optimal Transport under Lipschitz costs}
\label{ssec:lipschitz}

Semi-discrete OT can be regarded as a special OT setting where one probability measure has intrinsic dimension zero. We now broaden this perspective to higher dimensions and consider parameterized spaces and surfaces. The effective dimension is then governed by the (possibly low-dimensional) domain of the parameterization, and not by the (possibly high-dimensional) ambient space. 
In this section, we work with an additional Lipschitz requirement for the cost function $c\,\colon \XC\times \YC\rightarrow [0,1]$ and impose the following condition on the Polish space $\XC$. By rescaling, we may assume that the Lipschitz constant is equal to one.

\begin{assumption}\label{as:LipschitzCosts}
	Suppose $\XC= \bigcup_{i = 1}^{I} g_i(\UC_i)$ for $I \in \NN$ connected metric spaces $(\UC_{i}, d_i)$ and maps $g_i\,\colon\UC_i \to \XC$ so that $c(g_i(\cdot), y)$ is $1$-Lipschitz with respect to $d_i$ for all  $y\in\YC$.
\end{assumption}

This setting captures a broad notion of generalized surfaces in an ambient space. Since the mappings $g_i$ are not required to be injective, self-intersections are possible. Moreover, exploiting the (not necessarily disjoint) decomposition $\XC= \bigcup_{i = 1}^{I}\XC_i$ with $\XC_i \coloneqq g_i(\UC_i)$ for $i \in \{1, \dots, I\}$, it suffices by \Cref{lem:MetricEntropyUnionBound} to control the complexity of $\FCot|_{\XC_i}$ to bound the metric entropy of $\FCot$. For this purpose, we note that the Lipschitz continuity of $c(g_i(\cdot),y)$ implies that $f \circ g_i$ for any $c$-concave potential $f\in \FCot$ is Lipschitz as well. This relates the restricted function class ${\FCot|}_{\XC_i}$ to the class of bounded Lipschitz functions on $\UC_i$. For the latter, metric entropy bounds in terms of the covering number of $\UC_i$ are available
\citep[Section 9]{Kolmogorov1961}. For $\epsilon > 0$, the covering number of a metric space $(\UC, d)$ is defined by
\begin{equation*}
   \NC(\epsilon, \UC, d) \coloneqq \inf\left\{ n \in \NN \,\Big|\, \text{there exist } U_1, \dots, U_n\subseteq \UC \text{ with } \diam{U_k}\leq 2 \epsilon \text{ and } \UC = \bigcup_{k = 1}^{n} U_k\right\},
\end{equation*}
where $\diam{U}\coloneqq \sup_{u,v\in U}d(u,v)$ denotes the diameter of a subset $U \subseteq \UC$.

\begin{theorem}[Lipschitz LCA]\label{thm:LipschitzOT}
	Let $\XC$ and $\YC$ be Polish spaces and let $c\,\colon \XC\times \YC\rightarrow [0,1]$ be continuous. If \Cref{as:LipschitzCosts} holds and there exists $k > 0$ so that for all $i \in \{1, \ldots, I\}$
  \begin{equation*}
    \NC(\epsilon, \UC_i, d_i)\lesssim \epsilon^{-k}\qquad\text{for $\epsilon > 0$ sufficiently small},
  \end{equation*}
  then, for any $\mu\in\PC(\XC)$ and $\nu\in\PC(\YC)$, the empirical estimator $\OTn$ from \eqref{eq:EmpiricalEstimator} satisfies
  \begin{equation}
 \label{eq:LipschitzOTBound_Good}
    \EV{\big|\OTn  - \OT_c(\mu, \nu)\big|}
    \lesssim \begin{cases}
 	 n^{-1/2} & \text{ if } k < 2,\\
 	 n^{-1/2}\log(n) & \text{ if } k = 2,\\
 	 n^{-1/k} & \text{ if } k > 2.
 \end{cases}
  \end{equation}
\end{theorem}

\begin{proof}
  \Cref{lem:LipschitzMetricEntropy} in \Cref{sec:OmittedProofs} shows that the uniform metric entropy in this setting is bounded by $\log \NC( \epsilon, \FC_{c},\norm{\cdot}_\infty)\lesssim \epsilon^{-k}$. Applying \Cref{thm:AbstractUpperBound} then yields bound \eqref{eq:LipschitzOTBound_Good}.
\end{proof}

\begin{remark}[Disconnected domains]\label{rem:Connectedness}
If the metric spaces $\UC_1, \dots, \UC_I$ in \Cref{as:LipschitzCosts} consist of finitely many connected components, \Cref{thm:LipschitzOT} remains valid at the price of a possibly larger constant. Moreover, if some $\UC_i$ has infinitely many components, \Cref{lem:LipschitzMetricEntropy} ensures
\begin{equation*}
  \log \NC( \epsilon, \FC_{c},\norm{\cdot}_\infty)\lesssim \epsilon^{-k}\log(\epsilon^{-1})\lesssim \epsilon^{-k-\delta}
\end{equation*}
for any $\delta > 0$ (where the implicit constant depends on $\delta$). 
 Hence, \Cref{thm:AbstractUpperBound} shows that bound \eqref{eq:LipschitzOTBound_Good} still holds when $k$ is replaced by $k+\delta$. 
\end{remark}

We emphasize once more that no additional assumptions on the complexity of the Polish space $\YC$ are necessary. To highlight applications and noteworthy consequences of \Cref{thm:LipschitzOT}, we consider a number of examples. 

\begin{example}[Metric spaces with Lipschitz costs]\label{ex:LipschitzMetricspaces}
	Let $\XC$ and $\YC$ be closed subsets of a Polish metric space $(\ZC,d)$ and consider costs $c\,\colon \ZC^2\rightarrow \RR$ that are continuous and absolutely bounded on $\XC\times \YC$. Furthermore, assume that $c(\cdot,y)$ is Lipschitz on $\XC$ uniformly in $y\in \YC$, which, for example, holds for $c(x,y) = d^p(x,y)$ and $p \geq 1$ if $\YC$ is bounded.
	Then, \Cref{thm:LipschitzOT} provides convergence rates whenever $\NC(\epsilon, \XC, d)\lesssim \epsilon^{-k}$. This condition holds if the \emph{upper Minkowski-Bouligand dimension} of $\XC$ \citep[Section 5.3]{mattila1995geometry}, defined by
\begin{equation*}
  \dimM{\XC}\coloneqq \limsup_{\epsilon \searrow 0} \frac{\log \NC(\epsilon, \XC, d)}{\log(\epsilon^{-1})},
\end{equation*}
is strictly dominated by $k$. Note that non-integral values of $\dimM{\XC}$ are possible and that $\XC$ does not necessarily have to be connected (\Cref{rem:Connectedness}).	
\end{example}

\begin{example}[Euclidean spaces with locally Lipschitz costs]\label{ex:LipschitzCostIntrinsicDimension}
	Consider a locally Lipschitz cost function $c\,\colon \RR^{d}\times \RR^d\rightarrow \RR$ for $d\in\NN$. This setting entails the choice $c(x,y) = \norm{x-y}^p$ for the Euclidean norm $\norm{\cdot}$, or the $l_p^p$-costs $c(x,y) = \sum_{i = 1}^{d}|x_i - y_i|^p$, both of which are popular options for the Wasserstein distance of order $p \ge 1$. In the following examples, we implicitly assume that $\XC$ and $\YC$ are Polish subsets of $\RR^d$.
	\begin{enumerate}[leftmargin=1.5em]
		\item[$(i)$] \emph{Bounded sets}: If $\XC, \YC\subset\RR^d$ are bounded sets, then $c(\cdot,y)$ is Lipschitz continuous on $\XC$ uniformly in $y \in \YC$ and $\NC(\epsilon, \XC, \norm{\cdot})\lesssim \epsilon^{-d}$. Since one can always enlarge $\XC$ and $\YC$ to connected compact sets, \Cref{thm:LipschitzOT} asserts that the convergence rate of the empirical OT cost is dominated by \eqref{eq:LipschitzOTBound_Good} with $k=d$. 

		\item[$(ii)$] \emph{Surfaces}: Improved bounds can be derived if $\XC = \bigcup_{i = 1}^{I} g_i(\UC_i)$ for Lipschitz maps $g_i\,\colon \UC_i \rightarrow \RR^d$ on bounded and connected sets $\UC_i \subseteq \RR^s$ with $d > s\in \NN$.
      In this setting, the partially evaluated cost $c\big(g_i(\cdot),y\big)$ is Lipschitz on $\XC$ uniformly in $y \in \YC$ for bounded sets $\YC$, and it holds that $\NC(\epsilon, \UC_i, \norm{\cdot})\lesssim \epsilon^{-s}$. Hence, bound \eqref{eq:LipschitzOTBound_Good} is valid for $k = s$. 

		\item[$(iii)$] \emph{Manifolds}: The examples outlined in $(ii)$ include compact $s$-dimensional immersed $\CC^1$ submanifolds of $\RR^d$, possibly with boundary \citep{lee2013smooth}. Due to compactness, a finite atlas of charts with connected and bounded co-domains $\UC_i$ can always be found. 
		\end{enumerate} 
\end{example}
	
We next show that the upper bounds in \Cref{thm:LipschitzOT} can be complemented by matching lower bounds in settings where the primal approach to the LCA principle (see \Cref{prop:LCAadditiveCosts}) is available. 
	
	\begin{example}[Lower bounds under Lipschitz costs]\label{ex:LowerBoundsLipschitzcosts}
  Let $\YC = \YC_1\times \YC_2$ be the product of two connected Polish spaces and let $\XC = \YC_1$. Consider the costs $c(x,y) = d_1(x,y_1) + c_2(y_2)$, where $d_1$ is a Polish metric on $\YC_1$ such that $0 < \diam{\YC_1} \le 1$ and $c_2\,\colon\YC_2\to[0,1]$ is continuous.
  Then $c(\cdot,y)$ is Lipschitz with respect to $d_1$ uniformly in $y \in \YC$, so \Cref{as:LipschitzCosts} is fulfilled. If $\NC(\epsilon, \XC, d_1) \asymp \epsilon^{-k}$ for some $k > 0$ as $\epsilon\searrow 0$ and if the empirical measures $\hat \mu_n$ and $\hat \nu_n$ are independent, one can show that
  \begin{equation}\label{eq:LowerBoundLipschitz}
  \sup_{\mu, \nu}\,\EV{|T_{c}(\hat\mu_n, \hat\nu_n) -  	T_{c}(\mu, \nu)|} \gtrsim 
	\begin{cases} 
	n^{-1/2} & \text{ if } k \leq 2,\\
	n^{-1/k} & \text{ if } k > 2,
	\end{cases}
\end{equation}
where the supremum is taken over $\mu\in\PC(\XC)$ and $\nu\in\PC(\YC)$.
	For $k\leq 2$, inequality~\eqref{eq:LowerBoundLipschitz} follows by selecting suitable discrete measures $\mu$ and $\nu$ (see \citealt[Section~5]{sommerfeld19FastProb}).
	For $k>2$, inequality \eqref{eq:LowerBoundLipschitz} follows by \eqref{eq:LowerBoundTool1} and \eqref{eq:LowerBoundTool2} in conjunction with the minimax lower bounds for the $1$-Wasserstein distance $W_1(\mu, \nu) = T_{d_1}(\mu, \nu)$ by \citet[Theorem~2]{singh2018minimax}
	\begin{equation*}
    \sup_{\mu, \nu}\,\EV{|T_{d_1}(\hat\mu_n, \pf_{\#}\nu) -  	T_{d_1}(\mu, \pf_{\#}\nu)|} \geq \sup_{\mu}\,\EV{T_{d_1}(\hat\mu_n, \mu)} \gtrsim n^{-1/k}.
  \end{equation*}
	Hence, the upper bounds from \eqref{eq:LipschitzOTBound_Good} are sharp for $k \neq 2$ and sharp up to logarithmic terms in case of $k = 2$. 
\end{example}
	
\subsection{Optimal Transport under Semi-Concave Costs}
\label{ssec:semi}

It is known that better convergence rates than those offered by \Cref{thm:LipschitzOT} can be obtained on Euclidean spaces for more regular cost functions \citep{Manole21}.
In this section, we consider improvements for semi-concave costs (before we turn to H\"older smooth costs in \Cref{ssec:holder}).
A function $f\,\colon U\rightarrow \RR$ on a bounded, convex domain $U\subset\RR^d$ for $d \in \NN$ is called $\Lambda$-\emph{semi-concave} with \emph{modulus} $\Lambda \geq 0$ if the map
\begin{equation*}
  u\mapsto f(u) - \Lambda \norm{u}^2
\end{equation*}
is concave, where $\norm{\cdot}$ denotes the Euclidean norm (cf.\ \citealt{gangbo1996geometry}). The uniform metric entropy of the class of bounded, Lipschitz, and semi-concave functions on $U$ is of order $\epsilon^{-d/2}$ \citep{bronshtein1976varepsilon, guntuboyina2012covering}, compared to $\epsilon^{-d}$ without semi-concavity. Since boundedness, Lipschitz continuity, and semi-concavity are all inherited to $\FCot$ by the cost function, we impose the following conditions. For convenience, we assume the Lipschitz constant and the modulus of semi-concavity to be equal to one, since other constants can be accommodated by scaling the cost function. 

\begin{assumption}
	\label{as:SemiconcaveCosts}
  Suppose $\XC= \bigcup_{i = 1}^{I}g_i(\UC_i)$ for $I\in\NN$ bounded, convex subsets $\UC_i\subseteq \RR^d$ and maps $g_i\,\colon \UC_i \rightarrow \XC$ so that $c(g_i(\cdot),y)$ is $1$-Lipschitz and $1$-semi-concave for all $y\in\YC$.
\end{assumption}

Similar to \Cref{as:LipschitzCosts} in the previous section, \Cref{as:SemiconcaveCosts} enables the application of \Cref{lem:MetricEntropyUnionBound} to the union $\XC= \bigcup_{i = 1}^{I}\XC_i$ with $\XC_i = g_i(\UC_i)$ for all $i \in \{1, \dots, I\}$. Combined with the uniform metric entropy bounds by \cite{bronshtein1976varepsilon} and \cite{guntuboyina2012covering}, this leads to the following result. 

\begin{theorem}[Semi-concave LCA]\label{thm:SemiconcaveOT}
  Let $\XC$ and $\YC$ be Polish spaces and $c \,\colon \XC\times \YC \rightarrow [0,1]$ be continuous. If \Cref{as:SemiconcaveCosts} holds, then, for any $\mu \in \PC(\XC)$ and $\nu \in \PC(\YC)$, the empirical estimator $\OTn$ from \eqref{eq:EmpiricalEstimator} satisfies  
   \begin{equation}\label{eq:SemiconcaveBound}
    \EV{\big|\OTn  - \OT_c(\mu, \nu)\big|} \lesssim \begin{cases}
 	 n^{-1/2} & \text{ if } d  \leq 3,\\
 	 n^{-1/2}\log(n) & \text{ if } d = 4,\\
 	 n^{-2/d} & \text{ if } d \geq 5.
 \end{cases}
  \end{equation}
\end{theorem} 

\begin{proof}
  \Cref{lem:SemiconcaveMetricEntropy} in \Cref{sec:OmittedProofs} shows that the uniform metric entropy in this setting is bounded by
  $\log\NC(\epsilon, \FCot, \norm{\cdot}_\infty)\lesssim \epsilon^{-d/2}$. This implies bound \eqref{eq:SemiconcaveBound} via \Cref{thm:AbstractUpperBound}.
\end{proof}

Compared to \Cref{thm:LipschitzOT}, which would guarantee a convergence rate of $n^{-1/d}$ for $d \geq 5$ under \Cref{as:SemiconcaveCosts}, \Cref{thm:SemiconcaveOT} provides the considerably faster rate $n^{-2/d}$. To explore applications, we revisit the settings discussed in \Cref{ex:LipschitzCostIntrinsicDimension} and show how they fare under the stronger \Cref{as:SemiconcaveCosts}.
For this purpose, note that semi-concavity of a $\CC^2$ function $f$ on a convex domain is implied by the boundedness of the Eigenvalues of its Hessian. Indeed, if the Eigenvalues are bounded from above by $2\Lambda\geq 0$, then $u \mapsto f(u) - \Lambda \norm{u}^2$ has a negative semi-definite Hessian and is therefore concave (see \citealt[Section~6.4, Proposition~5]{luenberger2003linear}).

\begin{example}[Euclidean spaces with $\CC^2$ costs]\label{ex:SemiconcaveExamples}
  Extending \Cref{ex:LipschitzCostIntrinsicDimension}, consider a twice continuously differentiable cost function $c\,\colon \RR^d\times \RR^d\rightarrow \RR$ for $d\in\NN$.
  This setting includes $c(x,y) = \norm{x-y}^p$ as well as $c(x,y) = \sum_{i = 1}^{d}|x_i - y_i|^p$ for $p \geq 2$. We implicitly assume that $\XC$ and $\YC$ are Polish subsets of $\RR^d$ in the following.

	\begin{enumerate}[leftmargin=1.5em]
		\item[$(i)$] \emph{Bounded sets}: Let $\XC, \YC \subset\RR^d$ be bounded sets. Since $\XC$ and $\YC$ can always be enlarged to convex compact sets, the functions $c(\cdot, y)$ are Lipschitz continuous and semi-concave on $\XC$ uniformly in $y\in\YC$. Hence, \Cref{thm:SemiconcaveOT} provides the upper bounds \eqref{eq:SemiconcaveBound}.

		\item[$(ii)$] \emph{Surfaces}: Improved bounds can be obtained if $\XC= \bigcup_{i = 1}^{I}g_i(\UC_i)$ for $\CC^2$ maps $g_i\,\colon \UC_i \rightarrow \RR^d$ with bounded second derivatives on open, bounded, and convex sets $\UC_i\subset \RR^s$ for $s < d$. In this setting, the partially evaluated cost $c(g_i(\cdot), y)$ is Lipschitz and semi-concave on $\UC_i$ uniformly in $y\in\YC$ for bounded $\YC$. Hence, bound \eqref{eq:SemiconcaveBound} holds with $d$ replaced by $s$.

		\item[$(iii)$] \emph{Manifolds}: The setting described in $(ii)$ includes compact $s$-dimensional immersed $\CC^2$ submanifolds of $\RR^d$ \citep{lee2013smooth}. Compactness ensures the existence of an atlas with finitely many charts such that all involved maps have bounded second derivative and all co-domains of charts are convex and bounded. 
	\end{enumerate}
\end{example}

We continue with a setting in which lower bounds for the empirical OT cost that match the upper bounds in \Cref{thm:SemiconcaveOT} can be derived via \Cref{prop:LCAadditiveCosts}. 

\begin{example}[Lower bounds under semi-concave costs]\label{ex:LowerBoundsSemiconcavecosts}
  Let $1 \le d_1 \le d_2$ and consider the unit cubes $\XC = [0, 1]^{d_1}$ and $\YC = [0, 1]^{d_2}$. If $\XC$ is embedded into $\YC$ along the first $d_1$ coordinates, the squared Euclidean cost function amounts to
\begin{equation*}
  c(x, y) = \|x - y_1\|^2 + \|y_2\|^2 \eqqcolon c_1(x, y_1) + c_2(y_2),
\end{equation*}
where $y = (y_1, y_2) \in [0, 1]^{d_1 + (d_2-d_1)}$. Up to scaling of the cost function, this setting satisfies \Cref{as:SemiconcaveCosts}. For independent empirical measures $\hat\mu_n$ and $\hat\nu_n$, one can show that
	\begin{equation}\label{eq:LowerBoundSemiconcave}
	\sup_{\mu, \nu}\,\EV{\left|\OT_{c}(\hat\mu_n, \hat \nu_n) -\OT_{c}(\mu, \nu) \right|}\gtrsim \begin{cases}
    n^{-1/2} & \text{ if } d_1 \leq 4,\\
    n^{-2/d_1} & \text{ if } d_1 \geq 5,
	 \end{cases}
	\end{equation}
  where $\mu\in\PC(\XC)$ and $\nu\in\PC(\YC)$ in the supremum.
	For $d_1 \leq 4$, this lower bound follows by selecting suitable discrete measures (see \citealt[Section 5]{sommerfeld19FastProb}), and for $d_1 \geq 5$, it follows from inequality \eqref{eq:LowerBoundTool1} in conjunction with Proposition~21 in \citet{Manole21} (which is also applicable for more general strictly convex cost functions). 
	Overall, the upper bounds in \eqref{eq:SemiconcaveBound} match the lower bounds \eqref{eq:LowerBoundSemiconcave} in case $d_1 \neq 4$ and are sharp up to logarithmic factors for $d_1 = 4$.
\end{example}

We want to highlight that the fast rates of \Cref{thm:SemiconcaveOT} (compared to \Cref{thm:LipschitzOT}) can often be expected in the settings of \Cref{ex:SemiconcaveExamples} even if the cost function is not $\CC^2$ on all of $\RR^{2d}$. For example, if the set 
\begin{equation*}
  \DC \coloneqq\left\{(x,y)\,\big|\, c \text{ is not $\CC^2$ at $(x,y)$}\right\} \subset \RR^{2d}
\end{equation*}
is strictly separated from $\Sigma \coloneqq\supp(\mu) \times \supp(\nu) \subset\RR^{2d}$, one can extend the restriction $c|_{\Sigma}$ to a $\CC^2$ cost function on all of $\RR^{2d}$ (by the extension theorem of \citealt{whitney1934analytic}) without altering $\OTn$ or $\OT_c(\mu, \nu)$. If $c(x,y) = \norm{x-y}^p$ for any $0 < p < 2$, we find $\DC = \{(x, x)\,|\,x\in\RR^d\}$, implying the fast convergence rates in \eqref{eq:SemiconcaveBound} whenever the supports of $\mu$ and $\nu$ are strictly separated. Similar observations were pointed out by \citet[Corollary~3(ii)]{Manole21} under additional convexity assumptions. In contrast, considering the $l_p^p$-cost function $c(x, y) = \sum_{i = 1}^{d}|x_i - y_i|^p$ for $0 < p < 2$, the set 
\begin{equation}\label{eq:DifferentiabilityFailureL1}
  \DC = \left\{(x,y) \,\big|\, x_i = y_i \text{ for some } i\in \{1,\dots, d\} \right\} \subset \RR^{2d}
\end{equation}
is notably larger. Therefore, the $p$-Wasserstein distance based on the Euclidean norm may exhibit faster convergence rates than its $l_p$ counterpart, e.g., if $\mu$ is a translation of $\nu$ along a coordinate axis.

We conjecture that the faster rates implied by \Cref{thm:SemiconcaveOT} will occur under even more general circumstances. For example, in some settings it might suffice that $\DC$ is negligible under the actual optimal transport, meaning that $\pi(\DC) = 0$ for an OT plan between $\mu$ and $\nu$. A proof of this claim, however, would likely require quantitative statements on the regularity of the cost function along the support of empirical OT plans and lies beyond the scope of this work.
Still, we pick up on this hypothesis and observe some numerical evidence in \Cref{sec:Simulations}.

\subsection{Optimal Transport under H\"older Costs}
\label{ssec:holder}

In \Cref{ssec:lipschitz} and \ref{ssec:semi}, we have shown that the rate of convergence of the empirical OT cost in $\RR^d$ is bounded by $n^{-1/d}$ for Lipschitz and $n^{-2/d}$ for $\CC^2$ costs (if $d \ge 5$).
The recent work of \cite{Manole21} demonstrated that these results can be generalized to $\alpha$-H\"older smooth costs for $0 < \alpha \le 2$, deriving the rates $n^{-\alpha/d}$.
In this section, we employ similar arguments to bound the uniform metric entropy of the class $\FCot$ by $\epsilon^{-d/\alpha}$ (for any $d\in\NN$ and $0 < \alpha \le 2$) in settings resembling the ones of \Cref{as:LipschitzCosts} and \ref{as:SemiconcaveCosts}.

We say that a function $f\colon U\to \RR$ on a convex domain $U\subset\RR^d$ is $(\alpha,\Lambda)$-H\"older smooth for $0 < \alpha \le 1$ and $\Lambda > 0$ if $\|f\|_\infty < \Lambda$ and
\begin{equation*}
  |f(x) - f(y)| \le \Lambda\cdot\|x - y\|^\alpha.
\end{equation*}
Moreover, we say that $f$ is $(\alpha, \Lambda)$-H\"older smooth for $1 < \alpha \le 2$ if $\|f\|_\infty < \Lambda$ and $f$ is differentiable with $(\alpha - 1, \Lambda)$-H\"older smooth partial derivatives. If the convex domain $U$ in this definition is not open, we assume the existence of a H\"older smooth function on an open subset of $\RR^d$ containing $U$ that coincides with $f$ on $U$.

\begin{assumption}\label{as:HolderCosts}
Let $\alpha\in(0, 2]$ and suppose that $\XC= \bigcup_{i = 1}^{I} g_i(\UC_i)$ for $I\in\NN$ compact, convex subsets $\UC_i \subset \RR^d$ and maps $g_i\,\colon\UC_i \to \XC$ so that $c(g_i(\cdot), y)$ is $(\alpha, 1)$-H\"older for all $y\in\YC$.
\end{assumption}

\begin{theorem}[H\"older LCA]\label{thm:HolderOT}
  Let $\XC$ and $\YC$ be Polish spaces and $c \,\colon \XC\times \YC \rightarrow [0,1]$ be continuous. If \Cref{as:HolderCosts} holds, then, for any $\mu \in \PC(\XC)$ and $\nu \in \PC(\YC)$, the empirical estimator $\OTn$ from \eqref{eq:EmpiricalEstimator} satisfies  
 \begin{equation}\label{eq:HolderBound}
   \EV{\big|\OTn  - \OT_c(\mu, \nu)\big|} \lesssim \begin{cases}
 	 n^{-1/2} & \text{ if } d < 2\alpha,\\
 	 n^{-1/2}\log(n) & \text{ if } d = 2\alpha,\\
 	 n^{-\alpha/d} & \text{ if } d > 2\alpha.
   \end{cases}
  \end{equation}
\end{theorem} 
\begin{proof}
  \Cref{lem:HoelderMetricEntropy} in \Cref{sec:OmittedProofs} shows that the uniform metric entropy in this setting is bounded by $\log\NC(\epsilon, \FCot, \|\cdot\|_\infty) \lesssim \epsilon^{-d/\alpha}$. An application of \Cref{thm:AbstractUpperBound} yields the claim.
\end{proof}

\Cref{thm:HolderOT} can be used to derive upper bounds for H\"older smooth cost functions on Euclidean spaces analogous to \Cref{ex:SemiconcaveExamples}. Moreover, it is again possible to derive lower bounds in certain situations. For instance, in the setting of \Cref{ex:LowerBoundsSemiconcavecosts}, we can consider $\alpha$-H\"older costs of the form $c(x,y) = \sum_{i = 1}^{d_1}|x_i - y_{1i}|^\alpha + \sum_{i=d_1+1}^{d_2} |y_{2i}|^\alpha$ for $\alpha \in (0,2]$. Then, \Cref{as:HolderCosts} is fulfilled (for $d = d_1$) and one can show that 
\begin{equation*}
	\sup_{\mu, \nu}\,\EV{|T_{c}(\hat\mu_n, \hat\nu_n) -  	T_{c}(\mu, \nu)|} \gtrsim \begin{cases}n^{-1/2} &\text{ if } d_1 \leq 2 \alpha,\\ n^{-\alpha/d_1} &\text{ if } d_1 > 2 \alpha,	
 \end{cases}
\end{equation*}
where the supremum is taken over $\mu\in\PC(\XC)$ and $\nu\in\PC(\YC)$.
Herein, the lower bound for $d_1 \leq 2 \alpha$ follows by selecting discretely supported measures, whereas the regime $d_1 > 2 \alpha$ is covered by Proposition 21 in \cite{Manole21}. In particular, in case of $d_1\neq 2 \alpha$, the upper bound from \Cref{thm:HolderOT} matches the lower bound, whereas for $d_1= 2\alpha$ it is sharp only up to a logarithmic factor. 

\section{Simulations}
\label{sec:Simulations}

In the previous sections, we investigated the convergence rate for the empirical OT cost in various settings, stressing than an intrinsic adaptation to the less complex measure governs asymptotic statistical properties. 
We now turn to the question if these asymptotic properties can already be observed in the finite sample regime accessible to numerical analysis.
For this purpose, we fix probability measures $\mu$ and $\nu$ and approximate the mean absolute deviation
\begin{equation*}
  \Delta_n = \EV{\big|\OTn - \OT_c(\mu, \nu)\big|}
\end{equation*}
for various values of $n$ via Monte-Carlo simulations with $2000$ independent repetitions.\footnote{We employed the network-simplex based \texttt{C++} solver by \citet{Bonneel2011} for the computation of the empirical OT cost. The full source code used to produce the data in this section can be found under \url{https://gitlab.gwdg.de/staudt1/lca}.}
Since the value of $\OT_c(\mu, \nu)$ has to be known with high accuracy for conclusive results, we are restricted to relatively simple settings where analytical approaches are feasible and the optimal transport cost or map can be computed explicitly. The spaces $\XC$ and $\YC$ are considered to be subsets of $\RR^d$, with either $l_1$ and $l_2^2$ cost functions of the form
\begin{equation*}
  c_1(x, y) = \sum_{i=1}^d |x_i - y_i|
  \qquad\text{or}\qquad
  c_2(x, y) = \sum_{i=1}^d |x_i - y_i|^2.
\end{equation*}
In total, we look at the following choices for $\mu$ and $\nu$. The intrinsic dimension of the former is denoted by $d_1$, and the one of the latter by $d_2$, where $d_1 \le d_2 \le d$. The $r$-dimensional unit-sphere is denoted by $\sphere^r\subset\RR^{r+1}$.
\begin{enumerate}[leftmargin=1.5em]
  \item[$(i)$] \emph{Cube}:
    We choose $\mu = \uniform(\XC)$ for $\XC = [0,1]^{d_1}\times\{0\}^{d_2-d_1}$ and
    $\nu=\uniform(\YC)$ for $\YC = [0, 1]^{d_2}$. As the one-sample
    estimates $\OT_c(\hat\mu_n, \nu)$ and $\OT_c(\mu, \hat\nu_n)$ are
    computationally infeasible, we employ the two-sample estimates $\OTn
    = \OT_c(\hat\mu_n, \hat\nu_n)$ for up to $n = 2^{11}= 2048$. The value of
    $\OT_c(\mu, \nu)$ is calculated analytically.
    
  \item[$(ii)$] \emph{Sphere}:
    We choose $\mu = \uniform(\XC)$ for $\XC = \sphere^{d_1}\times\{0\}^{d_2-d_1}$ 
    and $\nu=\uniform(\YC)$ for $\YC = \sphere^{d_2}$. The two-sample estimate
    $\OTn = \OT_c(\hat\mu_n, \hat\nu_n)$ is used for up to $n = 2^{11}=2048$ and
    $\OT_c(\mu, \nu)$ is approximated numerically (the optimal transport map
    between $\nu$ to $\mu$ can be established due to the symmetry of the
    setting).

  \item[$(iii)$] \emph{Semi-discrete}:
    We choose $\nu=\uniform(\YC)$ for $\YC = [0, 1]^d$ and set
    $\mu=\pf_\#\nu$, where $\pf(y) = \textup{argmin}_{x\in\XC}\, c(x, y)$
    denotes the $c$-projection onto the finite set $\XC = \{x_i\}_{i=1}^I \subset [0,
    1]^d$ with $I \in \NN$.
    Consequently, $\mu(\{x_i\})$ equals the fraction of the volume of $[0,
    1]^d$ that lies closest to $x_i$.
    The positions $x_i$ have been fixed once for each pair $(I, d)$ by drawing
    them uniformly in $[0, 1]^d$. 
    The one-sample estimator $\OTn = \OT_c(\mu, \hat\nu_n)$ is used for up to $n
    = 2^{15} =32768$ and $\OT_c(\mu, \nu)$ is approximated numerically (based on the
    observation that $\pf$ is the optimal transport map between $\nu$ and
    $\mu$).
\end{enumerate}

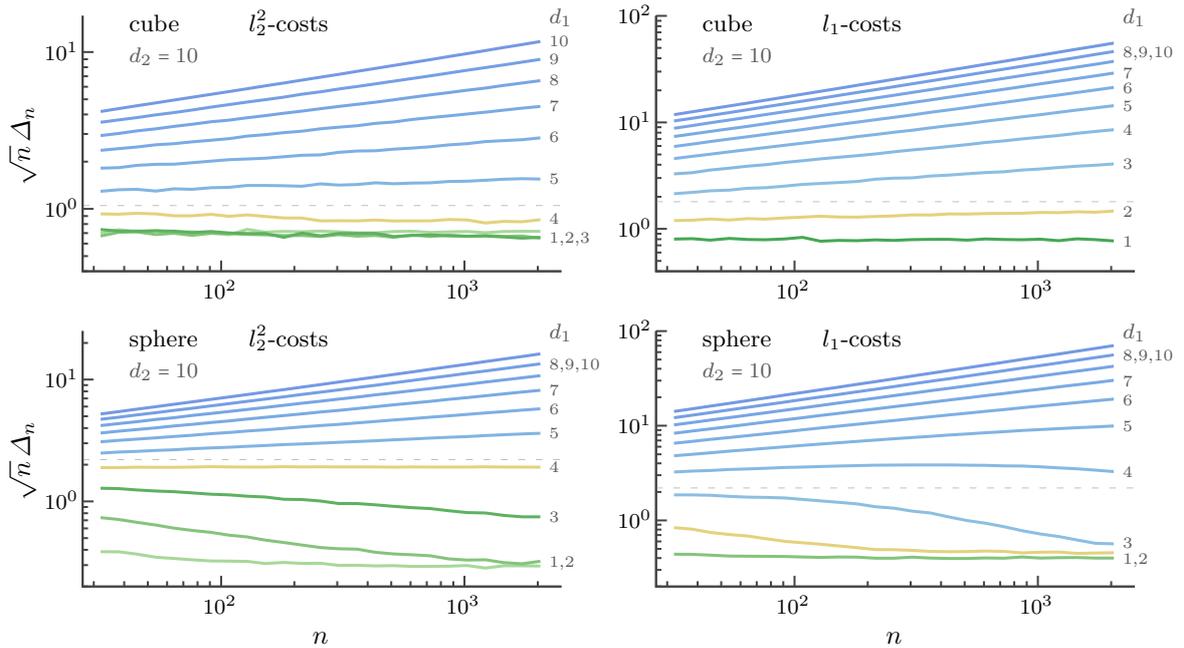
\begin{figure}[t!]
  \small
  \input{figures/tex_smooth-nonsmooth-cubes-l2.tex}
  \hspace{-0.40cm}\input{figures/tex_smooth-nonsmooth-cubes-l1.tex}

  \vspace{-0.30cm}
  \input{figures/tex_smooth-nonsmooth-spheres-l2.tex}
  \hspace{-0.40cm}\input{figures/tex_smooth-nonsmooth-spheres-l1.tex}

  \caption{Simulations of the mean absolute deviation $\Delta_n$ in the \emph{cube}
    and \emph{sphere} settings. The different curves correspond to $1 \le d_1 \le 10$.
    Green lines mark the dimensions $d_1$ for which
    the upper bounds in \Cref{thm:LipschitzOT} and \Cref{thm:SemiconcaveOT} suggest
    $\sqrt{n}\Delta_n$ to be bounded, while yellow and blue lines enjoy no such
    guarantee.
  }
  \label{fig:smooth-nonsmooth}
\end{figure}

A first set of simulation results in the \emph{cube} and \emph{sphere} settings with smooth~$l_2^2$ and non-smooth~$l_1$ costs for fixed $d_2 = 10$ can be seen in \Cref{fig:smooth-nonsmooth}. As anticipated by the LCA principle, the smaller dimension $d_1$ appears to dictate the convergence rate of $\Delta_n$ towards zero as $n$ becomes large. For smooth costs, $\Delta_n$ seems to converge with the rate $1/\sqrt{n}$ for $d_1 \le 3$ (and even in the critical case $d_1 = 4$, which is in line with results by \citealt{ledoux2019OptimalMatchingI}), while the convergence for $d_1 \ge 5$ is perceivably slower in both settings. This is in good agreement with the upper bounds \eqref{eq:SemiconcaveBound} established by \Cref{thm:SemiconcaveOT} for $\CC^2$ cost functions.
For non-smooth $l_1$ costs, in contrast, the \emph{cube} setting again exhibits the behavior to be expected if the bounds \eqref{eq:LipschitzOTBound_Good} of \Cref{thm:LipschitzOT} were sharp (i.e, only $d_1 = 1$ leads to a clear $n^{-1/2}$ convergence), but the results for the \emph{sphere} setting somewhat resemble the ones for smooth costs. This salient difference might be explained along the lines discussed in the context of equation \eqref{eq:DifferentiabilityFailureL1}. In fact, if $\pi$ denotes an optimal transport plan for $\OT_{c_1}(\mu, \nu)$, then it is straightforward to see that the cost function $c_1$  is \emph{not} differentiable $\pi$-almost surely in the \emph{cube} setting (since all optimal movement of mass leaves the first $d_1$ coordinates unchanged), while it \emph{is} differentiable $\pi$-almost surely in the \emph{sphere} setting (almost all mass is moved such that each coordinate changes).

\begin{figure}[t!]
  {\small\hspace{3.9cm}\textbf{(a)}\hspace{7cm}\textbf{(b)}}\\
  \input{figures/tex_second-dimension-cubes-l2.tex}\hspace{-0.40cm}
  \input{figures/tex_rates-cubes-l2.tex}
  \caption{Additional simulations of the mean absolute deviation $\Delta_n$ in the
    \emph{cube} setting.
    \textbf{(a)} shows analogous simulations to \Cref{fig:smooth-nonsmooth} for
    higher dimensions $d_2 = 100$ and $d_2 = 1000$. \textbf{(b)} contrasts $\Delta_n$
    to the power law behavior $n^{-2/d_1}$ (black lines) which is expected
    asymptotically from \Cref{thm:SemiconcaveOT} and \Cref{ex:LowerBoundsSemiconcavecosts}.}
  \label{fig:rates}
\end{figure}
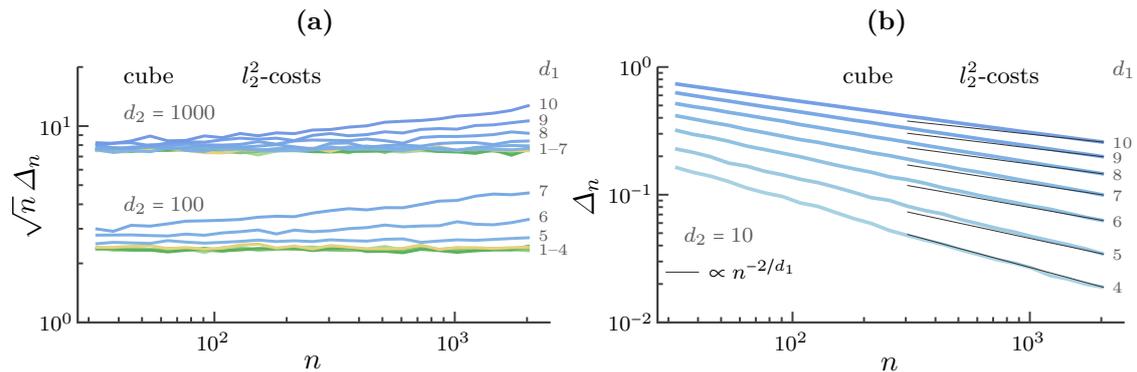

To understand the influence of a different choice of $d_2$, we also conducted
analogous simulations with $d_2 = 100$ and $d_2 = 1000$. While the basic
conclusions remained unchanged and the LCA principle could be confirmed (see
\Cref{fig:rates}a), the statistical fluctuation of the Monte-Carlo estimate of
$\Delta_n$ increased with increasing $d_2$ and larger sample sizes $n$ were
typically needed to observe linearity of $\Delta_n$ in the presented log-log
plots. In this regard, \Cref{fig:rates}b indicates for $d_2 = 10$ in the
\emph{cube} setting that maximal sample sizes of $n = 2^{11} = 2048$ might not
suffice to confidently discern the actual asymptotic convergence rates
$n^{-2/d_1}$ for $d_1 \ge 5$ (see Examples \ref{ex:SemiconcaveExamples} and
\ref{ex:LowerBoundsSemiconcavecosts}).

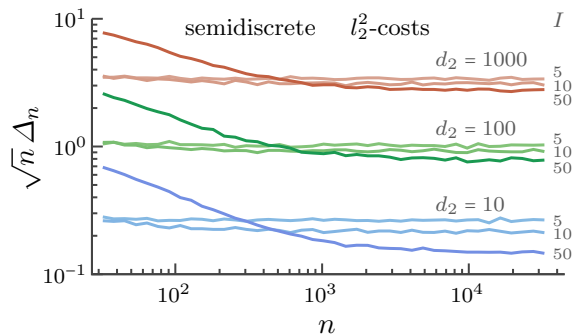
\begin{figure}[t!]
  \centering
  \input{figures/tex_semidiscrete-l2.tex}\hspace{-0.40cm}
  \caption{Simulations of the mean absolute deviation $\Delta_n$ in the
    \emph{semidiscrete} setting. The three curves per value of $d_2$ correspond
    to $I = 5, 10, 50$, respectively. Since the computational burden was
    significantly reduced compared to the \emph{cube} and \emph{sphere}
    settings, the range of $n$ could be increased substantially. Only this
    increased range makes the flattening of the curves for $I = 50$ apparent.
  }
  \label{fig:semidiscrete}
\end{figure}

Finally, we turn to the results obtained in the \emph{semidiscrete} setting.
According to upper bound~\eqref{eq:SemiDiscreteOTBound} established in
\Cref{thm:SemiDiscreteOT}, we anticipate asymptotically a convergence
rate of order  $n^{-1/2}$ for $\Delta_n$ independent of the choice of the cost $c$, the cardinality $I$ of $\XC$, and the dimension $d_2$ of the ambient space.
\Cref{fig:semidiscrete} confirms this expectation under smooth $l_2^2$ cost.
Indeed, in simulations with $I = 50$ (and higher), the convergence of $\Delta_n$
seems to be quicker than $n^{-1/2}$ at first, but eventually slows down for
sufficiently large $n$. Comparable results were observed for $l_1$
costs as well.

\section{Discussion}
\label{sec:Discussion}

In this work, we have established novel statistical guarantees for the empirical OT cost between \emph{different} probability measures, showing that the mean convergence rate is governed by the less complex measure.
In a broader sense, the LCA phenomenon suggests that the curse of dimensionality only affects the estimation of the OT cost when \emph{both} probability measure exhibit high intrinsic dimensions -- an observation with possibly significant repercussions for OT based data analysis applications, as the empirical OT functional automatically adapts to the complexity of the simpler measure and not to the ambient space.
In particular, our theory can also be applied for the popular Wasserstein distance $W_p(\mu, \nu) =\OT_{d^p}(\mu, \nu)^{1/p}$ with $p \ge 1$, since
\begin{equation*}
  \EV{\left|\hat{W}_{p,n} - W_p(\mu, \nu)\right|}\asymp \EV{\left|\hat\OT_{d^p,n} - \OT_{d^p}(\mu, \nu)\right|}
\end{equation*}
for fixed measures $\mu \neq \nu$ on compact metric spaces, where $\hat{W}_{p,n}^p = \hat\OT_{d^p,n}$.
  
Several extensions of our theory seem to be natural targets for future research. First, our arguments crucially rely on \emph{uniform} metric entropy bounds, which essentially restricts our results to bounded costs and spaces. A generalization to measure-dependent notions of metric entropy, or a technique to properly exploit the concentration of measures, might serve to overcome these limitations.
In a similar vein, it might be possible to adapt the LCA principle to more general notions of dimensionality. While our theory already includes general compact Lipschitz and $\CC^2$ surfaces in $\RR^d$, and even metric spaces with finite Minkowski-Bouligand dimension, it is as of yet unclear if general extensions to the Hausdorff dimension \cite[Chapter~4]{mattila1995geometry} or the (concentration-dependent) Wasserstein dimension \citep{weed2019sharp} are viable.

Another interesting problem is to find non-trivial settings where the upper bounds in \Cref{thm:AbstractUpperBound} fail to provide sharp rates. While it is easy to find simple examples where the rates suggested by \Cref{thm:LipschitzOT} and \ref{thm:SemiconcaveOT} are not sharp, the bottleneck for these examples is typically a suboptimal bound for the uniform metric entropy of $\FCot$ when applying \Cref{thm:AbstractUpperBound}.
For instance, additive costs of the form $c(x, y) = c_1(x) + c_2(y)$ always lead to the parametric convergence rate $n^{-1/2}$, which is not necessarily captured by \Cref{thm:LipschitzOT} and \ref{thm:SemiconcaveOT}. Adapting \Cref{thm:AbstractUpperBound} to these specific costs, however, yields the correct rate for non-constant costs.

We finally stress that our proof technique explicitly relies on empirical measure based estimators under i.i.d.\ observations. It would be interesting to analyze whether the same (or even faster) convergence rates can be verified for other estimators or for dependent observations.  In particular, it remains an open question  to what extent estimators leveraging  smoothness properties of the underlying measures, e.g., when $\mu$ and $\nu$ are measures on Riemannian manifolds with sufficiently regular densities with respect to the canonical volume forms, also obey the LCA principle.

\begin{ackno}
S. Hundrieser and T. Staudt gratefully acknowledge support from the DFG RTG 2088 and A. Munk of the DFG CRC 1456. S. Hundrieser and T. Staudt were in part funded by the DFG under Germany's Excellence Strategy - EXC 2067/1- 390729940. S. Hundrieser was also in part funded by the DFG RTG 2088. 
\end{ackno}



\begin{appendix}

\section{Bounds on the Uniform Metric Entropy}\label{sec:OmittedProofs}

In this appendix, we establish various upper bounds for the uniform metric entropy of the function class $\FCot$ defined by equation \eqref{eq:NiceTransforms}. To cover the settings introduced in \Cref{sec:Applications}, the following observation about uniform covering numbers under composition is useful.

\begin{lemma}[Composition bound]\label{lem:CompositionBound}
	Let $g \,\colon \UC\rightarrow \VC$ be a surjective map between sets $\UC$ and $\VC$, and let $\FC$ be a real-valued function class on $\VC$. For any $\epsilon > 0$, the class $\FC\circ g\coloneqq \{ f\circ g \mid f \in \FC\}$ satisfies
  \begin{equation*}
		\NC\big(\epsilon,{\FC}, \norm{\cdot}_{\infty, \VC}\big) \leq \NC\big(\epsilon,{\FC\circ g}, \norm{\cdot}_{\infty,\UC}\big).
	\end{equation*}
\end{lemma}

\begin{proof}
Assume that $N \coloneqq \NC\big(\epsilon,\FC\circ g, \norm{\cdot}_{\infty, \UC}\big)$ is finite, otherwise the inequality is trivial. Let $\smash{\{\tilde f_1, \dots, \tilde f_N\}}$ be an $\epsilon$-covering of $\FC\circ g$ and let
\begin{equation*}
  f_i(v) \coloneqq \sup_{u\in g^{-1}(v)} \tilde{f}_i(u).
\end{equation*}
For any $f \in \FC$, there is an $\tilde f_i$ such that $|f\circ g - \tilde f_i| \leq  \epsilon$ on $\UC$. By definition of $f_i$, this implies $f - f_i \le \epsilon$ and $f - f_i \ge -\epsilon$ on $\VC$, which shows that $\{f_1, \ldots, f_N\}$ is an $\epsilon$-covering of $\FC$.
\end{proof}

We now provide upper bounds on the metric entropy of $\FCot$ under the respective assumptions \ref{as:LipschitzCosts}, \ref{as:SemiconcaveCosts}, and \ref{as:HolderCosts}. Due to the union bound (\Cref{lem:MetricEntropyUnionBound}) in conjunction with the composition bound (\Cref{lem:CompositionBound}), it is in all three cases sufficient to bound 
\begin{equation*}
  \log\NC\big(\epsilon, \FCot\circ g_i, \|\cdot\|_{\infty, \UC_i}\big)
\end{equation*}
for all $i\in\{1, \ldots, I\}$ and sufficiently small $\epsilon > 0$. For
convenience, we suppress the index $i$ in the following proofs and work with
generic $g \coloneqq g_i$ and $\UC \coloneqq \UC_i$, as well as $d\coloneqq d_i$ for \Cref{as:LipschitzCosts}.

\begin{lemma}[Metric entropy under Lipschitz costs]\label{lem:LipschitzMetricEntropy}
	Let $\XC$ and $\YC$ be Polish spaces and let $c\,\colon \XC\times \YC\rightarrow [0,1]$ be continuous so that \Cref{as:LipschitzCosts} is fulfilled. Then, for any $\epsilon>0$, 
  \begin{subequations}\label{eq:MetricEntropyLipschitz}
  \begin{align}\label{eq:MetricEntropyBoundLipschitz_Connected}
    \log \NC(\epsilon,\FCot, \norm{\cdot}_{\infty} )
  &\lesssim \sum_{i = 1}^{I}\left(\NC(\epsilon/4,\UC_i,d_i)+ \log(\epsilon^{-1})\right)
  \intertext{Moreover, without the connectedness assumption on $\UC_i$ in \ref{as:LipschitzCosts}, it holds that}
  \label{eq:MetricEntropyBoundLipschitz_Disconnected}
  \log \NC(\epsilon,\FCot, \norm{\cdot}_{\infty} ) 
  &\lesssim \sum_{i = 1}^{I}\NC(\epsilon/4,\UC_i,d_i)\log(\epsilon^{-1}).
  \end{align}
  \end{subequations}
  The implicit constants in  \eqref{eq:MetricEntropyLipschitz} are universal. 
\end{lemma}

\begin{proof}
  By \Cref{as:LipschitzCosts}, it holds that $c(g(\cdot), y)$ is 1-Lipschitz for each $y\in\YC$. Hence, the class $\FCot\circ g$ is contained in $\mathrm{BL}_1(\UC, d)$, which denotes the 1-Lipschitz functions on $(\UC, d)$ that are absolutely bounded by one \citep[Section 1.2]{santambrogio2015optimal}. For connected $\UC$, their uniform metric entropy is bounded by \citep[Section~9]{Kolmogorov1961}
\begin{subequations}
\begin{equation}\label{eq:LipschitzBound2}
  \NC\big(\epsilon, \mathrm{BL}_{1}(\UC,d), \norm{\cdot}_{\infty, \UC}\big) \leq \left( 2 \left\lceil 2/\epsilon\right\rceil + 1\right)\,2^{\NC(\epsilon/4, \UC, d)},
\end{equation}
while general metric spaces only permit the bound
\begin{equation}\label{eq:LipschitzBound1}
  \NC\big(\epsilon, \mathrm{BL}_{1}(\UC,d), \norm{\cdot}_{\infty, \UC}\big) \leq \left( 2 \left\lceil 2/\epsilon\right\rceil + 1\right)^{\NC(\epsilon/4, \UC, d)}.
\end{equation}
\end{subequations}
This implies claim \eqref{eq:MetricEntropyLipschitz}. 
Note that \eqref{eq:LipschitzBound2} is a variation of equation (238) in \cite{Kolmogorov1961},
which only proves the stated bound for connected subsets of a \emph{centrable} metric space (with some improvements, e.g., $\epsilon/4$ can be relaxed to $s\epsilon/(s+1)$ for $s\in\NN$ at the cost of a possibly worse constant).
However, with minor adaptions, the proof also works without requiring centrability for $\epsilon/4$.
\end{proof}

\begin{lemma}[Metric entropy under semi-concave costs]\label{lem:SemiconcaveMetricEntropy}
Let $\XC$ and $\YC$ be Polish spaces and let $c\,\colon \XC\times \YC\rightarrow [0,1]$ be continuous so that \Cref{as:SemiconcaveCosts} is fulfilled. Then, for $\epsilon>0$ sufficiently small, 
\begin{equation}\label{eq:MetricEntropyBoundSemiConcave}
  \log \NC(\epsilon,\FCot, \norm{\cdot}_{\infty} )
	\lesssim I\epsilon^{-d/2},
\end{equation}
where the implicit constant depends on the sets $\UC_1, \dots, \UC_I\subset\RR^d$. 
\end{lemma}

\begin{proof}
Let $s \coloneqq \dim\big(\mathrm{span}(\UC - u)\big) \le d$ for an arbitrary $u\in \UC$. If $s = 0$, the metric entropy of $\FC_c\circ g$ is bounded as in \eqref{eq:MetricEntropyDiscrete}, so we consider $s \ge 1$. By translation and rotation, we may w.l.o.g.\ assume that $\UC$ is a bounded convex subset of $\RR^{s}$ that contains the origin.
By \Cref{as:SemiconcaveCosts} and the properties of the $c$-transform, any $f\in \FCot$ is absolutely bounded by one, and the composition $f\circ g$ is $1$-Lipschitz and $1$-semi-concave on $\UC$.
Thus, the function $u\mapsto f\circ g(u) - \norm{u}^2$ is concave, $L$-Lipschitz with $L \coloneqq 1+2\,\diam{\UC}$, and absolutely bounded by $1+\diam{\UC}^2$. According to \citet[Theorem 1 and Remark 2(ii)]{dragomirescu1992smallest} there exists a concave extension $\tilde f$ of this function to $\RR^s$ with identical Lipschitz-modulus. 
If $\DC\subset\RR^s$ denotes a bounded closed cube that contains $\UC$, then $\tilde{f}$ is absolutely bounded on $\DC$ by $B \coloneqq 1 + \diam{\UC}^2 + L\,\diam{\DC}$. Denoting
the class of concave functions on $\DC$ that are absolutely bounded by $B$ and $L$-Lipschitz by $\mathrm{C}_{B, L}(\DC)$, we conclude for small $\epsilon > 0$
\begin{equation*}
  \NC\big(\epsilon, {\FCot}\circ g, \norm{\cdot}_{\infty, \UC}\big)
  = \NC\big(\epsilon, {\FCot}\circ g - \norm{\cdot}^2, \norm{\cdot}_{\infty, \UC}\big)
  \leq \NC\big(\epsilon, \mathrm{C}_{B, L}(\DC), \norm{\cdot}_{\infty, \DC}\big)
  \lesssim \epsilon^{-s/2}
  \leq \epsilon^{-d/2},
\end{equation*} 
where we used uniform metric entropy bounds for the class $C_{B, L}(\DC)$ provided in \cite{bronshtein1976varepsilon} and \cite{guntuboyina2012covering}.
The implicit constants depend on $B$, $L$, and $\DC$, which in turn depend on $\UC$. 
\end{proof}

\begin{lemma}[Metric entropy under H\"older costs]\label{lem:HoelderMetricEntropy}
	Let $\XC$ and $\YC$ be Polish spaces and let $c\,\colon \XC\times \YC\rightarrow [0,1]$ be continuous so that \Cref{as:HolderCosts} is fulfilled for some $\alpha\in (0,2]$. Then, for $\epsilon>0$ sufficiently small, 
	\begin{equation}
  \log \NC(\epsilon, \FCot, \norm{\cdot}_\infty)\lesssim I \epsilon^{-d/\alpha},
\end{equation}
where the implicit constant depends on $\alpha$ and the sets $\UC_1, \dots, \UC_I\subset\RR^d$.
\end{lemma}

\begin{proof}
We consider $\alpha \in (0,1]$ first. An $(\alpha, 1)$-H\"older function with respect to the Euclidean norm $\|\cdot\|$ is a $1$-Lipschitz function with respect to the metric induced by $\|\cdot\|^\alpha$. It follows by \Cref{as:HolderCosts} that $\FCot \circ g\subseteq \BL{\UC, \|\cdot\|^\alpha}$ (see the proof of \Cref{lem:LipschitzMetricEntropy}).
Furthermore, each function in $\BL{\UC, \|\cdot\|^\alpha}$ can be extended to an element
in $\BL{\DC, \|\cdot\|^\alpha}$, where $\DC\subset\RR^d$ is bounded, connected, and contains $\UC$ \citep[Corollary~2]{mcshane1934extension}. Thus,
noting that $\NC(\epsilon, \DC, \|\cdot\|^\alpha) = \NC(\epsilon^{1/\alpha}, \DC, \|\cdot\|) \lesssim \epsilon^{-d/\alpha}$ and employing \eqref{eq:LipschitzBound2}, we find
\begin{equation*}
  \log\NC(\epsilon, \FCot \circ g, \norm{\cdot}_{\infty, \UC})
  \leq
  \log\NC\big(\epsilon, \BL{\DC, \|\cdot\|^\alpha}, \norm{\cdot}_{\infty, \DC}\big)
  \lesssim
  \epsilon^{-d/\alpha}.
\end{equation*}

For $\alpha\in (1,2]$, we apply \Cref{lem:ExistenceSmoothFunctions} to $c(g(\cdot), y)$ for each $y\in\YC$ separately to define a collection of smoothed, approximated cost functions $c_{\sigma}\colon \DC\times \YC \to \RR$ for $\sigma\in(0, 1]$, where $\DC\subseteq \RR^d$ contains $\UC$ and is convex, open, and bounded. Furthermore, there is $K > 0$ so that the functions $c_\sigma$ satisfy, for all $y\in\YC$,
\begin{equation}\label{eq:NiceCosts} 
 \norm{c(g(\cdot),y) - c_{\sigma}(\cdot,y)}_{\infty, \UC}\leq K \sigma^{\alpha}
 \quad\text{and}\quad
 \norm{c_{\sigma}(\cdot,y)}_{\CC^2(\DC)}\leq K \sigma^{\alpha-2}\eqqcolon \Gamma_\sigma,
\end{equation}
where the $\CC^2(\DC)$-norm of a twice continuously differentiable function $f\colon\DC \rightarrow \RR$ is
\begin{equation*}
\norm{f}_{\CC^{2}(\DC)} \coloneqq \max_{|\beta| \le 2} \norm{ D^\beta f}_{\infty, \DC}, \quad \text{where} \quad D^\beta f = \partial^{|\beta|}f / \partial x_1^{\beta_1} \cdots \partial x_d^{\beta_d} \text{ for } \beta \in \NN^d_0.
\end{equation*}
Note that a function with $\|f\|_{\CC^2(\DC)} \le \Gamma$ for $\Gamma > 0$ is absolutely bounded by $\Gamma$, $\Gamma$-Lipschitz, and $d\Gamma$-semi-concave (since the Eigenvalues of its Hessian are bounded by $d\cdot\Gamma$). 
For each $f \in \FCot$, we define
$f_\sigma\colon \DC\rightarrow \RR, u \mapsto \inf_{y\in \YC} c_\sigma(u,y) - f^{c}(y)$. Due to $f = f^{cc}$ \citep[Proposition~1.34]{santambrogio2015optimal} combined with the first inequality in \eqref{eq:NiceCosts}, we conclude $|f\circ g - f_\sigma|\leq K \sigma^{\alpha}$ on $\UC$. For $\sigma(\epsilon)\coloneqq (\epsilon/2K)^{1/\alpha}$, this implies $|f\circ g - f_{\sigma(\epsilon)}| \le \epsilon/2$ on $\UC$. Consequently, defining $\FC_{c,\sigma}\coloneqq \{f_\sigma | f \in \FCot\}$, 
\begin{equation*}
  \NC\big(\epsilon, \FC\circ g, \norm{\cdot}_{\infty, \UC}\big) \leq \NC\big(\epsilon/2, \FC_{c,\sigma(\epsilon)}, \norm{\cdot}_{\infty, \UC}\big) \leq \NC\big(\epsilon/2, \FC_{c,\sigma(\epsilon)}, \norm{\cdot}_{\infty, \DC}\big).
\end{equation*}
Since the functions $c_{\sigma}(\cdot,y) / d\Gamma_\sigma$ are bounded by one, $1$-Lipschitz, and $1$-semi-concave, we can apply the
metric entropy bounds derived in the proof of \Cref{lem:SemiconcaveMetricEntropy} to conclude
\begin{equation*}
 \log\NC\left(\frac{\epsilon}{2}, \FC_{\sigma(\epsilon)}, \norm{\cdot}_{\infty, \DC}\right) = \log \NC\left(\frac{\epsilon}{2d\Gamma_{\sigma(\epsilon)}}, \frac{\FC_{\sigma(\epsilon)}}{d\Gamma_{\sigma(\epsilon)}}, \norm{\cdot}_{\infty, \DC}\right) \lesssim \left(\frac{\epsilon}{\Gamma_{\sigma(\epsilon)}}\right)^{-d/2} \asymp \epsilon^{-d/\alpha},
\end{equation*}
	where the constants depend on $\alpha$ and  $\DC$, which in turn depends on $\UC$.
\end{proof}

\begin{lemma}\label{lem:ExistenceSmoothFunctions}
  Let $\DC\subset \RR^d$ be bounded, convex, and open, and let $\UC\subset\DC$ be a compact and convex subset. 
  Then, there exists $K>0$ such that for any $(\alpha, 1)$-H\"older function $h$ on $\UC$ with $1 < \alpha \le 2$ there is a collection of smooth functions  $h_\sigma\colon\DC\to\RR$ such that 
\begin{equation}\label{eq:NiceProperties}
		\norm{h-h_\sigma}_{\infty, \UC} \leq K\sigma^\alpha
    \qquad
    \text{and}
    \qquad
    \norm{h_\sigma}_{\CC^2(\DC)}\leq K \sigma^{\alpha-2}
    \qquad
    \text{for }
    \sigma\in (0,1].
	\end{equation}
\end{lemma}

\begin{proof}
  Recall the definition of $(\alpha, \Lambda)$-H\"older smooth functions for $1 < \alpha \le 2$ and $\Lambda > 0$ from \Cref{ssec:holder}, and let $u, u_0 \in \UC$. If we denote $z \coloneqq u - u_0$, then the mean value theorem asserts the existence of $t\in[0, 1]$ such that $h(u) = h(u_0) + \langle\nabla h(u_0 + tz), z\rangle$. This implies
\begin{equation*}
  h(u) = h(u_0) + \langle\nabla h(u_0), z\rangle + R_{u_0}(u),
  \quad\text{where}\quad
  R_{u_0}(u) = \langle\nabla h(u_0 + t z) - \nabla h(u_0), z\rangle.
\end{equation*}
Due to the $(\alpha - 1, 1)$-H\"older smoothness of the partial derivatives of $h$,
we find
\begin{equation}\label{eq:HolderResidual}
  |R_{u_0}(u)| \le \| \nabla h(u_0 + tz) - \nabla h(u_0)\| \|z\| \le \sqrt{d}\,\|u - u_0\|^\alpha.
\end{equation}
This shows that the function $h$ is an element of the class $\mathrm{Lip}(\alpha, \UC)$ defined in \citet[Chapter VI, Section 3]{stein1971singular}. By Theorem~4 in the same reference, the function $h$ admits an extension $\tilde h$ to $\RR^d$ that is $(\alpha, K')$-H\"older on $\RR^d$ for some $K' > 0$ (which is independent of $\UC$ and $h$). For an even and smooth mollifier $M \,\colon \RR^d\rightarrow [0,\infty)$ supported on the unit ball $B_1$, we define $M_\sigma \coloneqq \sigma^{-d} M(\,\cdot\,/ \sigma)$, which is supported on the ball with radius $\sigma\in(0, 1]$, and set
\begin{equation*}
  h_\sigma\colon  \DC  \rightarrow \RR, \quad  u \mapsto \big(\tilde h \ast M_\sigma\big)(u) = \int \tilde h(u-z) M_\sigma(z)\,\dif z,
\end{equation*}
  where integration is over $\RR^d$ (i.e., effectively over the support of $M_\sigma$).
  The desired properties \eqref{eq:NiceProperties} now follow analogously to the proof of Lemma~8 of \cite{Manole21}. 
  For completeness, we include the arguments here.
  We first observe 
  \begin{equation*}
  	\tilde h(u)  = \tilde h(u_0) + \langle \nabla \tilde h(u_0), u-u_0\rangle + \tilde R_{u_0}(u),
    \quad\text{where}\quad |\tilde R_{u_0}(u)| \le \sqrt{d}K'\|u - u_0\|^\alpha,
  \end{equation*}
  for any $u,u_0\in\DC$, which can be derived analogously to \eqref{eq:HolderResidual}. 
  For the first bound in \eqref{eq:NiceProperties}, we note that
  $M$ is even, implying $\int z_i\,M_\sigma(z)\,\dif z = 0$ for all $1 \le i \le d$. By expanding $\tilde h(u - z)$ around $u\in\UC$ for $\|z\| \le \sigma$, it follows that
		\begin{align*}
			\left|h_\sigma(u) - h(u)\right|
      &= \left|\int \left(\tilde h(u-z) - \tilde h(u)\right) M_\sigma(z)\,\dif z\right|\\
			& \leq \int \left|\tilde R_u(u-z)\right| M_\sigma(z)\,\dif z \\
      & \leq \sqrt{d}K' \sigma^\alpha.
		\end{align*}		
    For the second inequality in \eqref{eq:NiceProperties}, we fix some $u_0\in\DC$ and observe for any $u \in \DC$ that
    \begin{align*}
      h_\sigma(u)
      &= \int \tilde{h}(u - z) M_\sigma(z)\,\dif z \\
      &= \int \big(\tilde{h}(u_0) + \langle \nabla\tilde h(u_0), u - z - u_0\rangle + \tilde R_{u_0}(u-z)\big)\, M_\sigma(z)\,\dif z \\
      &\eqqcolon A_1 + \langle A_2, u \rangle + A_3(u),
    \end{align*}
    where $A_1\in\RR$, $A_2\in\RR^d$, and $A_3(u) = \int \tilde R_{u_0}(z) M_\sigma(u-z)\,\dif z$ (after a change of variables).
    For $\beta\in\NN_0^d$ with $|\beta| = 2$, we evaluate $D^\beta h_\sigma$ at $u_0$. Exchanging differentiation and integration in the first inequality, and employing substitution in the final one, we observe
    \begin{align*}
      \big|D^\beta h_\sigma(u_0)\big|
      =
      \big|D^\beta A_3(u_0)\big|
      &\le
      \sigma^{-d-2} \int \big|\tilde R_{u_0}(z)\big| \left|D^\beta M\left(\frac{u_0 - z}{\sigma}\right)\right|\,\dif z \\
      &\le
      \sqrt{d}K'\sigma^{-d-2} \int \|u_0 - z\|^\alpha \left|D^\beta M\left(\frac{u_0 - z}{\sigma}\right)\right|\,\dif z  \\
      &\le
      \sqrt{d}K'\sigma^{\alpha-2} \int \|z\|^\alpha |D^\beta M(z)|\,\dif z \\
      &= K'' \sigma^{\alpha - 2}
    \end{align*}
    for some $0 < K'' < \infty$. Since this holds for any $u_0\in\DC$ with $K''$ independent of $u_0$ and $\sigma$, we conclude $\|D^\beta h_\sigma\|_{\infty, \DC} \le K'' \sigma^{\alpha-2}$. Analogous inequalities for $|\beta| < 2$ follow from the fact that $\DC$ is convex and bounded, so $D^\beta h_\sigma$ can be bounded in terms of the second derivatives of $h_\sigma$ and the diameter of $\DC$.
\end{proof}

\end{appendix}

\end{document}

%% file: figures/tex_cartesian.pdf_tex
\begingroup%
  \makeatletter%
  \providecommand\color[2][]{%
    \errmessage{(Inkscape) Color is used for the text in Inkscape, but the package 'color.sty' is not loaded}%
    \renewcommand\color[2][]{}%
  }%
  \providecommand\transparent[1]{%
    \errmessage{(Inkscape) Transparency is used (non-zero) for the text in Inkscape, but the package 'transparent.sty' is not loaded}%
    \renewcommand\transparent[1]{}%
  }%
  \providecommand\rotatebox[2]{#2}%
  \newcommand*\fsize{\dimexpr\f@size pt\relax}%
  \newcommand*\lineheight[1]{\fontsize{\fsize}{#1\fsize}\selectfont}%
  \ifx\svgwidth\undefined%
    \setlength{\unitlength}{360.02591596bp}%
    \ifx\svgscale\undefined%
      \relax%
    \else%
      \setlength{\unitlength}{\unitlength * \real{\svgscale}}%
    \fi%
  \else%
    \setlength{\unitlength}{\svgwidth}%
  \fi%
  \global\let\svgwidth\undefined%
  \global\let\svgscale\undefined%
  \makeatother%
  \begin{picture}(1,0.3279717)%
    \lineheight{1}%
    \setlength\tabcolsep{0pt}%
    \put(0,0){\includegraphics[width=\unitlength,page=1]{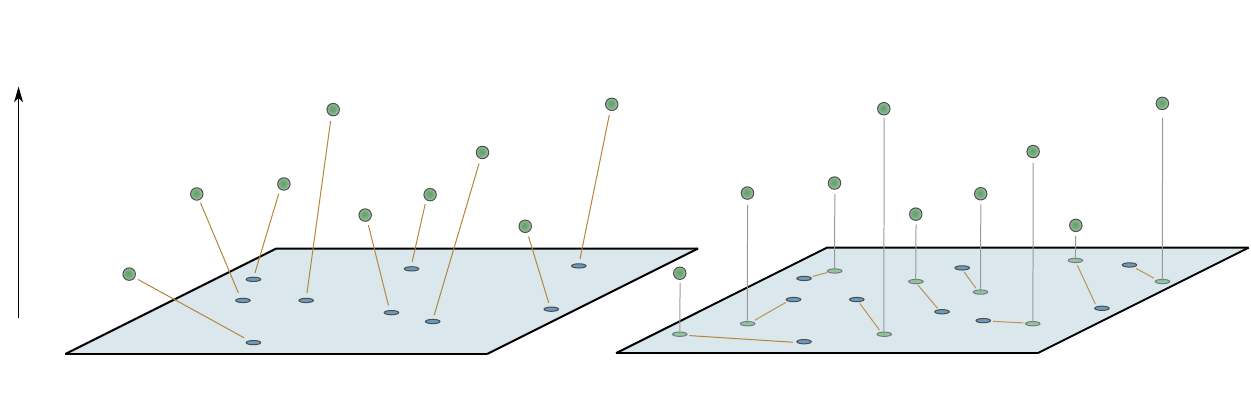}}%
    \put(0.31275234,0.30957404){\color[rgb]{0,0,0}\makebox(0,0)[lt]{\lineheight{1.25}\smash{\begin{tabular}[t]{l}\textbf{(a)}\end{tabular}}}}%
    \put(0.7685988,0.30957404){\color[rgb]{0,0,0}\makebox(0,0)[lt]{\lineheight{1.25}\smash{\begin{tabular}[t]{l}\textbf{(b)}\end{tabular}}}}%
    \put(0.44477341,0.00494815){\color[rgb]{0,0,0}\makebox(0,0)[t]{\lineheight{1.25}\smash{\begin{tabular}[t]{c}$\XC = [0, 1]^2\times\{0\}$\end{tabular}}}}%
    \put(0.0358321,0.26036425){\color[rgb]{0,0,0}\makebox(0,0)[t]{\lineheight{1.25}\smash{\begin{tabular}[t]{c}$\RR$\end{tabular}}}}%
  \end{picture}%
\endgroup%

%% file: figures/tex_smooth-nonsmooth-cubes-l2.tex
\begingroup
  \makeatletter
  \providecommand\color[2][]{%
    \GenericError{(gnuplot) \space\space\space\@spaces}{%
      Package color not loaded in conjunction with
      terminal option `colourtext'%
    }{See the gnuplot documentation for explanation.%
    }{Either use 'blacktext' in gnuplot or load the package
      color.sty in LaTeX.}%
    \renewcommand\color[2][]{}%
  }%
  \providecommand\includegraphics[2][]{%
    \GenericError{(gnuplot) \space\space\space\@spaces}{%
      Package graphicx or graphics not loaded%
    }{See the gnuplot documentation for explanation.%
    }{The gnuplot epslatex terminal needs graphicx.sty or graphics.sty.}%
    \renewcommand\includegraphics[2][]{}%
  }%
  \providecommand\rotatebox[2]{#2}%
  \@ifundefined{ifGPcolor}{%
    \newif\ifGPcolor
    \GPcolortrue
  }{}%
  \@ifundefined{ifGPblacktext}{%
    \newif\ifGPblacktext
    \GPblacktexttrue
  }{}%
  \let\gplgaddtomacro\g@addto@macro
  \gdef\gplbacktext{}%
  \gdef\gplfronttext{}%
  \makeatother
  \ifGPblacktext
    \def\colorrgb#1{}%
    \def\colorgray#1{}%
  \else
    \ifGPcolor
      \def\colorrgb#1{\color[rgb]{#1}}%
      \def\colorgray#1{\color[gray]{#1}}%
      \expandafter\def\csname LTw\endcsname{\color{white}}%
      \expandafter\def\csname LTb\endcsname{\color{black}}%
      \expandafter\def\csname LTa\endcsname{\color{black}}%
      \expandafter\def\csname LT0\endcsname{\color[rgb]{1,0,0}}%
      \expandafter\def\csname LT1\endcsname{\color[rgb]{0,1,0}}%
      \expandafter\def\csname LT2\endcsname{\color[rgb]{0,0,1}}%
      \expandafter\def\csname LT3\endcsname{\color[rgb]{1,0,1}}%
      \expandafter\def\csname LT4\endcsname{\color[rgb]{0,1,1}}%
      \expandafter\def\csname LT5\endcsname{\color[rgb]{1,1,0}}%
      \expandafter\def\csname LT6\endcsname{\color[rgb]{0,0,0}}%
      \expandafter\def\csname LT7\endcsname{\color[rgb]{1,0.3,0}}%
      \expandafter\def\csname LT8\endcsname{\color[rgb]{0.5,0.5,0.5}}%
    \else
      \def\colorrgb#1{\color{black}}%
      \def\colorgray#1{\color[gray]{#1}}%
      \expandafter\def\csname LTw\endcsname{\color{white}}%
      \expandafter\def\csname LTb\endcsname{\color{black}}%
      \expandafter\def\csname LTa\endcsname{\color{black}}%
      \expandafter\def\csname LT0\endcsname{\color{black}}%
      \expandafter\def\csname LT1\endcsname{\color{black}}%
      \expandafter\def\csname LT2\endcsname{\color{black}}%
      \expandafter\def\csname LT3\endcsname{\color{black}}%
      \expandafter\def\csname LT4\endcsname{\color{black}}%
      \expandafter\def\csname LT5\endcsname{\color{black}}%
      \expandafter\def\csname LT6\endcsname{\color{black}}%
      \expandafter\def\csname LT7\endcsname{\color{black}}%
      \expandafter\def\csname LT8\endcsname{\color{black}}%
    \fi
  \fi
    \setlength{\unitlength}{0.0500bp}%
    \ifx\gptboxheight\undefined%
      \newlength{\gptboxheight}%
      \newlength{\gptboxwidth}%
      \newsavebox{\gptboxtext}%
    \fi%
    \setlength{\fboxrule}{0.5pt}%
    \setlength{\fboxsep}{1pt}%
    \definecolor{tbcol}{rgb}{1,1,1}%
\begin{picture}(4460.00,2520.00)%
    \gplgaddtomacro\gplbacktext{%
      \colorrgb{0.27,0.27,0.27}
      \put(509,844){\makebox(0,0)[r]{\strut{}\scriptsize $10^{0}$}}%
      \colorrgb{0.27,0.27,0.27}
      \put(509,2023){\makebox(0,0)[r]{\strut{}\scriptsize $10^{1}$}}%
      \colorrgb{0.27,0.27,0.27}
      \put(1565,228){\makebox(0,0){\strut{}\scriptsize $10^{2}$}}%
      \colorrgb{0.27,0.27,0.27}
      \put(3380,228){\makebox(0,0){\strut{}\scriptsize $10^{3}$}}%
      \csname LTb\endcsname
      \put(889,2237){\makebox(0,0)[l]{\strut{}\footnotesize cube}}%
      \csname LTb\endcsname
      \put(1782,2237){\makebox(0,0)[l]{\strut{}\footnotesize $l_2^2$-costs}}%
      \csname LTb\endcsname
      \put(4024,2295){\makebox(0,0)[l]{\strut{}\color{DarkGray}\scriptsize $d_1$}}%
      \csname LTb\endcsname
      \put(889,2007){\makebox(0,0)[l]{\strut{}\color{DarkGray}\scriptsize $d_2 = 10$}}%
      \csname LTb\endcsname
      \put(4024,649){\makebox(0,0)[l]{\strut{}\color{DarkGray}\tiny 1,2,3}}%
      \csname LTb\endcsname
      \put(4024,788){\makebox(0,0)[l]{\strut{}\color{DarkGray}\tiny 4}}%
      \csname LTb\endcsname
      \put(4024,1093){\makebox(0,0)[l]{\strut{}\color{DarkGray}\tiny 5}}%
      \csname LTb\endcsname
      \put(4024,1403){\makebox(0,0)[l]{\strut{}\color{DarkGray}\tiny 6}}%
      \csname LTb\endcsname
      \put(4024,1639){\makebox(0,0)[l]{\strut{}\color{DarkGray}\tiny 7}}%
      \csname LTb\endcsname
      \put(4024,1833){\makebox(0,0)[l]{\strut{}\color{DarkGray}\tiny 8}}%
      \csname LTb\endcsname
      \put(4024,1993){\makebox(0,0)[l]{\strut{}\color{DarkGray}\tiny 9}}%
      \csname LTb\endcsname
      \put(4024,2127){\makebox(0,0)[l]{\strut{}\color{DarkGray}\tiny 10}}%
    }%
    \gplgaddtomacro\gplfronttext{%
      \csname LTb\endcsname
      \put(102,1335){\rotatebox{-270}{\makebox(0,0){\strut{}$\sqrt{n}\,\Delta_n$}}}%
    }%
    \gplbacktext
    \put(0,0){\includegraphics[width={223.00bp},height={126.00bp}]{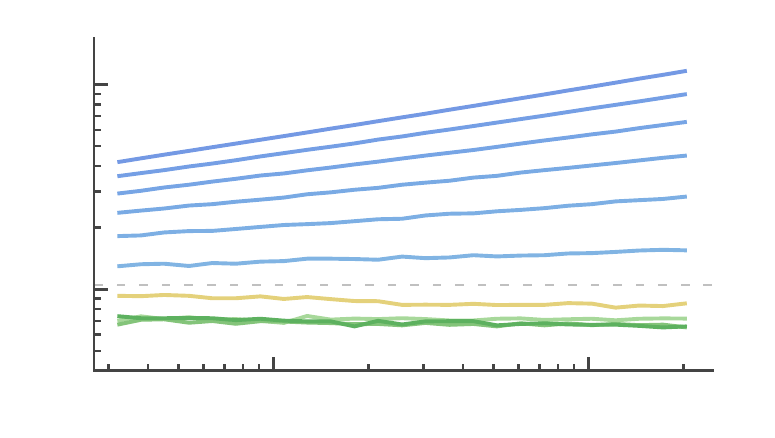}}%
    \gplfronttext
  \end{picture}%
\endgroup

%% file: figures/tex_smooth-nonsmooth-cubes-l1.tex
\begingroup
  \makeatletter
  \providecommand\color[2][]{%
    \GenericError{(gnuplot) \space\space\space\@spaces}{%
      Package color not loaded in conjunction with
      terminal option `colourtext'%
    }{See the gnuplot documentation for explanation.%
    }{Either use 'blacktext' in gnuplot or load the package
      color.sty in LaTeX.}%
    \renewcommand\color[2][]{}%
  }%
  \providecommand\includegraphics[2][]{%
    \GenericError{(gnuplot) \space\space\space\@spaces}{%
      Package graphicx or graphics not loaded%
    }{See the gnuplot documentation for explanation.%
    }{The gnuplot epslatex terminal needs graphicx.sty or graphics.sty.}%
    \renewcommand\includegraphics[2][]{}%
  }%
  \providecommand\rotatebox[2]{#2}%
  \@ifundefined{ifGPcolor}{%
    \newif\ifGPcolor
    \GPcolortrue
  }{}%
  \@ifundefined{ifGPblacktext}{%
    \newif\ifGPblacktext
    \GPblacktexttrue
  }{}%
  \let\gplgaddtomacro\g@addto@macro
  \gdef\gplbacktext{}%
  \gdef\gplfronttext{}%
  \makeatother
  \ifGPblacktext
    \def\colorrgb#1{}%
    \def\colorgray#1{}%
  \else
    \ifGPcolor
      \def\colorrgb#1{\color[rgb]{#1}}%
      \def\colorgray#1{\color[gray]{#1}}%
      \expandafter\def\csname LTw\endcsname{\color{white}}%
      \expandafter\def\csname LTb\endcsname{\color{black}}%
      \expandafter\def\csname LTa\endcsname{\color{black}}%
      \expandafter\def\csname LT0\endcsname{\color[rgb]{1,0,0}}%
      \expandafter\def\csname LT1\endcsname{\color[rgb]{0,1,0}}%
      \expandafter\def\csname LT2\endcsname{\color[rgb]{0,0,1}}%
      \expandafter\def\csname LT3\endcsname{\color[rgb]{1,0,1}}%
      \expandafter\def\csname LT4\endcsname{\color[rgb]{0,1,1}}%
      \expandafter\def\csname LT5\endcsname{\color[rgb]{1,1,0}}%
      \expandafter\def\csname LT6\endcsname{\color[rgb]{0,0,0}}%
      \expandafter\def\csname LT7\endcsname{\color[rgb]{1,0.3,0}}%
      \expandafter\def\csname LT8\endcsname{\color[rgb]{0.5,0.5,0.5}}%
    \else
      \def\colorrgb#1{\color{black}}%
      \def\colorgray#1{\color[gray]{#1}}%
      \expandafter\def\csname LTw\endcsname{\color{white}}%
      \expandafter\def\csname LTb\endcsname{\color{black}}%
      \expandafter\def\csname LTa\endcsname{\color{black}}%
      \expandafter\def\csname LT0\endcsname{\color{black}}%
      \expandafter\def\csname LT1\endcsname{\color{black}}%
      \expandafter\def\csname LT2\endcsname{\color{black}}%
      \expandafter\def\csname LT3\endcsname{\color{black}}%
      \expandafter\def\csname LT4\endcsname{\color{black}}%
      \expandafter\def\csname LT5\endcsname{\color{black}}%
      \expandafter\def\csname LT6\endcsname{\color{black}}%
      \expandafter\def\csname LT7\endcsname{\color{black}}%
      \expandafter\def\csname LT8\endcsname{\color{black}}%
    \fi
  \fi
    \setlength{\unitlength}{0.0500bp}%
    \ifx\gptboxheight\undefined%
      \newlength{\gptboxheight}%
      \newlength{\gptboxwidth}%
      \newsavebox{\gptboxtext}%
    \fi%
    \setlength{\fboxrule}{0.5pt}%
    \setlength{\fboxsep}{1pt}%
    \definecolor{tbcol}{rgb}{1,1,1}%
\begin{picture}(4460.00,2520.00)%
    \gplgaddtomacro\gplbacktext{%
      \colorrgb{0.27,0.27,0.27}
      \put(509,694){\makebox(0,0)[r]{\strut{}\scriptsize $10^{0}$}}%
      \colorrgb{0.27,0.27,0.27}
      \put(509,1494){\makebox(0,0)[r]{\strut{}\scriptsize $10^{1}$}}%
      \colorrgb{0.27,0.27,0.27}
      \put(509,2295){\makebox(0,0)[r]{\strut{}\scriptsize $10^{2}$}}%
      \colorrgb{0.27,0.27,0.27}
      \put(1565,228){\makebox(0,0){\strut{}\scriptsize $10^{2}$}}%
      \colorrgb{0.27,0.27,0.27}
      \put(3380,228){\makebox(0,0){\strut{}\scriptsize $10^{3}$}}%
      \csname LTb\endcsname
      \put(889,2237){\makebox(0,0)[l]{\strut{}\footnotesize cube}}%
      \csname LTb\endcsname
      \put(1782,2237){\makebox(0,0)[l]{\strut{}\footnotesize $l_1$-costs}}%
      \csname LTb\endcsname
      \put(4024,2295){\makebox(0,0)[l]{\strut{}\color{DarkGray}\scriptsize $d_1$}}%
      \csname LTb\endcsname
      \put(889,2007){\makebox(0,0)[l]{\strut{}\color{DarkGray}\scriptsize $d_2 = 10$}}%
      \csname LTb\endcsname
      \put(4024,620){\makebox(0,0)[l]{\strut{}\color{DarkGray}\tiny 1}}%
      \csname LTb\endcsname
      \put(4024,844){\makebox(0,0)[l]{\strut{}\color{DarkGray}\tiny 2}}%
      \csname LTb\endcsname
      \put(4024,1199){\makebox(0,0)[l]{\strut{}\color{DarkGray}\tiny 3}}%
      \csname LTb\endcsname
      \put(4024,1456){\makebox(0,0)[l]{\strut{}\color{DarkGray}\tiny 4}}%
      \csname LTb\endcsname
      \put(4024,1637){\makebox(0,0)[l]{\strut{}\color{DarkGray}\tiny 5}}%
      \csname LTb\endcsname
      \put(4024,1774){\makebox(0,0)[l]{\strut{}\color{DarkGray}\tiny 6}}%
      \csname LTb\endcsname
      \put(4024,1882){\makebox(0,0)[l]{\strut{}\color{DarkGray}\tiny 7}}%
      \csname LTb\endcsname
      \put(4024,2044){\makebox(0,0)[l]{\strut{}\color{DarkGray}\tiny 8,9,10}}%
    }%
    \gplgaddtomacro\gplfronttext{%
    }%
    \gplbacktext
    \put(0,0){\includegraphics[width={223.00bp},height={126.00bp}]{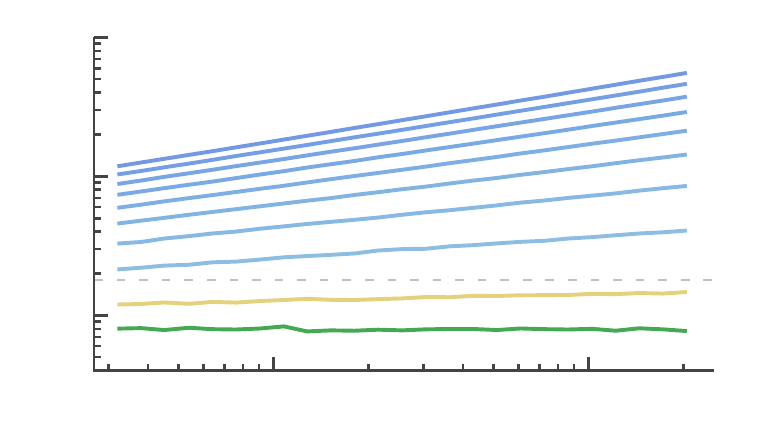}}%
    \gplfronttext
  \end{picture}%
\endgroup

%% file: figures/tex_smooth-nonsmooth-spheres-l2.tex
\begingroup
  \makeatletter
  \providecommand\color[2][]{%
    \GenericError{(gnuplot) \space\space\space\@spaces}{%
      Package color not loaded in conjunction with
      terminal option `colourtext'%
    }{See the gnuplot documentation for explanation.%
    }{Either use 'blacktext' in gnuplot or load the package
      color.sty in LaTeX.}%
    \renewcommand\color[2][]{}%
  }%
  \providecommand\includegraphics[2][]{%
    \GenericError{(gnuplot) \space\space\space\@spaces}{%
      Package graphicx or graphics not loaded%
    }{See the gnuplot documentation for explanation.%
    }{The gnuplot epslatex terminal needs graphicx.sty or graphics.sty.}%
    \renewcommand\includegraphics[2][]{}%
  }%
  \providecommand\rotatebox[2]{#2}%
  \@ifundefined{ifGPcolor}{%
    \newif\ifGPcolor
    \GPcolortrue
  }{}%
  \@ifundefined{ifGPblacktext}{%
    \newif\ifGPblacktext
    \GPblacktexttrue
  }{}%
  \let\gplgaddtomacro\g@addto@macro
  \gdef\gplbacktext{}%
  \gdef\gplfronttext{}%
  \makeatother
  \ifGPblacktext
    \def\colorrgb#1{}%
    \def\colorgray#1{}%
  \else
    \ifGPcolor
      \def\colorrgb#1{\color[rgb]{#1}}%
      \def\colorgray#1{\color[gray]{#1}}%
      \expandafter\def\csname LTw\endcsname{\color{white}}%
      \expandafter\def\csname LTb\endcsname{\color{black}}%
      \expandafter\def\csname LTa\endcsname{\color{black}}%
      \expandafter\def\csname LT0\endcsname{\color[rgb]{1,0,0}}%
      \expandafter\def\csname LT1\endcsname{\color[rgb]{0,1,0}}%
      \expandafter\def\csname LT2\endcsname{\color[rgb]{0,0,1}}%
      \expandafter\def\csname LT3\endcsname{\color[rgb]{1,0,1}}%
      \expandafter\def\csname LT4\endcsname{\color[rgb]{0,1,1}}%
      \expandafter\def\csname LT5\endcsname{\color[rgb]{1,1,0}}%
      \expandafter\def\csname LT6\endcsname{\color[rgb]{0,0,0}}%
      \expandafter\def\csname LT7\endcsname{\color[rgb]{1,0.3,0}}%
      \expandafter\def\csname LT8\endcsname{\color[rgb]{0.5,0.5,0.5}}%
    \else
      \def\colorrgb#1{\color{black}}%
      \def\colorgray#1{\color[gray]{#1}}%
      \expandafter\def\csname LTw\endcsname{\color{white}}%
      \expandafter\def\csname LTb\endcsname{\color{black}}%
      \expandafter\def\csname LTa\endcsname{\color{black}}%
      \expandafter\def\csname LT0\endcsname{\color{black}}%
      \expandafter\def\csname LT1\endcsname{\color{black}}%
      \expandafter\def\csname LT2\endcsname{\color{black}}%
      \expandafter\def\csname LT3\endcsname{\color{black}}%
      \expandafter\def\csname LT4\endcsname{\color{black}}%
      \expandafter\def\csname LT5\endcsname{\color{black}}%
      \expandafter\def\csname LT6\endcsname{\color{black}}%
      \expandafter\def\csname LT7\endcsname{\color{black}}%
      \expandafter\def\csname LT8\endcsname{\color{black}}%
    \fi
  \fi
    \setlength{\unitlength}{0.0500bp}%
    \ifx\gptboxheight\undefined%
      \newlength{\gptboxheight}%
      \newlength{\gptboxwidth}%
      \newsavebox{\gptboxtext}%
    \fi%
    \setlength{\fboxrule}{0.5pt}%
    \setlength{\fboxsep}{1pt}%
    \definecolor{tbcol}{rgb}{1,1,1}%
\begin{picture}(4460.00,2520.00)%
    \gplgaddtomacro\gplbacktext{%
      \colorrgb{0.27,0.27,0.27}
      \put(509,1015){\makebox(0,0)[r]{\strut{}\scriptsize $10^{0}$}}%
      \colorrgb{0.27,0.27,0.27}
      \put(509,1931){\makebox(0,0)[r]{\strut{}\scriptsize $10^{1}$}}%
      \colorrgb{0.27,0.27,0.27}
      \put(1565,228){\makebox(0,0){\strut{}\scriptsize $10^{2}$}}%
      \colorrgb{0.27,0.27,0.27}
      \put(3380,228){\makebox(0,0){\strut{}\scriptsize $10^{3}$}}%
      \csname LTb\endcsname
      \put(889,2237){\makebox(0,0)[l]{\strut{}\footnotesize sphere}}%
      \csname LTb\endcsname
      \put(1782,2237){\makebox(0,0)[l]{\strut{}\footnotesize $l_2^2$-costs}}%
      \csname LTb\endcsname
      \put(4024,2295){\makebox(0,0)[l]{\strut{}\color{DarkGray}\scriptsize $d_1$}}%
      \csname LTb\endcsname
      \put(889,2007){\makebox(0,0)[l]{\strut{}\color{DarkGray}\scriptsize $d_2 = 10$}}%
      \csname LTb\endcsname
      \put(4024,584){\makebox(0,0)[l]{\strut{}\color{DarkGray}\tiny 1,2}}%
      \csname LTb\endcsname
      \put(4024,918){\makebox(0,0)[l]{\strut{}\color{DarkGray}\tiny 3}}%
      \csname LTb\endcsname
      \put(4024,1291){\makebox(0,0)[l]{\strut{}\color{DarkGray}\tiny 4}}%
      \csname LTb\endcsname
      \put(4024,1546){\makebox(0,0)[l]{\strut{}\color{DarkGray}\tiny 5}}%
      \csname LTb\endcsname
      \put(4024,1731){\makebox(0,0)[l]{\strut{}\color{DarkGray}\tiny 6}}%
      \csname LTb\endcsname
      \put(4024,1870){\makebox(0,0)[l]{\strut{}\color{DarkGray}\tiny 7}}%
      \csname LTb\endcsname
      \put(4024,2069){\makebox(0,0)[l]{\strut{}\color{DarkGray}\tiny 8,9,10}}%
    }%
    \gplgaddtomacro\gplfronttext{%
      \csname LTb\endcsname
      \put(102,1335){\rotatebox{-270}{\makebox(0,0){\strut{}$\sqrt{n}\,\Delta_n$}}}%
      \csname LTb\endcsname
      \put(2317,1){\makebox(0,0){\strut{}$n$}}%
    }%
    \gplbacktext
    \put(0,0){\includegraphics[width={223.00bp},height={126.00bp}]{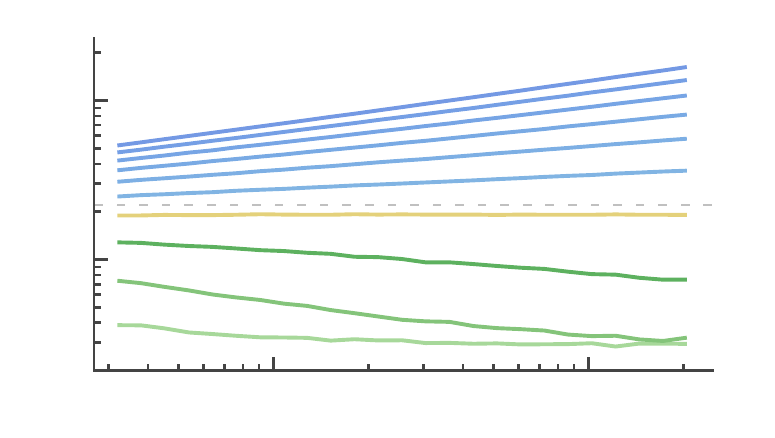}}%
    \gplfronttext
  \end{picture}%
\endgroup

%% file: figures/tex_smooth-nonsmooth-spheres-l1.tex
\begingroup
  \makeatletter
  \providecommand\color[2][]{%
    \GenericError{(gnuplot) \space\space\space\@spaces}{%
      Package color not loaded in conjunction with
      terminal option `colourtext'%
    }{See the gnuplot documentation for explanation.%
    }{Either use 'blacktext' in gnuplot or load the package
      color.sty in LaTeX.}%
    \renewcommand\color[2][]{}%
  }%
  \providecommand\includegraphics[2][]{%
    \GenericError{(gnuplot) \space\space\space\@spaces}{%
      Package graphicx or graphics not loaded%
    }{See the gnuplot documentation for explanation.%
    }{The gnuplot epslatex terminal needs graphicx.sty or graphics.sty.}%
    \renewcommand\includegraphics[2][]{}%
  }%
  \providecommand\rotatebox[2]{#2}%
  \@ifundefined{ifGPcolor}{%
    \newif\ifGPcolor
    \GPcolortrue
  }{}%
  \@ifundefined{ifGPblacktext}{%
    \newif\ifGPblacktext
    \GPblacktexttrue
  }{}%
  \let\gplgaddtomacro\g@addto@macro
  \gdef\gplbacktext{}%
  \gdef\gplfronttext{}%
  \makeatother
  \ifGPblacktext
    \def\colorrgb#1{}%
    \def\colorgray#1{}%
  \else
    \ifGPcolor
      \def\colorrgb#1{\color[rgb]{#1}}%
      \def\colorgray#1{\color[gray]{#1}}%
      \expandafter\def\csname LTw\endcsname{\color{white}}%
      \expandafter\def\csname LTb\endcsname{\color{black}}%
      \expandafter\def\csname LTa\endcsname{\color{black}}%
      \expandafter\def\csname LT0\endcsname{\color[rgb]{1,0,0}}%
      \expandafter\def\csname LT1\endcsname{\color[rgb]{0,1,0}}%
      \expandafter\def\csname LT2\endcsname{\color[rgb]{0,0,1}}%
      \expandafter\def\csname LT3\endcsname{\color[rgb]{1,0,1}}%
      \expandafter\def\csname LT4\endcsname{\color[rgb]{0,1,1}}%
      \expandafter\def\csname LT5\endcsname{\color[rgb]{1,1,0}}%
      \expandafter\def\csname LT6\endcsname{\color[rgb]{0,0,0}}%
      \expandafter\def\csname LT7\endcsname{\color[rgb]{1,0.3,0}}%
      \expandafter\def\csname LT8\endcsname{\color[rgb]{0.5,0.5,0.5}}%
    \else
      \def\colorrgb#1{\color{black}}%
      \def\colorgray#1{\color[gray]{#1}}%
      \expandafter\def\csname LTw\endcsname{\color{white}}%
      \expandafter\def\csname LTb\endcsname{\color{black}}%
      \expandafter\def\csname LTa\endcsname{\color{black}}%
      \expandafter\def\csname LT0\endcsname{\color{black}}%
      \expandafter\def\csname LT1\endcsname{\color{black}}%
      \expandafter\def\csname LT2\endcsname{\color{black}}%
      \expandafter\def\csname LT3\endcsname{\color{black}}%
      \expandafter\def\csname LT4\endcsname{\color{black}}%
      \expandafter\def\csname LT5\endcsname{\color{black}}%
      \expandafter\def\csname LT6\endcsname{\color{black}}%
      \expandafter\def\csname LT7\endcsname{\color{black}}%
      \expandafter\def\csname LT8\endcsname{\color{black}}%
    \fi
  \fi
    \setlength{\unitlength}{0.0500bp}%
    \ifx\gptboxheight\undefined%
      \newlength{\gptboxheight}%
      \newlength{\gptboxwidth}%
      \newsavebox{\gptboxtext}%
    \fi%
    \setlength{\fboxrule}{0.5pt}%
    \setlength{\fboxsep}{1pt}%
    \definecolor{tbcol}{rgb}{1,1,1}%
\begin{picture}(4460.00,2520.00)%
    \gplgaddtomacro\gplbacktext{%
      \colorrgb{0.27,0.27,0.27}
      \put(509,872){\makebox(0,0)[r]{\strut{}\scriptsize $10^{0}$}}%
      \colorrgb{0.27,0.27,0.27}
      \put(509,1584){\makebox(0,0)[r]{\strut{}\scriptsize $10^{1}$}}%
      \colorrgb{0.27,0.27,0.27}
      \put(509,2295){\makebox(0,0)[r]{\strut{}\scriptsize $10^{2}$}}%
      \colorrgb{0.27,0.27,0.27}
      \put(1565,228){\makebox(0,0){\strut{}\scriptsize $10^{2}$}}%
      \colorrgb{0.27,0.27,0.27}
      \put(3380,228){\makebox(0,0){\strut{}\scriptsize $10^{3}$}}%
      \csname LTb\endcsname
      \put(889,2237){\makebox(0,0)[l]{\strut{}\footnotesize sphere}}%
      \csname LTb\endcsname
      \put(1782,2237){\makebox(0,0)[l]{\strut{}\footnotesize $l_1$-costs}}%
      \csname LTb\endcsname
      \put(4024,2295){\makebox(0,0)[l]{\strut{}\color{DarkGray}\scriptsize $d_1$}}%
      \csname LTb\endcsname
      \put(889,2007){\makebox(0,0)[l]{\strut{}\color{DarkGray}\scriptsize $d_2 = 10$}}%
      \csname LTb\endcsname
      \put(4024,606){\makebox(0,0)[l]{\strut{}\color{DarkGray}\tiny 1,2}}%
      \csname LTb\endcsname
      \put(4024,720){\makebox(0,0)[l]{\strut{}\color{DarkGray}\tiny 3}}%
      \csname LTb\endcsname
      \put(4024,1254){\makebox(0,0)[l]{\strut{}\color{DarkGray}\tiny 4}}%
      \csname LTb\endcsname
      \put(4024,1597){\makebox(0,0)[l]{\strut{}\color{DarkGray}\tiny 5}}%
      \csname LTb\endcsname
      \put(4024,1799){\makebox(0,0)[l]{\strut{}\color{DarkGray}\tiny 6}}%
      \csname LTb\endcsname
      \put(4024,1940){\makebox(0,0)[l]{\strut{}\color{DarkGray}\tiny 7}}%
      \csname LTb\endcsname
      \put(4024,2131){\makebox(0,0)[l]{\strut{}\color{DarkGray}\tiny 8,9,10}}%
    }%
    \gplgaddtomacro\gplfronttext{%
      \csname LTb\endcsname
      \put(2317,1){\makebox(0,0){\strut{}$n$}}%
    }%
    \gplbacktext
    \put(0,0){\includegraphics[width={223.00bp},height={126.00bp}]{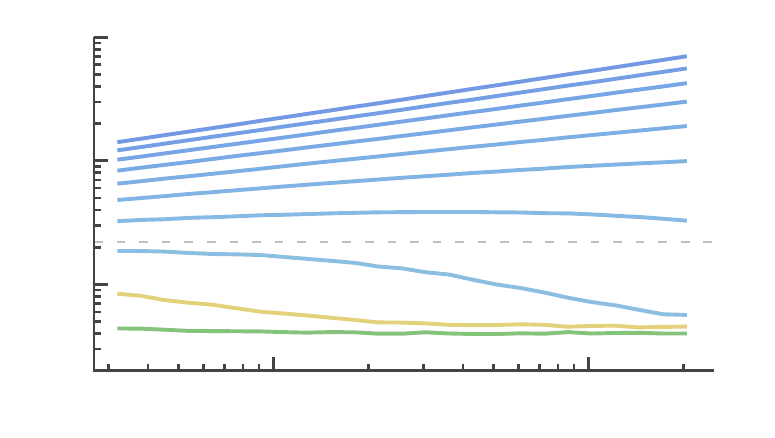}}%
    \gplfronttext
  \end{picture}%
\endgroup

%% file: figures/tex_second-dimension-cubes-l2.tex
\begingroup
  \makeatletter
  \providecommand\color[2][]{%
    \GenericError{(gnuplot) \space\space\space\@spaces}{%
      Package color not loaded in conjunction with
      terminal option `colourtext'%
    }{See the gnuplot documentation for explanation.%
    }{Either use 'blacktext' in gnuplot or load the package
      color.sty in LaTeX.}%
    \renewcommand\color[2][]{}%
  }%
  \providecommand\includegraphics[2][]{%
    \GenericError{(gnuplot) \space\space\space\@spaces}{%
      Package graphicx or graphics not loaded%
    }{See the gnuplot documentation for explanation.%
    }{The gnuplot epslatex terminal needs graphicx.sty or graphics.sty.}%
    \renewcommand\includegraphics[2][]{}%
  }%
  \providecommand\rotatebox[2]{#2}%
  \@ifundefined{ifGPcolor}{%
    \newif\ifGPcolor
    \GPcolortrue
  }{}%
  \@ifundefined{ifGPblacktext}{%
    \newif\ifGPblacktext
    \GPblacktexttrue
  }{}%
  \let\gplgaddtomacro\g@addto@macro
  \gdef\gplbacktext{}%
  \gdef\gplfronttext{}%
  \makeatother
  \ifGPblacktext
    \def\colorrgb#1{}%
    \def\colorgray#1{}%
  \else
    \ifGPcolor
      \def\colorrgb#1{\color[rgb]{#1}}%
      \def\colorgray#1{\color[gray]{#1}}%
      \expandafter\def\csname LTw\endcsname{\color{white}}%
      \expandafter\def\csname LTb\endcsname{\color{black}}%
      \expandafter\def\csname LTa\endcsname{\color{black}}%
      \expandafter\def\csname LT0\endcsname{\color[rgb]{1,0,0}}%
      \expandafter\def\csname LT1\endcsname{\color[rgb]{0,1,0}}%
      \expandafter\def\csname LT2\endcsname{\color[rgb]{0,0,1}}%
      \expandafter\def\csname LT3\endcsname{\color[rgb]{1,0,1}}%
      \expandafter\def\csname LT4\endcsname{\color[rgb]{0,1,1}}%
      \expandafter\def\csname LT5\endcsname{\color[rgb]{1,1,0}}%
      \expandafter\def\csname LT6\endcsname{\color[rgb]{0,0,0}}%
      \expandafter\def\csname LT7\endcsname{\color[rgb]{1,0.3,0}}%
      \expandafter\def\csname LT8\endcsname{\color[rgb]{0.5,0.5,0.5}}%
    \else
      \def\colorrgb#1{\color{black}}%
      \def\colorgray#1{\color[gray]{#1}}%
      \expandafter\def\csname LTw\endcsname{\color{white}}%
      \expandafter\def\csname LTb\endcsname{\color{black}}%
      \expandafter\def\csname LTa\endcsname{\color{black}}%
      \expandafter\def\csname LT0\endcsname{\color{black}}%
      \expandafter\def\csname LT1\endcsname{\color{black}}%
      \expandafter\def\csname LT2\endcsname{\color{black}}%
      \expandafter\def\csname LT3\endcsname{\color{black}}%
      \expandafter\def\csname LT4\endcsname{\color{black}}%
      \expandafter\def\csname LT5\endcsname{\color{black}}%
      \expandafter\def\csname LT6\endcsname{\color{black}}%
      \expandafter\def\csname LT7\endcsname{\color{black}}%
      \expandafter\def\csname LT8\endcsname{\color{black}}%
    \fi
  \fi
    \setlength{\unitlength}{0.0500bp}%
    \ifx\gptboxheight\undefined%
      \newlength{\gptboxheight}%
      \newlength{\gptboxwidth}%
      \newsavebox{\gptboxtext}%
    \fi%
    \setlength{\fboxrule}{0.5pt}%
    \setlength{\fboxsep}{1pt}%
    \definecolor{tbcol}{rgb}{1,1,1}%
\begin{picture}(4460.00,2520.00)%
    \gplgaddtomacro\gplbacktext{%
      \colorrgb{0.27,0.27,0.27}
      \put(553,375){\makebox(0,0)[r]{\strut{}\scriptsize $10^{0}$}}%
      \colorrgb{0.27,0.27,0.27}
      \put(553,1851){\makebox(0,0)[r]{\strut{}\scriptsize $10^{1}$}}%
      \colorrgb{0.27,0.27,0.27}
      \put(1597,228){\makebox(0,0){\strut{}\scriptsize $10^{2}$}}%
      \colorrgb{0.27,0.27,0.27}
      \put(3390,228){\makebox(0,0){\strut{}\scriptsize $10^{3}$}}%
      \csname LTb\endcsname
      \put(930,2237){\makebox(0,0)[l]{\strut{}\footnotesize cube}}%
      \csname LTb\endcsname
      \put(1811,2237){\makebox(0,0)[l]{\strut{}\footnotesize $l_2^2$-costs}}%
      \csname LTb\endcsname
      \put(4025,2295){\makebox(0,0)[l]{\strut{}\color{DarkGray}\scriptsize $d_1$}}%
      \csname LTb\endcsname
      \put(930,1277){\makebox(0,0)[l]{\strut{}\color{DarkGray}\scriptsize $d_2 = 100$}}%
      \csname LTb\endcsname
      \put(4025,950){\makebox(0,0)[l]{\strut{}\color{DarkGray}\tiny 1--4}}%
      \csname LTb\endcsname
      \put(4025,1055){\makebox(0,0)[l]{\strut{}\color{DarkGray}\tiny 5}}%
      \csname LTb\endcsname
      \put(4025,1195){\makebox(0,0)[l]{\strut{}\color{DarkGray}\tiny 6}}%
      \csname LTb\endcsname
      \put(4025,1393){\makebox(0,0)[l]{\strut{}\color{DarkGray}\tiny 7}}%
      \csname LTb\endcsname
      \put(930,1969){\makebox(0,0)[l]{\strut{}\color{DarkGray}\scriptsize $d_2 = 1000$}}%
      \csname LTb\endcsname
      \put(4025,1704){\makebox(0,0)[l]{\strut{}\color{DarkGray}\tiny 1--7}}%
      \csname LTb\endcsname
      \put(4025,1827){\makebox(0,0)[l]{\strut{}\color{DarkGray}\tiny 8}}%
      \csname LTb\endcsname
      \put(4025,1922){\makebox(0,0)[l]{\strut{}\color{DarkGray}\tiny 9}}%
      \csname LTb\endcsname
      \put(4025,2039){\makebox(0,0)[l]{\strut{}\color{DarkGray}\tiny 10}}%
    }%
    \gplgaddtomacro\gplfronttext{%
      \csname LTb\endcsname
      \put(212,1335){\rotatebox{-270}{\makebox(0,0){\strut{}$\sqrt{n}\,\Delta_n$}}}%
      \csname LTb\endcsname
      \put(2340,91){\makebox(0,0){\strut{}$n$}}%
    }%
    \gplbacktext
    \put(0,0){\includegraphics[width={223.00bp},height={126.00bp}]{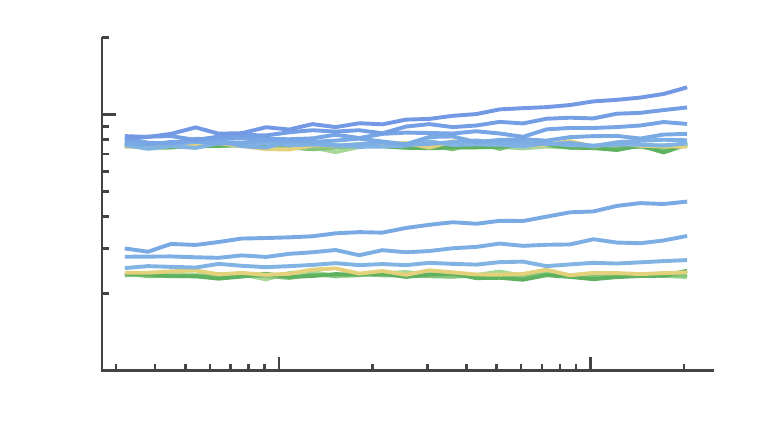}}%
    \gplfronttext
  \end{picture}%
\endgroup

%% file: figures/tex_rates-cubes-l2.tex
\begingroup
  \makeatletter
  \providecommand\color[2][]{%
    \GenericError{(gnuplot) \space\space\space\@spaces}{%
      Package color not loaded in conjunction with
      terminal option `colourtext'%
    }{See the gnuplot documentation for explanation.%
    }{Either use 'blacktext' in gnuplot or load the package
      color.sty in LaTeX.}%
    \renewcommand\color[2][]{}%
  }%
  \providecommand\includegraphics[2][]{%
    \GenericError{(gnuplot) \space\space\space\@spaces}{%
      Package graphicx or graphics not loaded%
    }{See the gnuplot documentation for explanation.%
    }{The gnuplot epslatex terminal needs graphicx.sty or graphics.sty.}%
    \renewcommand\includegraphics[2][]{}%
  }%
  \providecommand\rotatebox[2]{#2}%
  \@ifundefined{ifGPcolor}{%
    \newif\ifGPcolor
    \GPcolortrue
  }{}%
  \@ifundefined{ifGPblacktext}{%
    \newif\ifGPblacktext
    \GPblacktexttrue
  }{}%
  \let\gplgaddtomacro\g@addto@macro
  \gdef\gplbacktext{}%
  \gdef\gplfronttext{}%
  \makeatother
  \ifGPblacktext
    \def\colorrgb#1{}%
    \def\colorgray#1{}%
  \else
    \ifGPcolor
      \def\colorrgb#1{\color[rgb]{#1}}%
      \def\colorgray#1{\color[gray]{#1}}%
      \expandafter\def\csname LTw\endcsname{\color{white}}%
      \expandafter\def\csname LTb\endcsname{\color{black}}%
      \expandafter\def\csname LTa\endcsname{\color{black}}%
      \expandafter\def\csname LT0\endcsname{\color[rgb]{1,0,0}}%
      \expandafter\def\csname LT1\endcsname{\color[rgb]{0,1,0}}%
      \expandafter\def\csname LT2\endcsname{\color[rgb]{0,0,1}}%
      \expandafter\def\csname LT3\endcsname{\color[rgb]{1,0,1}}%
      \expandafter\def\csname LT4\endcsname{\color[rgb]{0,1,1}}%
      \expandafter\def\csname LT5\endcsname{\color[rgb]{1,1,0}}%
      \expandafter\def\csname LT6\endcsname{\color[rgb]{0,0,0}}%
      \expandafter\def\csname LT7\endcsname{\color[rgb]{1,0.3,0}}%
      \expandafter\def\csname LT8\endcsname{\color[rgb]{0.5,0.5,0.5}}%
    \else
      \def\colorrgb#1{\color{black}}%
      \def\colorgray#1{\color[gray]{#1}}%
      \expandafter\def\csname LTw\endcsname{\color{white}}%
      \expandafter\def\csname LTb\endcsname{\color{black}}%
      \expandafter\def\csname LTa\endcsname{\color{black}}%
      \expandafter\def\csname LT0\endcsname{\color{black}}%
      \expandafter\def\csname LT1\endcsname{\color{black}}%
      \expandafter\def\csname LT2\endcsname{\color{black}}%
      \expandafter\def\csname LT3\endcsname{\color{black}}%
      \expandafter\def\csname LT4\endcsname{\color{black}}%
      \expandafter\def\csname LT5\endcsname{\color{black}}%
      \expandafter\def\csname LT6\endcsname{\color{black}}%
      \expandafter\def\csname LT7\endcsname{\color{black}}%
      \expandafter\def\csname LT8\endcsname{\color{black}}%
    \fi
  \fi
    \setlength{\unitlength}{0.0500bp}%
    \ifx\gptboxheight\undefined%
      \newlength{\gptboxheight}%
      \newlength{\gptboxwidth}%
      \newsavebox{\gptboxtext}%
    \fi%
    \setlength{\fboxrule}{0.5pt}%
    \setlength{\fboxsep}{1pt}%
    \definecolor{tbcol}{rgb}{1,1,1}%
\begin{picture}(4460.00,2520.00)%
    \gplgaddtomacro\gplbacktext{%
      \colorrgb{0.27,0.27,0.27}
      \put(596,375){\makebox(0,0)[r]{\strut{}\scriptsize $10^{-2}$}}%
      \colorrgb{0.27,0.27,0.27}
      \put(596,1335){\makebox(0,0)[r]{\strut{}\scriptsize $10^{-1}$}}%
      \colorrgb{0.27,0.27,0.27}
      \put(596,2295){\makebox(0,0)[r]{\strut{}\scriptsize $10^{0}$}}%
      \colorrgb{0.27,0.27,0.27}
      \put(1628,228){\makebox(0,0){\strut{}\scriptsize $10^{2}$}}%
      \colorrgb{0.27,0.27,0.27}
      \put(3398,228){\makebox(0,0){\strut{}\scriptsize $10^{3}$}}%
      \csname LTb\endcsname
      \put(2014,2237){\makebox(0,0)[l]{\strut{}\footnotesize cube}}%
      \csname LTb\endcsname
      \put(2884,2237){\makebox(0,0)[l]{\strut{}\footnotesize $l_2^2$-costs}}%
      \csname LTb\endcsname
      \put(4026,2295){\makebox(0,0)[l]{\strut{}\color{DarkGray}\scriptsize $d_1$}}%
      \csname LTb\endcsname
      \put(830,1047){\makebox(0,0)[l]{\strut{}\color{DarkGray}\scriptsize $d_2=10$}}%
      \csname LTb\endcsname
      \put(4026,660){\makebox(0,0)[l]{\strut{}\color{DarkGray}\tiny 4}}%
      \csname LTb\endcsname
      \put(4026,908){\makebox(0,0)[l]{\strut{}\color{DarkGray}\tiny 5}}%
      \csname LTb\endcsname
      \put(4026,1160){\makebox(0,0)[l]{\strut{}\color{DarkGray}\tiny 6}}%
      \csname LTb\endcsname
      \put(4026,1352){\makebox(0,0)[l]{\strut{}\color{DarkGray}\tiny 7}}%
      \csname LTb\endcsname
      \put(4026,1511){\makebox(0,0)[l]{\strut{}\color{DarkGray}\tiny 8}}%
      \csname LTb\endcsname
      \put(4026,1641){\makebox(0,0)[l]{\strut{}\color{DarkGray}\tiny 9}}%
      \csname LTb\endcsname
      \put(4026,1750){\makebox(0,0)[l]{\strut{}\color{DarkGray}\tiny 10}}%
    }%
    \gplgaddtomacro\gplfronttext{%
      \csname LTb\endcsname
      \put(103,1335){\rotatebox{-270}{\makebox(0,0){\strut{}$\Delta_n$}}}%
      \csname LTb\endcsname
      \put(2362,82){\makebox(0,0){\strut{}$n$}}%
      \csname LTb\endcsname
      \put(1042,753){\makebox(0,0)[l]{\strut{}\raisebox{0.2ex}{\!\scriptsize $\propto n^{-2/d_1}$}}}%
    }%
    \gplbacktext
    \put(0,0){\includegraphics[width={223.00bp},height={126.00bp}]{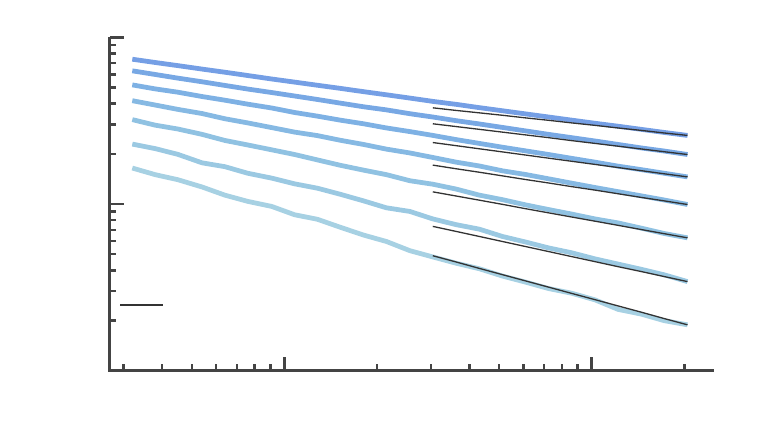}}%
    \gplfronttext
  \end{picture}%
\endgroup

%% file: figures/tex_semidiscrete-l2.tex
\begingroup
  \makeatletter
  \providecommand\color[2][]{%
    \GenericError{(gnuplot) \space\space\space\@spaces}{%
      Package color not loaded in conjunction with
      terminal option `colourtext'%
    }{See the gnuplot documentation for explanation.%
    }{Either use 'blacktext' in gnuplot or load the package
      color.sty in LaTeX.}%
    \renewcommand\color[2][]{}%
  }%
  \providecommand\includegraphics[2][]{%
    \GenericError{(gnuplot) \space\space\space\@spaces}{%
      Package graphicx or graphics not loaded%
    }{See the gnuplot documentation for explanation.%
    }{The gnuplot epslatex terminal needs graphicx.sty or graphics.sty.}%
    \renewcommand\includegraphics[2][]{}%
  }%
  \providecommand\rotatebox[2]{#2}%
  \@ifundefined{ifGPcolor}{%
    \newif\ifGPcolor
    \GPcolortrue
  }{}%
  \@ifundefined{ifGPblacktext}{%
    \newif\ifGPblacktext
    \GPblacktexttrue
  }{}%
  \let\gplgaddtomacro\g@addto@macro
  \gdef\gplbacktext{}%
  \gdef\gplfronttext{}%
  \makeatother
  \ifGPblacktext
    \def\colorrgb#1{}%
    \def\colorgray#1{}%
  \else
    \ifGPcolor
      \def\colorrgb#1{\color[rgb]{#1}}%
      \def\colorgray#1{\color[gray]{#1}}%
      \expandafter\def\csname LTw\endcsname{\color{white}}%
      \expandafter\def\csname LTb\endcsname{\color{black}}%
      \expandafter\def\csname LTa\endcsname{\color{black}}%
      \expandafter\def\csname LT0\endcsname{\color[rgb]{1,0,0}}%
      \expandafter\def\csname LT1\endcsname{\color[rgb]{0,1,0}}%
      \expandafter\def\csname LT2\endcsname{\color[rgb]{0,0,1}}%
      \expandafter\def\csname LT3\endcsname{\color[rgb]{1,0,1}}%
      \expandafter\def\csname LT4\endcsname{\color[rgb]{0,1,1}}%
      \expandafter\def\csname LT5\endcsname{\color[rgb]{1,1,0}}%
      \expandafter\def\csname LT6\endcsname{\color[rgb]{0,0,0}}%
      \expandafter\def\csname LT7\endcsname{\color[rgb]{1,0.3,0}}%
      \expandafter\def\csname LT8\endcsname{\color[rgb]{0.5,0.5,0.5}}%
    \else
      \def\colorrgb#1{\color{black}}%
      \def\colorgray#1{\color[gray]{#1}}%
      \expandafter\def\csname LTw\endcsname{\color{white}}%
      \expandafter\def\csname LTb\endcsname{\color{black}}%
      \expandafter\def\csname LTa\endcsname{\color{black}}%
      \expandafter\def\csname LT0\endcsname{\color{black}}%
      \expandafter\def\csname LT1\endcsname{\color{black}}%
      \expandafter\def\csname LT2\endcsname{\color{black}}%
      \expandafter\def\csname LT3\endcsname{\color{black}}%
      \expandafter\def\csname LT4\endcsname{\color{black}}%
      \expandafter\def\csname LT5\endcsname{\color{black}}%
      \expandafter\def\csname LT6\endcsname{\color{black}}%
      \expandafter\def\csname LT7\endcsname{\color{black}}%
      \expandafter\def\csname LT8\endcsname{\color{black}}%
    \fi
  \fi
    \setlength{\unitlength}{0.0500bp}%
    \ifx\gptboxheight\undefined%
      \newlength{\gptboxheight}%
      \newlength{\gptboxwidth}%
      \newsavebox{\gptboxtext}%
    \fi%
    \setlength{\fboxrule}{0.5pt}%
    \setlength{\fboxsep}{1pt}%
    \definecolor{tbcol}{rgb}{1,1,1}%
\begin{picture}(4460.00,2520.00)%
    \gplgaddtomacro\gplbacktext{%
      \colorrgb{0.27,0.27,0.27}
      \put(553,375){\makebox(0,0)[r]{\strut{}\scriptsize $10^{-1}$}}%
      \colorrgb{0.27,0.27,0.27}
      \put(553,1335){\makebox(0,0)[r]{\strut{}\scriptsize $10^{0}$}}%
      \colorrgb{0.27,0.27,0.27}
      \put(553,2295){\makebox(0,0)[r]{\strut{}\scriptsize $10^{1}$}}%
      \colorrgb{0.27,0.27,0.27}
      \put(1199,228){\makebox(0,0){\strut{}\scriptsize $10^{2}$}}%
      \colorrgb{0.27,0.27,0.27}
      \put(2294,228){\makebox(0,0){\strut{}\scriptsize $10^{3}$}}%
      \colorrgb{0.27,0.27,0.27}
      \put(3388,228){\makebox(0,0){\strut{}\scriptsize $10^{4}$}}%
      \csname LTb\endcsname
      \put(1282,2237){\makebox(0,0)[l]{\strut{}\footnotesize semidiscrete}}%
      \csname LTb\endcsname
      \put(2516,2237){\makebox(0,0)[l]{\strut{}\footnotesize $l_2^2$-costs}}%
      \csname LTb\endcsname
      \put(4025,2295){\makebox(0,0)[l]{\strut{}\color{DarkGray}\scriptsize $I$}}%
      \csname LTb\endcsname
      \put(3151,913){\makebox(0,0)[l]{\strut{}\color{DarkGray}\scriptsize $d_2 = 10$}}%
      \csname LTb\endcsname
      \put(4025,804){\makebox(0,0)[l]{\strut{}\color{DarkGray}\tiny 5}}%
      \csname LTb\endcsname
      \put(4025,707){\makebox(0,0)[l]{\strut{}\color{DarkGray}\tiny 10}}%
      \csname LTb\endcsname
      \put(4025,553){\makebox(0,0)[l]{\strut{}\color{DarkGray}\tiny 50}}%
      \csname LTb\endcsname
      \put(3151,1489){\makebox(0,0)[l]{\strut{}\color{DarkGray}\scriptsize $d_2 = 100$}}%
      \csname LTb\endcsname
      \put(4025,1415){\makebox(0,0)[l]{\strut{}\color{DarkGray}\tiny 5}}%
      \csname LTb\endcsname
      \put(4025,1317){\makebox(0,0)[l]{\strut{}\color{DarkGray}\tiny 10}}%
      \csname LTb\endcsname
      \put(4025,1200){\makebox(0,0)[l]{\strut{}\color{DarkGray}\tiny 50}}%
      \csname LTb\endcsname
      \put(3151,2007){\makebox(0,0)[l]{\strut{}\color{DarkGray}\scriptsize $d_2 = 1000$}}%
      \csname LTb\endcsname
      \put(4025,1922){\makebox(0,0)[l]{\strut{}\color{DarkGray}\tiny 5}}%
      \csname LTb\endcsname
      \put(4025,1817){\makebox(0,0)[l]{\strut{}\color{DarkGray}\tiny 10}}%
      \csname LTb\endcsname
      \put(4025,1715){\makebox(0,0)[l]{\strut{}\color{DarkGray}\tiny 50}}%
    }%
    \gplgaddtomacro\gplfronttext{%
      \csname LTb\endcsname
      \put(100,1335){\rotatebox{-270}{\makebox(0,0){\strut{}$\sqrt{n}\,\Delta_n$}}}%
      \csname LTb\endcsname
      \put(2340,1){\makebox(0,0){\strut{}$n$}}%
    }%
    \gplbacktext
    \put(0,0){\includegraphics[width={223.00bp},height={126.00bp}]{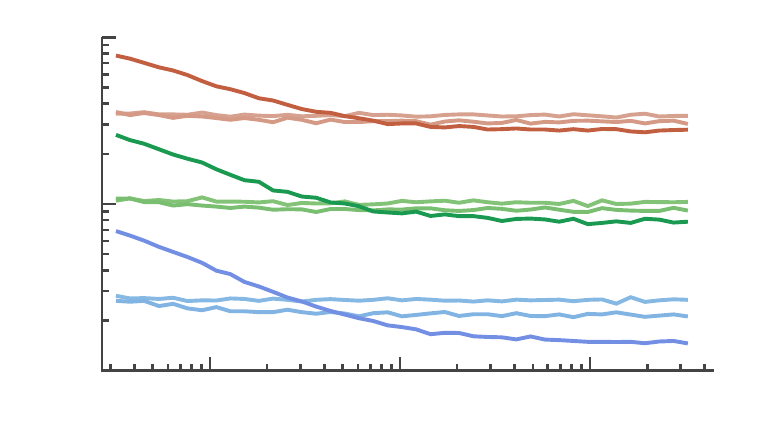}}%
    \gplfronttext
  \end{picture}%
\endgroup

%% file: OTunderLCA.bbl
\begin{thebibliography}{}

\bibitem[Ajtai et~al., 1984]{ajtai1984optimal}
Ajtai, M., Koml{\'o}s, J., \& Tusn{\'a}dy, G. (1984).
\newblock On optimal matchings.
\newblock {\em Combinatorica}, 4(4), 259--264.

\bibitem[Altschuler et~al., 2017]{altschuler2017near}
Altschuler, J., Niles-Weed, J., \& Rigollet, P. (2017).
\newblock Near-linear time approximation algorithms for optimal transport via
  {Sinkhorn} iteration.
\newblock In I. Guyon, U. von Luxburg, \&  others (Eds.), {\em Advances in
  Neural Information Processing Systems}, volume~30: Curran Associates, Inc.

\bibitem[Arjovsky et~al., 2017]{arjovsky2017wasserstein}
Arjovsky, M., Chintala, S., \& Bottou, L. (2017).
\newblock Wasserstein generative adversarial networks.
\newblock In {\em International Conference on Machine Learning}  \!\!, pages
  214--223.: Proceedings of Machine Learning Research.

\bibitem[Aurenhammer et~al., 1998]{aurenhammer1998minkowski}
Aurenhammer, F., Hoffmann, F., \& Aronov, B. (1998).
\newblock Minkowski-type theorems and least-squares clustering.
\newblock {\em Algorithmica}, 20(1), 61--76.

\bibitem[Bertsimas \& Tsitsiklis, 1997]{bertsimas1997introduction}
Bertsimas, D. \& Tsitsiklis, J. (1997).
\newblock {\em Introduction to linear optimization}.
\newblock Athena Scientific Series in Optimization and Neural computation.
  Athena Scientific.

\bibitem[Bickel \& Freedman, 1981]{bickel1981}
Bickel, P.~J. \& Freedman, D.~A. (1981).
\newblock Some asymptotic theory for the bootstrap.
\newblock {\em The Annals of Statistics}, 9(6), 1196--1217.

\bibitem[Bobkov \& Ledoux, 2021]{bobkov2019simple}
Bobkov, S. \& Ledoux, M. (2021).
\newblock A simple {Fourier} analytic proof of the {AKT} optimal matching
  theorem.
\newblock {\em The Annals of Applied Probability}, 31(6), 2567--2584.

\bibitem[Boissard \& Le~Gouic, 2014]{boissard2014}
Boissard, E. \& Le~Gouic, T. (2014).
\newblock On the mean speed of convergence of empirical and occupation measures
  in {Wasserstein} distance.
\newblock {\em Annales de l'Institut Henri Poincar\'e, Probabilit\'es et
  Statistiques}, 50(2), 539--563.

\bibitem[Bonneel et~al., 2011]{Bonneel2011}
Bonneel, N., van~de Panne, M., Paris, S., \& Heidrich, W. (2011).
\newblock {Displacement interpolation using Lagrangian mass transport}.
\newblock {\em ACM Transactions on Graphics (SIGGRAPH ASIA 2011)}, 30(6).

\bibitem[Bronshtein, 1976]{bronshtein1976varepsilon}
Bronshtein, E.~M. (1976).
\newblock $\varepsilon$-entropy of convex sets and functions.
\newblock {\em Siberian Mathematical Journal}, 17(3), 393--398.

\bibitem[Chernozhukov et~al., 2017]{chernozhukov2017monge}
Chernozhukov, V., Galichon, A., Hallin, M., \& Henry, M. (2017).
\newblock Monge--{Kantorovich} depth, quantiles, ranks and signs.
\newblock {\em The Annals of Statistics}, 45(1), 223--256.

\bibitem[Chizat et~al., 2020]{chizat2020}
Chizat, L., Roussillon, P., L\'{e}ger, F., Vialard, F.-X., \& Peyr\'{e}, G.
  (2020).
\newblock Faster {Wasserstein} distance estimation with the {Sinkhorn}
  divergence.
\newblock In H. Larochelle, M. Ranzato, \&  others (Eds.), {\em Advances in
  Neural Information Processing Systems}, volume~33  \!\!, pages 2257--2269.:
  Curran Associates, Inc.

\bibitem[Deb et~al., 2021]{deb2021rates}
Deb, N., Ghosal, P., \& Sen, B. (2021).
\newblock Rates of estimation of optimal transport maps using plug-in
  estimators via barycentric projections.
\newblock In M. Ranzato, A. Beygelzimer, \&  others (Eds.), {\em Advances in
  Neural Information Processing Systems}, volume~34: Curran Associates, Inc.

\bibitem[Deb \& Sen, 2021]{deb2021multivariate}
Deb, N. \& Sen, B. (2021).
\newblock Multivariate rank-based distribution-free nonparametric testing using
  measure transportation.
\newblock {\em preprint arXiv:1909.08733 (Accepted to Journal of the American
  Statistical Association)}.

\bibitem[del Barrio et~al., 1999]{del1999tests}
del Barrio, E., Cuesta-Albertos, J.~A., Matr{\'a}n, C., \&
  Rodr{\'\i}guez-Rodr{\'\i}guez, J.~M. (1999).
\newblock Tests of goodness of fit based on the {$L_2$}-{W}asserstein distance.
\newblock {\em The Annals of Statistics}, 27(4), 1230--1239.

\bibitem[del Barrio et~al., 2021]{delBarrio2021SemidiscreteCLT}
del Barrio, E., {Gonz{\'a}lez-Sanz}, A., \& {Loubes}, J.-M. (2021).
\newblock {A central limit theorem for semidiscrete Wasserstein distances}.
\newblock {\em preprint arXiv:2105.11721}.

\bibitem[Dereich et~al., 2013]{dereich2013constructive}
Dereich, S., Scheutzow, M., \& Schottstedt, R. (2013).
\newblock Constructive quantization: Approximation by empirical measures.
\newblock {\em Annales de l'Institut Henri Poincar\'e, Probabilit\'es et
  Statistiques}, 49(4), 1183--1203.

\bibitem[Dobri{\'c} \& Yukich, 1995]{dobric1995asymptotics}
Dobri{\'c}, V. \& Yukich, J.~E. (1995).
\newblock Asymptotics for transportation cost in high dimensions.
\newblock {\em Journal of Theoretical Probability}, 8(1), 97--118.

\bibitem[Dragomirescu \& Ivan, 1992]{dragomirescu1992smallest}
Dragomirescu, F. \& Ivan, C. (1992).
\newblock The smallest convex extensions of a convex function.
\newblock {\em Optimization}, 24(3-4), 193--206.

\bibitem[Dudley, 1969]{dudley1969}
Dudley, R.~M. (1969).
\newblock The speed of mean {G}livenko-{C}antelli convergence.
\newblock {\em The Annals of Mathematical Statistics}, 40(1), 40--50.

\bibitem[Dvurechensky et~al., 2018]{dvurechensky2018computational}
Dvurechensky, P., Gasnikov, A., \& Kroshnin, A. (2018).
\newblock Computational optimal transport: Complexity by accelerated gradient
  descent is better than by {S}inkhorn's algorithm.
\newblock In J. Dy \& A. Krause (Eds.), {\em Proceedings of the 35th
  International Conference on Machine Learning}, volume~80 of {\em Proceedings
  of Machine Learning Research}  \!\!, pages 1367--1376.: Proceedings of
  Machine Learning Research.

\bibitem[Evans \& Matsen, 2012]{evans2012phylogenetic}
Evans, S.~N. \& Matsen, F.~A. (2012).
\newblock The phylogenetic {Kantorovich}--{Rubinstein} metric for environmental
  sequence samples.
\newblock {\em Journal of the Royal Statistical Society: Series B (Statistical
  Methodology)}, 74(3), 569--592.

\bibitem[Fournier \& Guillin, 2015]{fournier2015rate}
Fournier, N. \& Guillin, A. (2015).
\newblock On the rate of convergence in {Wasserstein} distance of the empirical
  measure.
\newblock {\em Probability Theory and Related Fields}, 162(3), 707--738.

\bibitem[Gangbo \& McCann, 1996]{gangbo1996geometry}
Gangbo, W. \& McCann, R.~J. (1996).
\newblock The geometry of optimal transportation.
\newblock {\em Acta Mathematica}, 177(2), 113--161.

\bibitem[Gei{\ss} et~al., 2013]{geiss2013optimally}
Gei{\ss}, D., Klein, R., Penninger, R., \& Rote, G. (2013).
\newblock Optimally solving a transportation problem using {Voronoi} diagrams.
\newblock {\em Computational Geometry}, 46(8), 1009--1016.

\bibitem[Guntuboyina \& Sen, 2013]{guntuboyina2012covering}
Guntuboyina, A. \& Sen, B. (2013).
\newblock Covering numbers for convex functions.
\newblock {\em IEEE Transactions on Information Theory}, 59(4), 1957--1965.

\bibitem[Hallin et~al., 2021a]{hallin2021distribution}
Hallin, M., del Barrio, E., Cuesta-Albertos, J., \& Matr{\'a}n, C. (2021a).
\newblock Distribution and quantile functions, ranks and signs in dimension
  $d$: A measure transportation approach.
\newblock {\em The Annals of Statistics}, 49(2), 1139--1165.

\bibitem[Hallin et~al., 2021b]{hallin2021multivariate}
Hallin, M., Mordant, G., \& Segers, J. (2021b).
\newblock Multivariate goodness-of-fit tests based on wasserstein distance.
\newblock {\em Electronic Journal of Statistics}, 15(1), 1328--1371.

\bibitem[Hartmann \& Schuhmacher, 2020]{hartmann2020semi}
Hartmann, V. \& Schuhmacher, D. (2020).
\newblock Semi-discrete optimal transport: a solution procedure for the
  unsquared {Euclidean} distance case.
\newblock {\em Mathematical Methods of Operations Research}, 92, 133--163.

\bibitem[{Heinemann} et~al., 2020]{Heinemann2020}
{Heinemann}, F., {Munk}, A., \& {Zemel}, Y. (2020).
\newblock {Randomised Wasserstein barycenter computation: Resampling with
  statistical guarantees}.
\newblock {\em preprint arXiv:2012.06397}.

\bibitem[Kantorovich, 1942]{kant1942_original}
Kantorovich, L. (1942).
\newblock On the translocation of masses.
\newblock {\em Doklady Akademii Nauk URSS}, 37, 7--8.

\bibitem[Kantorovich, 1958]{kant58}
Kantorovich, L. (1958).
\newblock On the translocation of masses.
\newblock {\em Management Science}, 5(1), 1--4.

\bibitem[Kolmogorov \& Tikhomirov, 1961]{Kolmogorov1961}
Kolmogorov, A. \& Tikhomirov, V.~M. (1961).
\newblock $\epsilon$-entropy and $\epsilon$-capacity of sets in functional
  spaces.
\newblock In S. Cernikov, N. Cernikova, A. Kolmogorov, A. Mal'cev, \& B.
  Plotkin (Eds.), {\em Twelve Papers on Algebra and Real Functions}, American
  Mathematical Society Translations--series 2  \!\!, pages 277--364. American
  Mathematical Society.

\bibitem[Ledoux, 2019]{ledoux2019OptimalMatchingI}
Ledoux, M. (2019).
\newblock On optimal matching of {Gaussian} samples.
\newblock {\em Journal of Mathematical Sciences}, 238, 495--522.

\bibitem[Lee, 2013]{lee2013smooth}
Lee, J.~M. (2013).
\newblock {\em Introduction to smooth manifolds}, volume 218 of {\em Graduate
  Texts in Mathematics}.
\newblock Springer.

\bibitem[Liang, 2019]{liang2019minimax}
Liang, T. (2019).
\newblock On the minimax optimality of estimating the {Wasserstein} metric.
\newblock {\em preprint arXiv:1908.10324}.

\bibitem[Luenberger, 2003]{luenberger2003linear}
Luenberger, D. (2003).
\newblock {\em Linear and Nonlinear Programming: Second Edition}.
\newblock Springer US.

\bibitem[Mallows, 1972]{mallows1972note}
Mallows, C.~L. (1972).
\newblock A note on asymptotic joint normality.
\newblock {\em The Annals of Mathematical Statistics}, 43(2), 508--515.

\bibitem[{Manole} et~al., 2021]{Manole2021_Plugin}
{Manole}, T., {Balakrishnan}, S., {Niles-Weed}, J., \& {Wasserman}, L. (2021).
\newblock Plugin estimation of smooth optimal transport maps.
\newblock {\em preprint arXiv:2107.12364}.

\bibitem[{Manole} \& {Niles-Weed}, 2021]{Manole21}
{Manole}, T. \& {Niles-Weed}, J. (2021).
\newblock Sharp convergence rates for empirical optimal transport with smooth
  costs.
\newblock {\em preprint arXiv:2106.13181v2}.

\bibitem[Mattila, 1995]{mattila1995geometry}
Mattila, P. (1995).
\newblock {\em Geometry of Sets and Measures in Euclidean Spaces: Fractals and
  Rectifiability}.
\newblock Cambridge studies in advanced mathematics. Cambridge University
  Press.

\bibitem[McShane, 1934]{mcshane1934extension}
McShane, E.~J. (1934).
\newblock Extension of range of functions.
\newblock {\em Bulletin of the American Mathematical Society}, 40(12),
  837--842.

\bibitem[M{\'e}rigot, 2011]{merigot2011multiscale}
M{\'e}rigot, Q. (2011).
\newblock A multiscale approach to optimal transport.
\newblock In {\em Computer Graphics Forum}, volume~30  \!\!, pages 1583--1592.:
  Wiley Online Library.

\bibitem[Monge, 1781]{monge}
Monge, G. (1781).
\newblock M{\'e}moire sur la th{\'e}orie des d{\'e}blais et des remblais.
\newblock In {\em Histoire de l'Acad{\'e}mie Royale des Sciences de Paris}
  \!\!, pages 666--704.

\bibitem[Mordant \& Segers, 2022]{mordant2022measuring}
Mordant, G. \& Segers, J. (2022).
\newblock Measuring dependence between random vectors via optimal transport.
\newblock {\em Journal of Multivariate Analysis}, 189, 104912.

\bibitem[Munk \& Czado, 1998]{Munk98}
Munk, A. \& Czado, C. (1998).
\newblock Nonparametric validation of similar distributions and assessment of
  goodness of fit.
\newblock {\em Journal of the Royal Statistical Society: Series B (Statistical
  Methodology)}, 60(1), 223--241.

\bibitem[Nies et~al., 2021]{nies2021transport}
Nies, T.~G., Staudt, T., \& Munk, A. (2021).
\newblock Transport dependency: Optimal transport based dependency measures.
\newblock {\em preprint arXiv:2105.02073}.

\bibitem[Niles-Weed \& Rigollet, 2019]{niles2019estimation}
Niles-Weed, J. \& Rigollet, P. (2019).
\newblock Estimation of {Wasserstein} distances in the spiked transport model.
\newblock {\em preprint arXiv:1909.07513}.

\bibitem[Panaretos \& Zemel, 2019]{panaretos2019statistical}
Panaretos, V.~M. \& Zemel, Y. (2019).
\newblock Statistical aspects of {W}asserstein distances.
\newblock {\em Annual Review of Statistics and Its Application}, 6, 405--431.

\bibitem[Peyr{\'e} \& Cuturi, 2019]{cuturi18}
Peyr{\'e}, G. \& Cuturi, M. (2019).
\newblock Computational optimal transport: {With} applications to data science.
\newblock {\em Foundations and Trends in Machine Learning}, 11(5-6), 355--607.

\bibitem[Rachev \& R{\"u}schendorf, 1998a]{rachev1998massTheory}
Rachev, S. \& R{\"u}schendorf, L. (1998a).
\newblock {\em Mass transportation problems: Volume I: Theory}.
\newblock Probability and Its Applications. Springer.

\bibitem[Rachev \& R{\"u}schendorf, 1998b]{rachev1998massApplications}
Rachev, S. \& R{\"u}schendorf, L. (1998b).
\newblock {\em Mass transportation problems: Volume II: Applications}.
\newblock Probability and Its Applications. Springer.

\bibitem[Santambrogio, 2015]{santambrogio2015optimal}
Santambrogio, F. (2015).
\newblock {\em Optimal transport for applied mathematicians: Calculus of
  variations, PDEs, and modeling}.
\newblock Progress in Nonlinear Differential Equations and Their Applications.
  Springer.

\bibitem[Schiebinger et~al., 2019]{Schiebinger19}
Schiebinger, G., Shu, J., Tabaka, M., Cleary, B., Subramanian, V., Solomon, A.,
  Gould, J., Liu, S., Lin, S., Berube, P., Lee, L., Chen, J., Brumbaugh, J.,
  Rigollet, P., Hochedlinger, K., Jaenisch, R., Regev, A., \& Lander, E.~S.
  (2019).
\newblock Optimal-transport analysis of single-cell gene expression identifies
  developmental trajectories in reprogramming.
\newblock {\em Cell}, 176(4), 928 -- 943.e22.

\bibitem[Shorack \& Wellner, 1986]{shorack1986empirical}
Shorack, G. \& Wellner, J. (1986).
\newblock {\em Empirical processes with applications to statistics}.
\newblock Wiley Series in Probability and Mathematical Statistics: Probability
  and mathematical statistics. Wiley.

\bibitem[Singh \& P{\'o}czos, 2018]{singh2018minimax}
Singh, S. \& P{\'o}czos, B. (2018).
\newblock Minimax distribution estimation in {Wasserstein} distance.
\newblock {\em preprint arXiv:1802.08855}.

\bibitem[Sommerfeld \& Munk, 2018]{sommerfeld2018}
Sommerfeld, M. \& Munk, A. (2018).
\newblock Inference for empirical {Wasserstein} distances on finite spaces.
\newblock {\em Journal of the Royal Statistical Society: Series B (Statistical
  Methodology)}, 80(1), 219--238.

\bibitem[{Sommerfeld} et~al., 2019]{sommerfeld19FastProb}
{Sommerfeld}, M., {Schrieber}, J., {Zemel}, Y., \& {Munk}, A. (2019).
\newblock Optimal transport: Fast probabilistic approximation with exact
  solvers.
\newblock {\em Journal of Machine Learning Research}, 20, 1--23.

\bibitem[Sriperumbudur et~al., 2012]{sriperumbudur2012empirical}
Sriperumbudur, B.~K., Fukumizu, K., Gretton, A., Sch{\"o}lkopf, B., \&
  Lanckriet, G.~R. (2012).
\newblock On the empirical estimation of integral probability metrics.
\newblock {\em Electronic Journal of Statistics}, 6, 1550--1599.

\bibitem[Stein, 1971]{stein1971singular}
Stein, E.~M. (1971).
\newblock {\em Singular Integrals and Differentiability Properties of
  Functions}, volume~30 of {\em Princeton Mathematical Series}.
\newblock Princeton University Press.

\bibitem[Talagrand, 1994]{talagrand1994matching}
Talagrand, M. (1994).
\newblock Matching theorems and empirical discrepancy computations using
  majorizing measures.
\newblock {\em Journal of the American Mathematical Society}, 7(2), 455--537.

\bibitem[Talwalkar et~al., 2008]{talwalkar2008large}
Talwalkar, A., Kumar, S., \& Rowley, H. (2008).
\newblock Large-scale manifold learning.
\newblock In {\em 2008 IEEE Conference on Computer Vision and Pattern
  Recognition}  \!\!, pages 1--8.: IEEE.

\bibitem[Tameling et~al., 2021]{tameling2021Colocalization}
Tameling, C., Stoldt, S., Stephan, T., Naas, J., Jakobs, S., \& Munk, A.
  (2021).
\newblock Colocalization for super-resolution microscopy via optimal transport.
\newblock {\em Nature Computational Science}, 1(3), 199--211.

\bibitem[Vacher et~al., 2021]{Vacher21}
Vacher, A., Muzellec, B., Rudi, A., Bach, F., \& Vialard, F.-X. (2021).
\newblock A dimension-free computational upper-bound for smooth optimal
  transport estimation.
\newblock In M. Belkin \& S. Kpotufe (Eds.), {\em Proceedings of Thirty Fourth
  Conference on Learning Theory}, volume 134 of {\em Proceedings of Machine
  Learning Research}  \!\!, pages 4143--4173.: Proceedings of Machine Learning
  Research.

\bibitem[Villani, 2003]{vil03}
Villani, C. (2003).
\newblock {\em Topics in optimal transportation}, volume~58 of {\em Graduate
  Studies in Mathematics}.
\newblock American Mathematical Society.

\bibitem[Villani, 2008]{villani2008optimal}
Villani, C. (2008).
\newblock {\em Optimal transport: old and new}, volume 338 of {\em A Series of
  Comprehensive Studies in Mathematics}.
\newblock Springer.

\bibitem[von Luxburg \& Bousquet, 2004]{luxburg2004distance}
von Luxburg, U. \& Bousquet, O. (2004).
\newblock Distance-based classification with {Lipschitz} functions.
\newblock {\em Journal of Machine Learning Research}, 5(Jun), 669--695.

\bibitem[Wainwright, 2019]{wainwright2019high}
Wainwright, M.~J. (2019).
\newblock {\em High-dimensional statistics: A non-asymptotic viewpoint},
  volume~48 of {\em Cambridge Series in Statistical and Probabilistic
  Mathematics}.
\newblock Cambridge University Press.

\bibitem[Wang et~al., 2021]{wang2020optimal}
Wang, S., Cai, T.~T., \& Li, H. (2021).
\newblock Optimal estimation of {Wasserstein} distance on a tree with an
  application to microbiome studies.
\newblock {\em Journal of the American Statistical Association}, 116(535),
  1237--1253.

\bibitem[Weed \& Bach, 2019]{weed2019sharp}
Weed, J. \& Bach, F. (2019).
\newblock Sharp asymptotic and finite-sample rates of convergence of empirical
  measures in {W}asserstein distance.
\newblock {\em Bernoulli}, 25(4A), 2620--2648.

\bibitem[Weed \& Berthet, 2019]{weedBerthet19}
Weed, J. \& Berthet, Q. (2019).
\newblock Estimation of smooth densities in {Wasserstein} distance.
\newblock In A. Beygelzimer \& D. Hsu (Eds.), {\em Proceedings of the
  Thirty-Second Conference on Learning Theory}, volume~99 of {\em Proceedings
  of Machine Learning Research}  \!\!, pages 3118--3119.: Proceedings of
  Machine Learning Research.

\bibitem[Whitney, 1934]{whitney1934analytic}
Whitney, H. (1934).
\newblock Analytic extensions of differentiable functions defined in closed
  sets.
\newblock {\em Transactions of the American Mathematical Society}, 36(1),
  63--89.

\bibitem[Zhu et~al., 2018]{zhu2018image}
Zhu, B., Liu, J.~Z., Cauley, S.~F., Rosen, B.~R., \& Rosen, M.~S. (2018).
\newblock Image reconstruction by domain-transform manifold learning.
\newblock {\em Nature}, 555(7697), 487--492.

\end{thebibliography}
